\documentclass[12pt,titlepage]{article}
\usepackage{amsmath, amsthm,amssymb}
\usepackage{graphicx}
\usepackage{caption}
\usepackage{subcaption}
\usepackage{array}
\usepackage[mathscr]{eucal}
\usepackage{color}
\usepackage{xurl} % José: I added this package to allow breaks across lines in the url
\usepackage{enumitem} % José: I added this package for references
\usepackage{float}
\usepackage{comment}
\usepackage{mathtools}
\usepackage {tikz}
\usetikzlibrary{positioning}
\definecolor {processblue}{cmyk}{0.96,0,0,0}
\usetikzlibrary{calc}
\usetikzlibrary{arrows}
\usepackage{fancybox}
\usepackage{tcolorbox}
\usepackage{pdflscape}
\usepackage{comment}
\usepackage{ulem}
\usepackage{multicol}
\usepackage{bm}
\usepackage{comment}
\usepackage{dsfont}
\usepackage{hyperref}
\newtheorem{proposition}{Proposition}

\newtheorem{example}[proposition]{Example}

\newtheorem{assumption}[proposition]{Assumption}
\theoremstyle{Definition}
\newtheorem{remark}[proposition]{Remark}
\numberwithin{equation}{section}
\numberwithin{proposition}{section}
\numberwithin{table}{section}

\newcommand{\lcx}{\preceq_{\mathrm{CX}}}

\newcommand{\Esp}{\mathrm{E}}
\newcommand{\Prob}{\mathrm{P}}

\newcommand{\Var}{\mathrm{Var}}

\newcommand{\der}{\mathrm{d}}

\newcommand{\Avec}{\boldsymbol{A}}

\newcommand{\mvec}{\boldsymbol{m}}

\newcommand{\eqD}{\overset{\mathrm{d}}{=}}

\begin{document}

\title{LINEAR RISK SHARING IN COMMUNITY-BASED INSURANCE: RUIN REDUCTION IN THE COMPOUND POISSON MODEL}
\author{{\Large M}ICHEL {\Large D}ENUIT \\
Institute of Statistics, Biostatistics and Actuarial Science (ISBA)\\
Louvain Institute of Data Analysis and Modeling (LIDAM)\\
UCLouvain\\
Louvain-la-Neuve, Belgium \\
\texttt{\small michel.denuit@uclouvain.be} 
\\[3mm] 
{\Large J}OS\'E {\Large M}IGUEL {\Large F}LORES-{\Large C}ONTR\'O \\
Institute of Statistics, Biostatistics and Actuarial Science (ISBA)\\
Louvain Institute of Data Analysis and Modeling (LIDAM)\\
UCLouvain\\
Louvain-la-Neuve, Belgium \\
\texttt{\small jose.flores@uclouvain.be} 
\\[3mm]
{\Large C}HRISTIAN {\Large Y}. {\Large R}OBERT\\
Laboratoire de Sciences Actuarielle et Financi\`ere (SAF)\\
Institut de science financi\`ere et d'assurances (ISFA)\\
Université Claude Bernard Lyon 1\\
Lyon, France\\
\texttt{\small christian.robert@univ-lyon1.fr} }
\maketitle

\begin{abstract}
This paper studies proportional risk sharing at claim occurrence time in community-based insurance. Each participant is modeled by an individual Cramér–Lundberg surplus process, and, whenever a claim is reported within the pool, its cost is redistributed according to a fixed allocation matrix. We compare the infinite-time ruin probability of each participant under stand-alone operation and under pool participation. Our main result shows that pooling reduces, for every participant, the infinite-time ruin probability when claim severities belong to a common scale family, the allocation rule satisfies full allocation and actuarial fairness, and each transfer remains bounded by an individual capacity condition. The proof relies on a convex-order comparison between the losses borne inside the pool and the corresponding stand-alone losses. We also clarify the role of these assumptions by showing that, outside this framework, pooling need not be beneficial for all participants. Numerical illustrations with Exponential and LogNormal severities support the theoretical findings and highlight how the design of proportional sharing rules affects solvency. The paper thus provides simple and interpretable sufficient conditions under which transparent linear risk-sharing arrangements improve individual solvency in community-based insurance.
 
\vfil\noindent 
\vfil
\end{abstract}

%\tableofcontents

\section{Introduction and motivation}
\label{sec:Introduction and motivation}

Human societies have always operated under conditions of uncertainty, where unpredictable shocks such as illness, fire or death
pose recurrent threats. These vulnerabilities are particularly acute in low-income communities where market imperfections, limited enforcement capacity, weak infrastructure, and incomplete credit and insurance markets restrict access to formal mechanisms for managing risk (Dercon, 2002; Fitzsimons et al., 2018; Saha and Qin, 2023). In such environments, households frequently rely on informal, socially embedded arrangements to cope with adverse events.

Among the range of informal strategies used to manage risk, risk pooling occupies a distinctive position. Unlike precautionary behavior, risk avoidance, or self-insurance through the accumulation of private reserves, which can be undertaken individually, risk pooling is inherently collective, relying on social interaction and cooperation to distribute losses among participants. Through this mechanism, individuals share exposure to uncertainty by smoothing shocks over time and across the group, thereby transforming potentially large or volatile losses into more manageable contributions. Much like formal insurance, which involves regular premium payments in exchange for protection against adverse events, informal risk-pooling systems redistribute losses within networks of kin, friends, or community members, thereby reducing the burden borne by any single individual (Cronk and Aktipis, 2021).

In a risk-sharing pool, participants are indemnified from a collective fund for their individual losses and, in return, make \textcolor{black}{ex-post} contributions so that the aggregate of individual payments exactly matches the total loss experienced by the pool (Denuit et al., 2022; Feng, 2023). The actuarial literature on such arrangements has investigated a variety of modeling assumptions, including the distribution of losses, the homogeneity of participants, and the exchangeability of risks, as well as alternative allocation rules governing the distribution of contributions and benefits. Among these rules, the conditional mean risk-sharing rule, originally introduced in Denuit and Dhaene (2012), has received particular attention; see, for instance, Denuit (2019), Denuit and Robert (2024), and the references therein. Under this rule, each participant’s contribution is given by the conditional expectation of their loss given the aggregate loss of the pool. Beyond this benchmark, several alternative allocation principles have also been proposed and studied.

\textcolor{black}{The studies discussed above primarily examine risk-sharing arrangements in a static setting, where interactions are confined to a single period. In practice, however, individuals often participate in long-term risk-sharing agreements that allow them to smooth adverse shocks over time. This naturally gives rise to a dynamic framework in which the temporal dimension plays a fundamental role.}

\textcolor{black}{This article adopts such a dynamic perspective on risk sharing. Recent contributions in this area include Abdikerimova et al. (2024), who show that, within the risk-sharing framework they consider, the optimal dynamic risk-sharing rule is horizon independent. They further establish a connection between their model and the classical portfolio selection problem in mathematical finance. Likewise, Boonen and Zhou (2026) formulate risk sharing as a continuous-time stochastic control problem in which participants enter social contracts to determine Pareto-efficient proportional sharing arrangements. Consistent with practical risk-sharing mechanisms, the authors impose non-negativity constraints on the risk-sharing proportions. Under this framework, they derive the Hamilton--Jacobi--Bellman equation that characterizes the optimal sharing rule, which they subsequently use to quantify the benefits of risk sharing and to investigate the impact of the imposed constraints.}

\textcolor{black}{Other recent studies have explored dynamic risk sharing in different settings. For example, Hanbali et al. (2025) examine time-varying risk-sharing arrangements in the context of life insurance, focusing on the interaction between annuity buyers and providers. Meanwhile, Tao et al. (2025) study a group of economic agents belonging to two distinct types and derive the optimal investment, consumption, and risk-sharing strategies using martingale methods.}

While the actuarial literature has primarily focused on the mathematical properties of risk-sharing rules (see Denuit et al., 2022, for a broad overview), practical implementations of collective risk sharing arise in a variety of institutional settings. In many low-income environments, informal pooling mechanisms have evolved into more structured and institutionalized arrangements, such as community-based insurance schemes (CBIs). CBIs are locally organized, member-financed schemes that pool contributions in order to provide pre-specified benefits, \textcolor{black}{formalizing} rules on contributions, eligibility, and payouts while preserving the underlying logic of collective risk sharing. See Dercon et al. (2006) for an example of funeral CBIs in Ethiopia and Tanzania, and Lemay-Boucher (2012) for an example from Benin. CBIs exist in many forms, with community-based health insurance (CBHI) being especially prominent in the literature (Ekman, 2004; Carrin et al., 2005; Mebratie et al., 2013; Eze et al., 2023), alongside schemes covering funeral expenses, agricultural risks, and other insurable events.

In practice, CBIs display substantial heterogeneity in their contribution mechanisms, with many schemes relying either on regular flat-rate contributions paid periodically by members or on fixed ex-post contributions triggered by specific events. For instance, Dercon et al. (2006) document that, in Tanzanian funeral CBIs, participants contribute only when a funeral occurs, whereas in Ethiopian CBIs members typically pay regular monthly contributions. The contribution mechanism therefore defines the risk-sharing rule: it specifies the allocation borne by each member in exchange for participation in collective coverage. In many real-world settings, these designs prioritize simplicity and transparency, two features that are essential for building trust and sustaining participation when information asymmetries are severe and administrative capacity is limited (Dror et al., 2016; Fadlallah et al., 2018; Levantesi and Piscopo, 2022).
In this paper, we study whether participation in a CBI yields a measurable solvency benefit for individuals relative to stand-alone operation.

Following Denuit and Robert (2023), we model individual insurance accounts within a community by means of the classical Cramér–Lundberg risk process; see Asmussen and Albrecher (2010) and Mandjes and Boxma (2023) for standard references on this model and its variants. Pooling is introduced directly at claim occurrence times, so that each insured loss is immediately redistributed across participants according to a prescribed sharing rule. While the conditional mean risk-sharing rule studied by Denuit and Robert (2023) enjoys strong theoretical properties and is naturally motivated by conditional averaging, it may be computationally demanding for some probability models (Yang and Wei, 2025). More importantly for our present purpose, it may be less transparent for non-technical participants than simpler linear arrangements. In community-based settings, where legitimacy and participation often depend on the perceived clarity of the contribution rule, this interpretability issue is not merely cosmetic. Motivated by these considerations, we focus here on linear risk-sharing rules under which each participant bears a fixed fraction of each loss reported in the pool. We then compare the infinite-time ruin probability of each participant under stand-alone operation and under pool participation.

\textcolor{black}{It is important to note that, unlike most previous studies on risk sharing in the actuarial science literature, where individual losses are first aggregated into a common pool and only subsequently redistributed (see, for example, Chapter 6 in Feng (2023)), our model allocates individual losses among pool members immediately upon their occurrence. Moreover, existing models typically assume that risk sharing is activated only as a mechanism to prevent ruin, that is, support from peers is provided only when an insurer\rq s surplus is about to become (or becomes) negative. See, e.g., Albrecher et al. (2017) and Gu et al. (2018). By contrast, in the present framework, losses are shared instantaneously among pool members regardless of whether any participant is close to ruin, making risk sharing an integral and continuous feature of the surplus process rather than an intervention triggered only by the risk of insolvency.}

The main contribution of this paper is to provide tractable sufficient conditions under which joining a community-based insurance pool with a linear sharing rule reduces the infinite-time ruin probability of every participant. In contrast with the conditional mean rule, the proportional mechanism studied here is fully specified by a deterministic allocation matrix and is therefore easier to explain, implement, and monitor. Our analysis shows that this simplicity does not prevent rigorous solvency improvements: under appropriate assumptions, pooling makes each participant weakly better off in ruin-probability terms. The paper thus contributes to the actuarial theory of risk sharing by identifying a class of transparent allocation rules for which individual participation can be justified through a classical solvency criterion.

%The three key assumptions underlying our main result deserve further discussion. First, actuarial fairness ensures that, for each participant, the expected claim amount borne inside the pool coincides with the expected claim amount borne under stand-alone operation. This condition isolates the pure effect of risk reallocation from any systematic transfer in expectation and guarantees that the comparison is not driven by hidden subsidies across participants. Second, the capacity constraint requires that the expected share assumed by participant $i$ from any other participant’s claim should not exceed participant $i$’s own expected claim amount. Economically, this prevents the sharing rule from assigning to a participant an exposure that is disproportionate relative to their own risk profile; mathematically, it controls the extent of cross-participant transfers and is instrumental in establishing the relevant convex-order comparison. Third, the assumption that claim severities belong to a common scale family means that heterogeneity across participants operates through scale parameters only, while the normalized claim-size distributions coincide. This structure is restrictive, but it is precisely what allows us to compare pooled and stand-alone losses through dispersion arguments. It should therefore be viewed as a transparent sufficient condition for ruin reduction rather than as a universal description of all community-based insurance settings.

The remainder of the paper is organized as follows. Section \ref{sec:Loss model} introduces the compound Poisson loss model and the corresponding individual and pooled surplus processes. Section \ref{sec:Linear risk sharing at occurrence time} defines the proportional risk-sharing rule at occurrence time and establishes the main result showing that pooling reduces individual infinite-time ruin probabilities under some conditions. This section also provides numerical illustrations. Section \ref{sec:Discussions of the assumptions} discusses the scope and limitations of the assumptions. 
\textcolor{black}{Section \ref{Pareto-OptimalityforRuinProbabilities-Section5} discusses the design of allocation matrices through scalarized optimization problems and illustrates how Pareto-efficient candidates can be
obtained under the full-allocation, actuarial-fairness and capacity constraints.} Section \ref{sec:Conclusion} concludes.
\textcolor{black}{The appendices provide additional material: Appendix \ref{sec:Incorporating dependencies} discusses an extension incorporating dependence through common shocks, Appendix \ref{sec:Mathematical Proofs} details the mathematical proofs, Appendix \ref{sec:Examples illustrating the role of AssumptionSF} gives supplementary examples illustrating the role of the scale-family assumption, and Appendix \ref{sec:Supplementary example illustrating the role of AssumptionCC} reports an additional numerical example on the violation of the capacity constraint.}

\section{Loss model}
\label{sec:Loss model}

\subsection{Individual accounts}

Consider a community-based insurance pool with $n$ participants, indexed by $i=1,\ldots,n$, and starting operations at time $0$. Each participant is associated with an individual insurance account reflecting his or her own claims experience.

For participant $i$, let \textcolor{black}{$\{N_{i,t}, t\geq 0\}$} denote the number of claims reported over the \textcolor{black}{interval $(0,t]$}, with $N_{i,0}=0$. Throughout the paper, $\{N_{i,t},\, t\ge 0\}$ is assumed to be a Poisson process with intensity $\lambda_i>0$. The size of the $k$th claim reported by participant $i$ is denoted by $Y_{i,k}$. For each fixed $i$, the random variables $Y_{i,1},Y_{i,2},\ldots$ are assumed to be non-negative, independent and identically distributed, with common distribution function $F_{Y_i}$, mean $\textcolor{black}{b_{i}=\Esp[Y_{i,1}]}$, and finite variance $\sigma^2_{Y_{i}}=\Var[Y_{i,1}]$. Claim sizes are furthermore assumed to be independent of the counting process $\{N_{i,t},\, t\ge 0\}$.

The cumulative claim amount reported by participant $i$ up to time $t$ is
\begin{align}
    S_{i,t}=\sum_{k=1}^{N_{i,t}} Y_{i,k}, \qquad t\ge 0,
    \label{LinearRiskSharing-Section2-Equation1}
\end{align}
with $S_{i,0}=0$. Hence, for each $i$, $\{S_{i,t},\, t\ge 0\}$ is a compound Poisson process. We also assume that the processes $\{S_{i,t},\, t\ge 0\}$ are mutually independent for $i=1,\ldots,n$.

Participant $i$ contributes to the scheme at constant rate $c_i$ and provides an initial reserve $\kappa_i\ge 0$. The corresponding individual surplus process is
\begin{align}
    V_{i,t}=c_i t-S_{i,t}+\kappa_i, \qquad t\ge 0.
    \label{LinearRiskSharing-Section2-Equation2}
\end{align}
We take the contribution rate to be the expected claim amount per unit of time, multiplied by a common safety loading:
\begin{align}
    c_i=(1+\eta)\Esp[S_{i,1}]=(1+\eta)\lambda_i\textcolor{black}{b_{i}},
    \label{LinearRiskSharing-Section2-Equation3}
\end{align}
where $\eta>0$ is the safety loading coefficient.

\subsection{Pooled fund}

When participants join the pool, contributions, claim amounts, and initial reserves are aggregated as
\begin{align}
    c=\sum_{i=1}^n c_i, \qquad
    S_t=\sum_{i=1}^n S_{i,t}, \qquad
    \kappa=\sum_{i=1}^n \kappa_i.
    \label{LinearRiskSharing-Section2-Equation4}
\end{align}
The aggregate surplus of the pooled fund at time $t$ is then given by
\begin{align}
    V_t=\sum_{i=1}^n V_{i,t}=ct-S_t+\kappa, \qquad t\ge 0.
    \label{LinearRiskSharing-Section2-Equation5}
\end{align}
Both \eqref{LinearRiskSharing-Section2-Equation2} and \eqref{LinearRiskSharing-Section2-Equation5} are of classical Cramér--Lundberg type.

At the pool level, the total number of claims reported over \textcolor{black}{$(0,t]$} is $N_t=\sum_{i=1}^n N_{i,t}$.
Since the participant-specific claim arrival processes are independent Poisson processes, $\{N_t,\, t\ge 0\}$ is itself a Poisson process with intensity $\lambda_\bullet=\sum_{i=1}^n \lambda_i$. Let
\begin{align}
    T_k=\inf\{t\ge 0: N_t=k\},
    \label{LinearRiskSharing-Section2-Equation6}
\end{align}
denote the occurrence time of the $k$th claim at the pool level. The aggregate claim amount up to time $t$ can then be written as
\begin{align}
    S_t \overset{\textcolor{black}{\mathrm{d}}}{=} \sum_{k=1}^{N_t} X_k, \qquad t\ge 0,
    \label{LinearRiskSharing-Section2-Equation7}
\end{align}
where $X_1,X_2,\ldots$ are independent and identically distributed with common distribution function
$F_X=\frac{1}{\lambda_\bullet}\sum_{i=1}^n \lambda_i F_{Y_i}$.
Therefore, $\{S_t,\, t\ge 0\}$ is a compound Poisson process with intensity $\lambda_\bullet$ and claim-size distribution $F_X$.

More explicitly, the $k$th claim amount at the pool level can be represented as
\begin{align}
    X_k=\sum_{j=1}^n I_{j,k}\, Y_{j,N_{j,T_k}},
    \label{LinearRiskSharing-Section2-Equation8}
\end{align}
where, for each $k$, exactly one indicator among $I_{1,k},\ldots,I_{n,k}$ is equal to $1$, namely the one corresponding to the participant who generated the $k$th claim. In particular,
\[
\Prob[I_{j,k}=1]=\frac{\lambda_j}{\lambda_\bullet}, \hspace{2mm} j=1,\ldots,n,\text{ and }
\sum_{j=1}^n I_{j,k}=1.
\]
Hence, $X_k$ is a mixture of the individual claim severities, with mixing weights proportional to the claim frequencies. The summands in \eqref{LinearRiskSharing-Section2-Equation8} are mutually exclusive in the sense of Dhaene and Denuit (1999).

\section{Linear risk sharing at occurrence time}
\label{sec:Linear risk sharing at occurrence time}

\subsection{Proportional risk sharing rule}
\label{sec:Proportional risk sharing at occurrence time}

Section \ref{sec:Loss model} introduces the aggregate claim process at pool level through the sequence of
claim occurrence times $(T_k)_{k\ge 1}$ and the corresponding claim amounts
$(X_k)_{k\ge 1}$. We now specify how each aggregate claim $X_k$ is redistributed among
participants when it occurs.

For each ordered pair $(i,j)\in\{1,\ldots,n\}^2$, let $a_{i,j}\in[0,1]$ denote the fraction of a
claim generated by participant $j$ that is borne by participant $i$ within the pool. Following
Abdikerimova and Feng (2022), we refer to these coefficients as transfer ratios, and
we collect them in the allocation matrix $\Avec=(a_{i,j})_{1\le i,j\le n}$. In particular,
$a_{i,i}$ is the retention level of participant $i$, that is, the fraction of his or her own claim
that remains on the individual account after pooling.

A basic requirement is that each claim be fully redistributed across the members of the pool.
This is the full-allocation property, also referred to as budget balance; see
Denuit et al.\ (2022) and Charpentier and Ratz (2025). In matrix form, it requires
\begin{align}
    \sum_{i=1}^{n} a_{i,j}=1, \qquad j=1,\ldots,n,
    \label{PropertiesoftheLinearRisk-SharingRule-Section3-Equation2}
\end{align}
so that $\Avec$ is column-stochastic. Under \eqref{PropertiesoftheLinearRisk-SharingRule-Section3-Equation2},
the whole amount of each loss is allocated within the pool, with no residual gain or deficit
at occurrence time.

We now define the individual payments generated by the $k$th pooled claim. Recall from
\eqref{LinearRiskSharing-Section2-Equation8} that the $k$th claim at pool level is generated
by exactly one participant. Let $Z_{i,j,k}$ denote the payment made by participant $i$ at time
$T_k$ when the $k$th claim is generated by participant $j$. Then
\begin{align}
    Z_{i,j,k}=a_{i,j}\, I_{j,k}\, Y_{j,N_{j,T_k}}, \qquad i,j=1,\ldots,n,
    \label{LinearRiskSharing-Section2-Equation9}
\end{align}
where $I_{j,k}$ is the indicator introduced in Section~\ref{sec:Loss model}. Since exactly one indicator among
$I_{1,k},\ldots,I_{n,k}$ is equal to $1$, the total payment made by participant $i$ at the
$k$th occurrence time is
\[
    Z_{i,k}=\sum_{j=1}^{n} Z_{i,j,k}.
\]
Equivalently, if $J_k$ denotes the unique participant such that $I_{J_k,k}=1$, then
\[
    Z_{i,k}=a_{i,J_k}\,Y_{J_k,N_{J_k,T_k}}\text{ and }  \sum_{i=1}^{n} Z_{i,k}=X_k.
\]

For each fixed $i$, the sequence $(Z_{i,k})_{k\ge 1}$ is independent and identically
distributed. More precisely, $Z_{i,k}$ is a mixture of the scaled severities
$a_{i,1}Y_1,\ldots,a_{i,n}Y_n$, with mixing probabilities proportional to the claim
frequencies:
\begin{equation}
\label{DefJk}
    \Prob[J_k=j]=\frac{\lambda_j}{\lambda_{\bullet}}, \qquad j=1,\ldots,n.
\end{equation}
Hence, for $z\ge 0$,
\[
    F_{Z_i}(z)
    =
    \frac{1}{\lambda_{\bullet}}
    \sum_{j:\,a_{i,j}=0}\lambda_j
    +
    \frac{1}{\lambda_{\bullet}}
    \sum_{j:\,a_{i,j}>0}\lambda_j
    F_{Y_j}\!\left(\frac{z}{a_{i,j}}\right).
\]
Its mean is given by
\[
    \mu_{Z_i}
    =
    \Esp[Z_{i,1}]
    =
    \frac{1}{\lambda_{\bullet}}
    \sum_{j=1}^{n}\lambda_j a_{i,j}\textcolor{black}{b_{j}},
\]
and we denote its variance by $\sigma^2_{Z_i}=\Var[Z_{i,1}]$.

Therefore, the cumulative claim amount allocated to participant $i$ up to time $t$ after
pooling is
\[
    S_{i,t}^{\mathrm{pool}}=\sum_{k=1}^{N_t} Z_{i,k}, \qquad t\ge 0,
\]
which is a compound Poisson process with intensity $\lambda_{\bullet}$ and claim-size
distribution $F_{Z_i}$. Keeping the same contribution rate $c_i$ and initial reserve
$\kappa_i$ as in Section~\ref{sec:Loss model}, the pooled surplus process of participant $i$ is
\begin{equation}
\label{LinearRiskSharing-Section2-Equation10}
    V_{i,t}^{\mathrm{pool}}
    =
    c_i t-S_{i,t}^{\mathrm{pool}}+\kappa_i,
    \qquad t\ge 0.
\end{equation}
Whereas $S_{i,t}$ records the claims originally reported by participant $i$,
$S_{i,t}^{\mathrm{pool}}$ records the claims eventually borne by participant $i$ after the
occurrence-time redistribution rule has been applied throughout the pool.

\begin{figure}[p]
    \centering

    % Row 1
    \begin{subfigure}{0.4\textwidth}
        \centering
        \includegraphics[width=\linewidth]{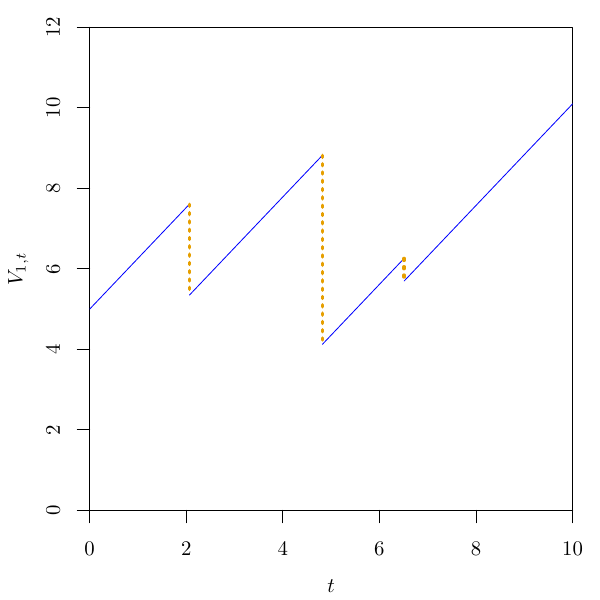}
    \end{subfigure}
    \hfill
    \begin{subfigure}{0.4\textwidth}
        \centering
        \includegraphics[width=\linewidth]{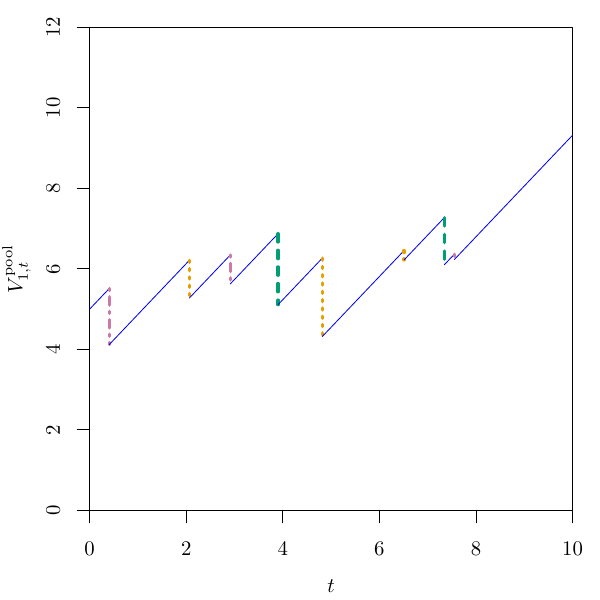}
    \end{subfigure}

    \vspace{-0.3cm}

    % Row 2
    \begin{subfigure}{0.4\textwidth}
        \centering
        \includegraphics[width=\linewidth]{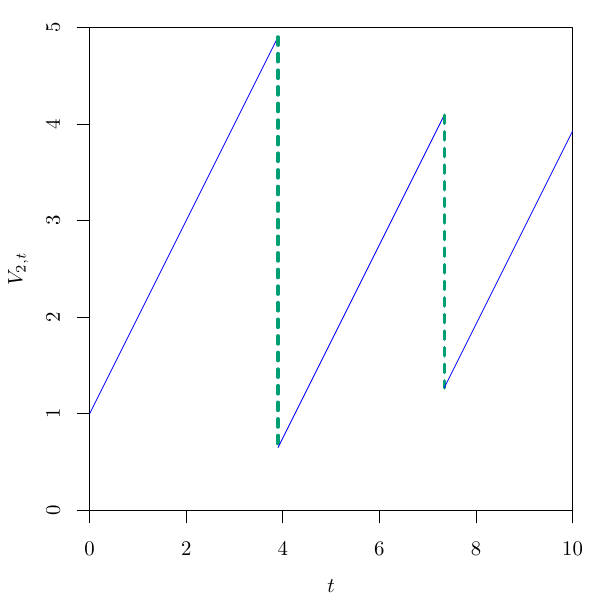}
    \end{subfigure}
    \hfill
    \begin{subfigure}{0.4\textwidth}
        \centering
        \includegraphics[width=\linewidth]{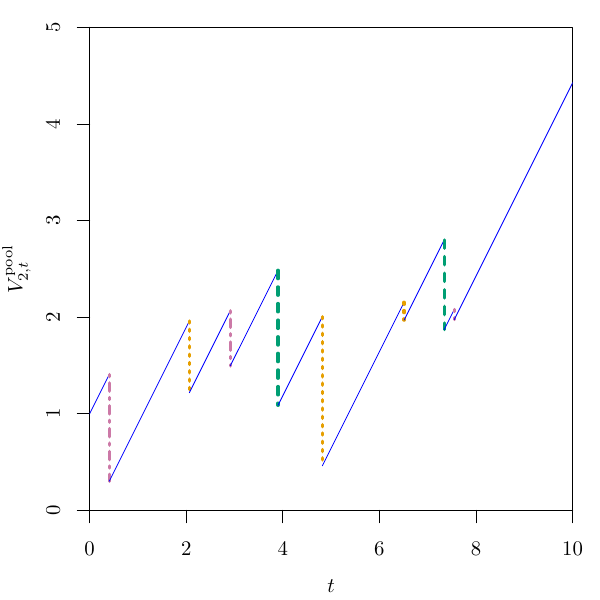}
    \end{subfigure}

    \vspace{-0.3cm}

    % Row 3
    \begin{subfigure}{0.40\textwidth}
        \centering
        \includegraphics[width=\linewidth]{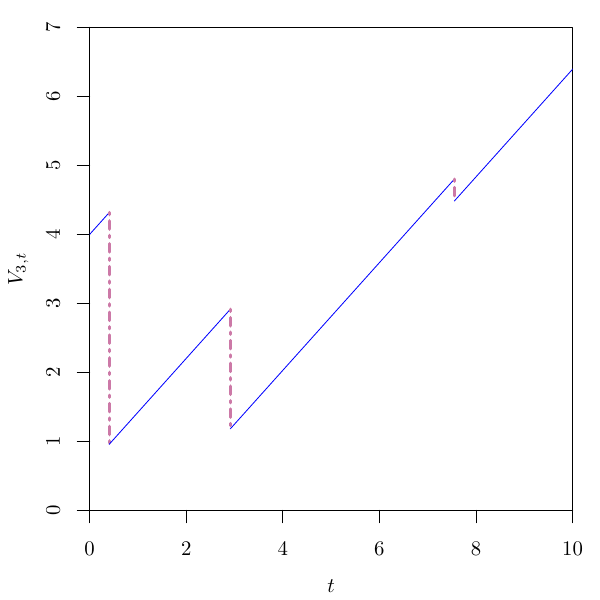}
    \end{subfigure}
    \hfill
    \begin{subfigure}{0.40\textwidth}
        \centering
        \includegraphics[width=\linewidth]{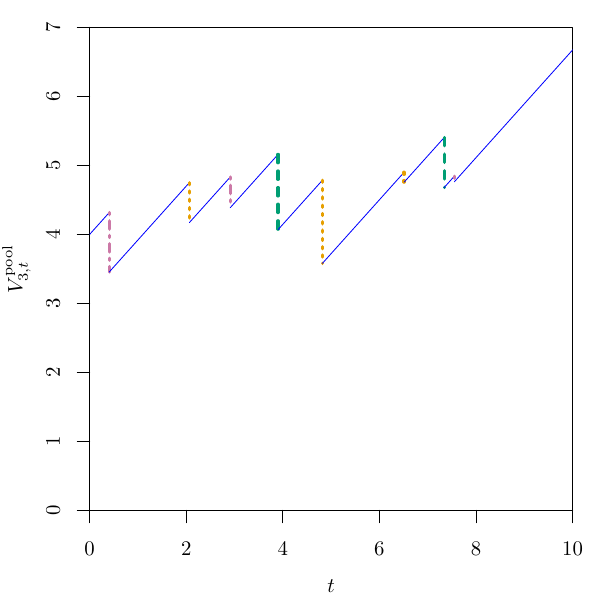}
    \end{subfigure}

    \caption{\textcolor{black}{Illustration of the individual surplus processes under the proportional risk-sharing rule introduced in Example \ref{LinearRiskSharing-Section3-ExampleMP} for a pool of three individuals with initial reserves $\kappa_{1}=5$, $\kappa_{2}=1$ and $\kappa_{3}=4$. The left and right columns correspond to the dynamics before pooling $\{V_{i,t},\hspace{2mm}t\geq 0\}$ and after pooling $\{V^{\text{pool}}_{i,t},\hspace{2mm}t\geq 0\}$, for $i = 1,2,3$, respectively. The model parameters are set to $\eta=2/5$, $\lambda_{1}=0.3$, $\lambda_{2}=0.25$, $\lambda_{3}=0.2$ and $Y_{i,k} \sim \text{LogNormal}(\mu_i,\sigma^{2}_i)$, with $\mu_{1}=0.6$, $\mu_{2}=0.55$, $\mu_{3}=0.525$ and $\sigma^{2}_{1} = \sigma^{2}_{2} = \sigma^{2}_{3}=1$. The origin of each claim is indicated by line style and color: orange dotted for participant 1, teal dashed for participant 2, and pink dot-dashed for participant 3.}
    \label{LinearRiskSharingAtOccurrenceTime-Section3-Figure1}}
\end{figure}

{\color{black}
Figure \ref{LinearRiskSharingAtOccurrenceTime-Section3-Figure1} illustrates how pooling affects the claim experience (i.e., claim arrival times and claim sizes) of an individual\rq s account surplus process. As expected, after pooling, an individual makes more frequent payments, since they are required to participate in the losses incurred by other members of the pool. However, these payments are smaller compared to those made in the stand-alone setting. Furthermore, participant $1$ makes the largest contributions, reflecting the fact that this participant has the highest expected loss per unit of time among all members of the pool.
}

\subsection{Allocation constraints and scale loss family assumption}

\subsubsection{Actuarial fairness constraint}
\label{ACtuarial fairness}

Actuarial fairness means that, for each participant, joining the pool does not change the
expected claim amount borne per unit of time. In other words, the risk-sharing rule redistributes
losses without creating systematic expected gains or losses across participants; see, e.g.,
Denuit et al.\ (2022).

Under stand-alone operation, the cumulative claim amount borne by participant $i$ up to time $t$
is $S_{i,t}$, so that
\begin{align}
    \Esp\!\left[S_{i,t}\right]
    = \Esp[N_{i,t}]\,\Esp[Y_{i,1}]
    = \lambda_i t\, b_i. %\text{ where }b_i:=\Esp[Y_{i,1}]. 
    \label{PropertiesoftheLinearRisk-SharingRule-Section3-Equation3}
\end{align}
Under the proportional sharing rule introduced in Section~\ref{sec:Proportional risk sharing at occurrence time}, the cumulative claim amount
borne by participant $i$ after pooling is $S_{i,t}^{\mathrm{pool}}=\sum_{k=1}^{N_t} Z_{i,k}$. Hence,
\begin{align}
    \Esp\!\left[S_{i,t}^{\mathrm{pool}}\right]
    = \Esp[N_t]\,\Esp[Z_{i,1}]
    = \lambda_\bullet t\,\Esp[Z_{i,1}]
    = t\sum_{j=1}^n \lambda_j a_{i,j} b_j .
    \label{PropertiesoftheLinearRisk-SharingRule-Section3-Equation4}
\end{align}

This leads to the following assumption.

\begin{assumption}[Actuarial fairness]
\label{LinearRiskSharing-Section3-AssumptionAF}
The proportional risk-sharing rule with allocation matrix $\Avec$ is said to be actuarially fair if,
for every participant $i=1,2,\ldots,n$,
\begin{equation}
    \lambda_i b_i = \sum_{j=1}^n \lambda_j a_{i,j} b_j \Leftrightarrow
   \lambda_i b_i(1-a_{i,i})
    =
    \sum_{j\ne i}\lambda_j a_{i,j} b_j .
    \label{PropertiesoftheLinearRisk-SharingRule-Section3-Equation5}
\end{equation}
\end{assumption}

Condition \eqref{PropertiesoftheLinearRisk-SharingRule-Section3-Equation5} states that the
expected loss per unit of time borne by participant $i$ inside the pool coincides with the expected
loss per unit of time borne under stand-alone operation. The equivalent formulation highlights the
balance between the expected loss ceded by participant $i$ to the rest of the pool, namely
$\lambda_i b_i(1-a_{i,i})$, and the expected loss assumed by participant $i$ from the claims of
the other participants, namely $\sum_{j\ne i}\lambda_j a_{i,j} b_j$.
When claim frequencies are homogeneous, that is, $\lambda_1=\cdots=\lambda_n$, Assumption
\ref{LinearRiskSharing-Section3-AssumptionAF} reduces to
$\sum_{j=1}^n a_{i,j} b_j = b_i$, $i=1,\ldots,n$.

We now present two benchmark examples of proportional sharing rules. The first one
satisfies both the full-allocation property and Assumption
\ref{LinearRiskSharing-Section3-AssumptionAF} by construction. The second one always
satisfies full allocation, but is actuarially fair only under an additional symmetry condition.

\begin{example}[Mean-proportional rule]
\label{LinearRiskSharing-Section3-ExampleMP}
Consider the \emph{mean-proportional} (MP) sharing rule defined by
\begin{equation}
    a_{i,j}^{\mathrm{MP}} = a_i^{\mathrm{MP}}
    := \frac{\lambda_i b_i}{\sum_{k=1}^n \lambda_k b_k},
    \qquad i,j=1,\ldots,n.
    \label{PropertiesoftheLinearRisk-SharingRule-Section3-Equation6}
\end{equation}
Under this rule, participant $i$ bears the same fixed fraction $a_i^{\mathrm{MP}}$ of every
claim reported in the pool, independently of its origin.

Since the coefficient does not depend on $j$, each column sum is equal to
\[
    \sum_{i=1}^n a_{i,j}^{\mathrm{MP}}
    =
    \sum_{i=1}^n \frac{\lambda_i b_i}{\sum_{k=1}^n \lambda_k b_k}
    =1,
\]
so the full-allocation property holds.
Moreover, for every participant $i$,
\[
    \sum_{j=1}^n \lambda_j a_{i,j}^{\mathrm{MP}} b_j
    =
    \sum_{j=1}^n \lambda_j
    \frac{\lambda_i b_i}{\sum_{k=1}^n \lambda_k b_k}
    b_j
    =
    \frac{\lambda_i b_i}{\sum_{k=1}^n \lambda_k b_k}
    \sum_{j=1}^n \lambda_j b_j
    =
    \lambda_i b_i.
\]
Hence Assumption \ref{LinearRiskSharing-Section3-AssumptionAF} is satisfied. Therefore,
the mean-proportional rule defines an actuarially fair proportional sharing scheme.
\end{example}

\begin{example}[Uniform-proportional rule]
\label{LinearRiskSharing-Section3-ExampleUP}
Consider now the \emph{uniform-proportional} (UP) sharing rule defined by
\begin{equation}
    a_{i,j}^{\mathrm{UP}}=\frac{1}{n},
    \qquad i,j=1,\ldots,n.
    \label{PropertiesoftheLinearRisk-SharingRule-Section3-Equation7}
\end{equation}
Under this rule, every claim is shared equally among the $n$ participants, regardless of the
identity of the claimant.

The full-allocation property is immediate.
However, actuarial fairness does not hold in general. Indeed, Assumption
\ref{LinearRiskSharing-Section3-AssumptionAF} becomes
\[
    \lambda_i b_i
    =
    \sum_{j=1}^n \lambda_j \frac{1}{n} b_j
    =
    \frac{1}{n}\sum_{j=1}^n \lambda_j b_j,
    \qquad i=1,\ldots,n.
\]
Therefore, the uniform-proportional rule is actuarially fair if, and only if,
\begin{equation}
    \lambda_1 b_1=\lambda_2 b_2=\cdots=\lambda_n b_n.
    \label{PropertiesoftheLinearRisk-SharingRule-Section3-Equation8}
\end{equation}
In other words, equal sharing is actuarially fair precisely when all participants have the same
expected claim amount per unit of time.

As a particular case, if claim frequencies are homogeneous, that is,
$\lambda_1=\cdots=\lambda_n$, then \eqref{PropertiesoftheLinearRisk-SharingRule-Section3-Equation8}
reduces to $b_1=b_2=\cdots=b_n$. 
Hence, under homogeneous claim frequencies, the uniform-proportional rule is actuarially
fair if and only if the participants have identical mean claim severities.
\end{example}

\subsubsection{Capacity constraint}

Actuarial fairness controls the average loss borne by each participant per unit of time after pooling,
but it does not by itself prevent the allocation matrix from assigning to participant $i$ an
excessively large share of claims generated by a participant with a much larger severity level.
To rule out such disproportional transfers, we impose the following bilateral constraint.

\begin{assumption}[Capacity constraint]
\label{LinearRiskSharing-Section3-AssumptionCC}
The proportional risk-sharing rule with allocation matrix $\Avec$ is said to satisfy the
capacity constraint if, for every pair $(i,j)\in\{1,\ldots,n\}^2$,
\begin{equation}
    a_{i,j} b_j \le b_i .
    \label{PropertiesoftheLinearRisk-SharingRule-Section3-Equation9}
\end{equation}

\end{assumption}

Assumption \ref{LinearRiskSharing-Section3-AssumptionCC} means that the expected payment
borne by participant $i$ when a claim from participant $j$ occurs, namely
$a_{i,j}\Esp[Y_{j,1}]=a_{i,j}b_j$, cannot exceed participant $i$'s own mean claim size $b_i$.
Economically, the rule prevents a participant from being asked to absorb, claim by claim,
an exposure that is larger than his or her own typical loss level. Mathematically, this
bilateral bound will be used in the proof of the ruin-probability comparison to control the
dispersion of transferred claims.

%It is worth noting that Assumption
%\ref{LinearRiskSharing-Section3-AssumptionCC} is stronger than actuarial fairness. The
%latter balances expected losses only after aggregation over all claim origins $j$, whereas the
%capacity constraint must hold separately for each bilateral transfer. However, when claim
%frequencies are homogeneous, that is, 

\textcolor{black}{It is worth noting that Assumption \ref{LinearRiskSharing-Section3-AssumptionCC} is more restrictive with respect to the bilateral transfers done by each member of the pool than actuarial fairness. Specifically, the latter balances expected losses only after aggregation over all claim origins $j$, whereas the capacity constraint must hold separately for each bilateral transfer. In general, neither assumption implies the other; however, when claim frequencies are homogeneous, that is,} $\lambda_1=\cdots=\lambda_n$, Assumption
\ref{LinearRiskSharing-Section3-AssumptionAF} implies Assumption
\ref{LinearRiskSharing-Section3-AssumptionCC}. Indeed, in that case actuarial fairness reads
\[
    \sum_{j=1}^n a_{i,j} b_j=b_i,
    \qquad i=1,\ldots,n,
\]
and since all summands are non-negative, one necessarily has $a_{i,j}b_j\le b_i$ for every
$j$.

Assumption \ref{LinearRiskSharing-Section3-AssumptionCC} also provides some intuition on
the admissible heterogeneity of the pool. Summing
\eqref{PropertiesoftheLinearRisk-SharingRule-Section3-Equation9} over $i$ and using the
full-allocation property \eqref{PropertiesoftheLinearRisk-SharingRule-Section3-Equation2},
we obtain
\[
    b_j
    =\sum_{i=1}^n a_{i,j}b_j
    \le \sum_{i=1}^n b_i,
    \qquad j=1,\ldots,n.
\]
Thus, the capacity constraint is compatible with heterogeneous pools, but it prevents
participants with small mean severities from absorbing too large a fraction of claims generated
by participants with much larger mean severities.

For the benchmark rules introduced in Section~\ref{ACtuarial fairness}, the capacity constraint takes a simple
form.

\paragraph{Mean-proportional rule.}
Under the mean-proportional rule of Example \ref{LinearRiskSharing-Section3-ExampleMP},
substituting $a_{i,j}^{\mathrm{MP}}=a_i^{\mathrm{MP}}$ into \eqref{PropertiesoftheLinearRisk-SharingRule-Section3-Equation9} gives
\[
    a_{i}^{\mathrm{MP}} b_j
    =
    \frac{\lambda_i b_i b_j}{\sum_{k=1}^n \lambda_k b_k}
    \le b_i
    \Leftrightarrow
    \lambda_i b_j \le \sum_{k=1}^n \lambda_k b_k .
\]
Hence, Assumption \ref{LinearRiskSharing-Section3-AssumptionCC} holds for the
mean-proportional rule if, and only if,
\begin{equation}
    \lambda_i b_j \le \sum_{k=1}^n \lambda_k b_k,
    \hspace{2mm} i,j=1,\ldots,n\Leftrightarrow     \lambda_i \max_{1\le j\le n} b_j
    \le
    \sum_{k=1}^n \lambda_k b_k,
    \hspace{2mm} i=1,\ldots,n.
    \label{PropertiesoftheLinearRisk-SharingRule-Section3-Equation10}
\end{equation}
In particular, when claim frequencies are homogeneous, the mean-proportional rule
automatically satisfies the capacity constraint, because $ b_j \le \sum_{k=1}^n b_k$ obviously holds true for
$j=1,\ldots,n$.

\paragraph{Uniform-proportional rule.}
Under the uniform-proportional rule of Example
\ref{LinearRiskSharing-Section3-ExampleUP},
the capacity constraint becomes
\begin{equation}
    \frac{1}{n} b_j \le b_i,
    \hspace{2mm} i,j=1,\ldots,n\Leftrightarrow    b_j \le n\, b_i,
    \hspace{2mm} i,j=1,\ldots,n.
    \label{PropertiesoftheLinearRisk-SharingRule-Section3-Equation11}
\end{equation}
Therefore, the uniform-proportional rule satisfies Assumption
\ref{LinearRiskSharing-Section3-AssumptionCC} if, and only if, the mean severities are not
too dispersed across participants. In particular, if the uniform-proportional rule is
actuarially fair under homogeneous claim frequencies, then
$b_1=\cdots=b_n$, and the capacity constraint is automatically satisfied.

\subsubsection{Scale loss family}
\label{Scale loss family constraint}

We now impose a structural assumption on claim severities. It states that the heterogeneity
across participants operates through scale parameters only, while the normalized severities
share a common distribution.

\begin{assumption}[Scale loss family]
\label{LinearRiskSharing-Section3-AssumptionSF}
There exists a non-negative random variable $W$ such that $\Esp[W]=1$ and, for every
$i=1,\ldots,n$ and every $k\ge 1$,
\begin{equation}
    Y_{i,k} \eqD b_i W,
    \qquad b_i>0.
    \label{PropertiesoftheLinearRisk-SharingRule-Section3-Equation12}
\end{equation}
Equivalently, the normalized severities satisfy $ \frac{Y_{i,k}}{b_i} \eqD W$ for $i=1,\ldots,n$ and $k\ge 1$.
\end{assumption}

Assumption \ref{LinearRiskSharing-Section3-AssumptionSF} means that all participants have
the same severity shape after normalization by their mean scale $b_i=\Esp[Y_{i,1}]$.
Thus, differences across participants are captured only by multiplicative scale effects.
This assumption is restrictive, but it provides a transparent framework in which the
dispersion of pooled and stand-alone losses can be compared on a common basis.

Under Assumption \ref{LinearRiskSharing-Section3-AssumptionSF}, the payment made by
participant $i$ at a claim occurrence time remains a scaled version of the same benchmark
random variable $W$. Indeed, when the triggering claim comes from participant $j$, the
payment borne by participant $i$ is
\[
    a_{i,j}Y_{j,1}\eqD a_{i,j} b_j W.
\]
Hence, after normalization by $b_i$, the pooled payment borne by participant $i$ is a
mixture of the random variables $ \frac{a_{i,j}b_j}{b_i}\,W$, $j=1,\ldots,n$.
In particular, under Assumption
\ref{LinearRiskSharing-Section3-AssumptionCC}, all these scaling coefficients satisfy
\begin{align}
    0\le \frac{a_{i,j}b_j}{b_i}\le 1,
    \label{Const_b}
\end{align}
which will be crucial in the proof of the main ruin-probability comparison.

We conclude this subsection with two standard examples.

\begin{example}[Gamma scale family]
\label{LinearRiskSharing-Section3-ExampleGamma}
Let $W$ follow a Gamma distribution with shape parameter $\alpha>0$ and rate parameter
$\alpha$, that is, $W \sim \mathrm{Gamma}(\alpha,\alpha)$, so that
$\Esp[W]=1$ and $\Var[W]=\frac{1}{\alpha}$.
Then, Assumption \ref{LinearRiskSharing-Section3-AssumptionSF} holds if, and only if, for
each participant $i$,
\[
    Y_{i,1}\sim \mathrm{Gamma}\!\left(\alpha,\frac{\alpha}{b_i}\right),
\]
in the shape-rate parametrization. 
Thus, Gamma severities belong to a common scale family precisely when the shape parameter
is common across participants, while the scale parameter is proportional to $b_i$.
\end{example}

\begin{example}[LogNormal scale family]
\label{LinearRiskSharing-Section3-ExampleLognormal}
Let $W$ follow a LogNormal distribution with parameters
\[
    W \sim \mathrm{LogNormal}\!\left(-\frac{\sigma^2}{2},\sigma^2\right),
    \qquad \sigma^2>0,
\]
so that $\Esp[W]=1$. Then Assumption
\ref{LinearRiskSharing-Section3-AssumptionSF} holds if, and only if, for each participant
$i$,
\[
    Y_{i,1}\sim \mathrm{LogNormal}\!\left(\log b_i-\frac{\sigma^2}{2},\sigma^2\right).
\]
Equivalently, a collection of LogNormally-distributed severities $ Y_{i,1}\sim \mathrm{LogNormal}(\mu_i,\sigma_i^2)$,
$i=1,\ldots,n$, belongs to a common scale family if, and only if, the variance parameter is the same for all
participants, that is,
\[
    \sigma_1^2=\cdots=\sigma_n^2=:\sigma^2,
\]
and the location parameters satisfy
\[
    \mu_i=\log b_i-\frac{\sigma^2}{2},
    \qquad i=1,\ldots,n.
\]
Hence, within the LogNormal family, a common scale family is obtained by keeping the
shape parameter $\sigma^2$ fixed and allowing only location shifts across participants.
\end{example}

\subsection{Main result on ruin-reducing effect of pooling}

Define the infinite-time ruin probabilities as
\begin{equation}
\begin{aligned}
    \textcolor{black}{\psi\left(\kappa\right) = \Prob\left[\inf_{t \geq 0} V_{t} < 0 \mid V_{0} = \kappa\right],}& \qquad
    \psi_{i}\left(\kappa_{i}\right) = \Prob\left[\inf_{t \geq 0} V_{i,t} < 0 \mid V_{i, 0} = \kappa_{i}\right], \\
\text{ and } \qquad
\psi_{i}^{\mathrm{pool}}\left(\kappa_{i}\right) &= \Prob\left[\inf_{t \geq 0} V_{i,t}^{\mathrm{pool}} < 0 \mid V_{i,0}^{\mathrm{pool}} = \kappa_{i}\right].
    \label{ImpactofPooling-Section4-Equation1}
\end{aligned}
\end{equation}
\textcolor{black}{The quantity $\psi(\kappa)$ is the probability that the pooled fund will ever become negative, whereas $\psi_i(\kappa_i)$ is the probability that participant $i$ is eventually ruined when operating on a stand-alone basis, and $\psi_i^{\mathrm{pool}}(\kappa_i)$ is the corresponding
probability after joining the pool.} Assumption
\ref{LinearRiskSharing-Section3-AssumptionAF} implies that the pooled model of participant $i$
has the same expected claim amount per unit of time as the stand-alone model, namely
$\lambda_i b_i$. Since $c_i=(1+\eta)\lambda_i b_i$ by \eqref{LinearRiskSharing-Section2-Equation3}, the net profit condition holds for both $V_{i,t}$ and $V_{i,t}^{\mathrm{pool}}$, and thus $ \psi_i(\kappa_i)<1$ and $ \psi_i^{\mathrm{pool}}(\kappa_i)<1$
for all $\kappa_i\ge 0$.

\textcolor{black}{It is important to distinguish between the solvency of the pooled fund and the solvency of the individual accounts after risk sharing. Under the full-allocation condition, the allocation matrix $\Avec$ does not affect the aggregate surplus process of the pooled fund. Indeed, at each claim occurrence time one has $\sum_{i=1}^n Z_{i,k}=X_k$, and therefore $\sum_{i=1}^n S^{\mathrm{pool}}_{i,t}=S_t$ for all $t\geq 0$. Consequently, for a given aggregate contribution rate $c$ and aggregate initial reserve $\kappa$, the pooled fund ruin probability $\psi(\kappa)$, defined in \eqref{ImpactofPooling-Section4-Equation1}, is invariant with respect to the choice of $\Avec$. The role of the sharing matrix is instead to redistribute the aggregate claim stream among the individual accounts, thereby affecting the individual ruin probabilities $\psi_i^{\mathrm{pool}}(\kappa_i)$, also defined in \eqref{ImpactofPooling-Section4-Equation1}. Thus, the financial health of the pooled fund is governed by the aggregate reserve $\kappa=\sum_{i=1}^n\kappa_i$, whereas individual solvency depends both on the individual reserve $\kappa_i$ and on the allocation rule.}

\textcolor{black}{An alternative avenue is to study finite-time ruin probabilities, defined as the probability of ruin occurring within a finite time horizon. In this paper, however, we focus on infinite-time ruin probabilities. This choice is motivated by their greater mathematical tractability compared to their finite-time counterparts, while still providing an informative measure of each participant\rq s solvency. Moreover, in the context of risk-sharing, Denuit and Robert (2023) show that, when a participant has zero initial reserve, its finite-time ruin probability is higher under pooling than in the corresponding stand-alone setting (see Proposition 3.4 in that paper), thus demonstrating that pooling is detrimental. Their numerical illustrations suggest that this adverse effect of pooling on finite-time ruin probabilities also persists for small positive initial reserves. Nevertheless, it diminishes as the time horizon increases, eventually vanishing. This is rather intuitive: pooling makes participants more likely to contribute shortly after joining the pool. Hence, participants with low initial reserves are more likely to experience ruin under pooling than in the corresponding stand-alone setting, even though their contributions are expected to be smaller after pooling takes place. In general, finite-time ruin probabilities under the stand-alone and pooling settings cannot be directly compared for initial reserves greater than zero. Therefore, following Denuit and Robert (2023), our work seeks to determine whether the risk-sharing mechanism introduced in Section \ref{sec:Proportional risk sharing at occurrence time} reduces participants\rq \ infinite-time ruin probabilities when pooling, compared with the corresponding stand-alone setting.}

To compare ruin probabilities, we use the convex order $\lcx$. Recall that for two integrable
random variables $W_1$ and $W_2$, one writes $W_1 \lcx W_2$ if
$ \Esp[W_1]=\Esp[W_2]$ and $ \Esp[(W_1-t)_+] \le \Esp[(W_2-t)_+]$ for all $t\in\mathbb R$.
Equivalently, $W_1 \lcx W_2$ if, and only if, $ \Esp[g(W_1)]\le \Esp[g(W_2)]$
for every convex function $g$ for which the expectations exist. In actuarial terms, this means
that $W_2$ is more variable, or more dispersed, than $W_1$ while having the same mean; see,
for example, Denuit et al.\ (2005).

We are now in a position to state the main result of this section. \textcolor{black}{The proof of Proposition \ref{ImpactofPooling-Section4-Proposition1} is given in Appendix \ref{sec:Proof of Proposition 3.8}.}

\begin{proposition}
\label{ImpactofPooling-Section4-Proposition1}
Suppose that the proportional risk-sharing rule satisfies Assumptions
\ref{LinearRiskSharing-Section3-AssumptionAF},
\ref{LinearRiskSharing-Section3-AssumptionCC}, and
\ref{LinearRiskSharing-Section3-AssumptionSF}. Then, $\psi_i^{\mathrm{pool}}(\kappa_i)\le \psi_i(\kappa_i)$ for every participant
$i=1,\ldots,n$ and all $\kappa_i\ge 0$.
\end{proposition}

Proposition \ref{ImpactofPooling-Section4-Proposition1} gives a clear actuarial interpretation of pooling under the proportional rule. Under Assumptions \ref{LinearRiskSharing-Section3-AssumptionAF}, \ref{LinearRiskSharing-Section3-AssumptionCC}, and \ref{LinearRiskSharing-Section3-AssumptionSF}, joining the pool does not improve solvency through hidden subsidies or through a reduction in the expected loss rate borne by a participant. Instead, the gain comes from a genuine risk-pooling effect: losses are redistributed at occurrence time in such a way that each participant faces a claim stream that is less dispersed than in stand-alone operation, while preserving the same expected claim amount per unit of time. In ruin-theoretic terms, the pool therefore acts as a volatility-smoothing mechanism. Proposition \ref{ImpactofPooling-Section4-Proposition1} shows that, within this framework, a sufficiently well-designed linear sharing rule transforms collective solidarity into an individual solvency benefit, since every participant weakly reduces his or her infinite-time ruin probability by entering the pool.

\subsection{Numerical illustrations} 

The purpose of the numerical illustrations below is to complement Proposition
\ref{ImpactofPooling-Section4-Proposition1} by showing, on concrete examples, how the
risk-reducing effect of pooling appears in terms of infinite-time ruin probabilities. In both
subsections, the parameters are chosen so that the proportional sharing rules under
consideration satisfy the full-allocation property together with Assumptions
\ref{LinearRiskSharing-Section3-AssumptionAF},
\ref{LinearRiskSharing-Section3-AssumptionCC}, and
\ref{LinearRiskSharing-Section3-AssumptionSF}. Hence, the assumptions of Proposition
\ref{ImpactofPooling-Section4-Proposition1} are satisfied. The numerical results therefore
provide explicit illustrations of the theoretical conclusion that, under these conditions,
pooling weakly decreases the ruin probability of every participant.

\subsubsection{Exponentially-distributed losses}
\label{NumericalIllustrations-Section5-Example1}

Even if the choice of the Exponential distribution for modeling losses is questionable in practice,
it remains useful because analytical expressions are available for infinite-time ruin probabilities.
This provides the actuary with a better understanding of the risk-sharing mechanism so that we begin with this case.

Consider a pool of three participants ($n = 3$). For each individual $i = 1,2,3$, the claim amounts $Y_{i,1}, Y_{i,2}, \ldots$ are independent and identically distributed according to the Exponential distribution with rate parameter $\alpha_{i} > 0$, that is,
$ F_{Y_i}(y) = 1 - e^{-\alpha_i y}$, $y \ge 0$.
We consider the following parameter values:
\begin{itemize}
    \item[(i)] Claim frequency: $\lambda_1 = 2$, $\lambda_2 = 1$, and $\lambda_3 = 3$,
    \item[(ii)] Claim severity: $\alpha_1 = \tfrac{1}{2}$, $\alpha_2 = 2$, and $\alpha_3 = 1$.
\end{itemize}
The corresponding mean claim sizes are $ \textcolor{black}{b_{1}} = \tfrac{1}{\alpha_1} = 2$,
$ \textcolor{black}{b_{2}} = \tfrac{1}{\alpha_2} = \tfrac{1}{2}$, and $ \textcolor{black}{b_{3}} = \tfrac{1}{\alpha_3} = 1$.
Additionally, we assume a safety loading factor of $\eta = \tfrac{2}{5}$. Under these conditions, it is clear that the net profit condition \eqref{LinearRiskSharing-Section2-Equation3} is satisfied for the three participants in the pool. Hence, it is also satisfied for the pooled individual risk account for each participant $i$.

Under this setting, the probability density function of the random variables $Z_{i,k}$, for $i = 1,2,3$, is given by
\begin{align}
    f_{Z_i}(z)
    = \frac{1}{\lambda_{\bullet}}\sum_{j=1}^{3}\frac{\lambda_j}{a_{i,j}}
    f_{Y_j}\!\left(\frac{z}{a_{i,j}}\right)
    = p_{1} \, \beta_{i, 1} e^{-\beta_{i, 1} z} + p_{2} \, \beta_{i, 2} e^{-\beta_{i, 2} z} + p_{3} \, \beta_{i, 3} e^{-\beta_{i, 3} z}, \quad z \geq 0,
    \label{NumericalIllustrations-Section5-Equation3}
\end{align}
where $p_j := \frac{\lambda_j}{\lambda_{\bullet}}$ and $\beta_{i,j} := \frac{\alpha_j}{a_{i,j}}$. Thus, $Z_{i,k}$ follows a mixture of Exponential distributions with rate parameters $\beta_{i,j}$ and mixing probabilities $p_j$. In particular, $p_j$ represents the probability that $Z_{i,k}$ is Exponentially distributed with rate $\beta_{i,j}$. 

In this example, we consider the mean-proportional (MP) risk-sharing scheme introduced in Example \ref{LinearRiskSharing-Section3-ExampleMP} with allocation matrix $\Avec^{\mathrm{MP}}$ given by
\begin{align*}
\Avec^{\mathrm{MP}}= \begin{pmatrix}
    \frac{\frac{\lambda_{1}}{\alpha_{1}}}{\frac{\lambda_{1}}{\alpha_{1}} + \frac{\lambda_{2}}{\alpha_{2}} + \frac{\lambda_{3}}{\alpha_{3}}} & \frac{\frac{\lambda_{1}}{\alpha_{1}}}{\frac{\lambda_{1}}{\alpha_{1}} + \frac{\lambda_{2}}{\alpha_{2}} + \frac{\lambda_{3}}{\alpha_{3}}} & \frac{\frac{\lambda_{1}}{\alpha_{1}}}{\frac{\lambda_{1}}{\alpha_{1}} + \frac{\lambda_{2}}{\alpha_{2}} + \frac{\lambda_{3}}{\alpha_{3}}} \\ \\
    \frac{\frac{\lambda_{2}}{\alpha_{2}}}{\frac{\lambda_{1}}{\alpha_{1}} + \frac{\lambda_{2}}{\alpha_{2}} + \frac{\lambda_{3}}{\alpha_{3}}} & \frac{\frac{\lambda_{2}}{\alpha_{2}}}{\frac{\lambda_{1}}{\alpha_{1}} + \frac{\lambda_{2}}{\alpha_{2}} + \frac{\lambda_{3}}{\alpha_{3}}} & \frac{\frac{\lambda_{2}}{\alpha_{2}}}{\frac{\lambda_{1}}{\alpha_{1}} + \frac{\lambda_{2}}{\alpha_{2}} + \frac{\lambda_{3}}{\alpha_{3}}} \\ \\
    \frac{\frac{\lambda_{3}}{\alpha_{3}}}{\frac{\lambda_{1}}{\alpha_{1}} + \frac{\lambda_{2}}{\alpha_{2}} + \frac{\lambda_{3}}{\alpha_{3}}} & \frac{\frac{\lambda_{3}}{\alpha_{3}}}{\frac{\lambda_{1}}{\alpha_{1}} + \frac{\lambda_{2}}{\alpha_{2}} + \frac{\lambda_{3}}{\alpha_{3}}} & \frac{\frac{\lambda_{3}}{\alpha_{3}}}{\frac{\lambda_{1}}{\alpha_{1}} + \frac{\lambda_{2}}{\alpha_{2}} + \frac{\lambda_{3}}{\alpha_{3}}}
    \end{pmatrix} = \begin{pmatrix}
    \frac{8}{15} & \frac{8}{15} & \frac{8}{15} \\ \\
    \frac{1}{15} & \frac{1}{15} & \frac{1}{15} \\ \\
    \frac{2}{5} & \frac{2}{5} & \frac{2}{5}
    \end{pmatrix}.
    \label{NumericalIllustrations-Section5-Equation4}
\end{align*}
%The corresponding $\beta_{i,j}$ are then given by
%$$
% \left(\beta_{i,j}^{\mathrm{MP}}\right) = \begin{pmatrix}
%    \frac{15}{16} & \frac{15}{4} & \frac{15}{8} \\ \\
%    \frac{15}{2} & 30 & 15 \\ \\
%    \frac{5}{4} & 5 & \frac{5}{2}
%    \end{pmatrix}.
%$$

We also consider an alternative (ALT) risk-sharing with allocation matrix $\Avec^{\mathrm{ALT}}$ constructed as
follows. The procedure begins by choosing how much each participant retains from his or her own loss, that is, by specifying the diagonal elements $a_{i,i}^{\mathrm{ALT}}$. These choices must be made with care, as they are required to satisfy the conditions stated in Proposition~\ref{ImpactofPooling-Section4-Proposition1}.
For example, in our setting, participant $2$ is expected to experience the smallest claims, as his or her mean loss $\textcolor{black}{b_{2}}$ is the lowest among the 3 participants. Therefore, it would be unreasonable to require this participant to bear a large share of losses originating from participants $1$ and $3$, as doing so would likely violate the actuarial-fairness condition. In other words, the proportions $a_{2,1}^{\mathrm{ALT}}$ and $a_{2,3}^{\mathrm{ALT}}$ should be relatively small in order to maintain actuarial fairness for participant $2$.
Accordingly, we select the following retention fractions: $a_{1,1}^{\mathrm{ALT}}= 0.800$, $a_{2,2}^{\mathrm{ALT}}= 0.400$, and $a_{3,3}^{\mathrm{ALT}} = 0.700$.

After having selected the diagonal retention levels for each participant, we specify only one additional “free” sharing fraction, namely an off-diagonal element of the matrix $\Avec^{\mathrm{ALT}}$. In this case, we assign participant $1$ (who is expected to incur the largest losses, since $\textcolor{black}{b_{1}}$ is the highest among all participants) to take on $10\%$ of participant $2$ risk, that is, we set $a_{1,2}^{\mathrm{ALT}} = 0.100$.
Now, the full-allocation property (or column-stochasticity of the matrix $\Avec^{\mathrm{ALT}}$) and the actuarial fairness condition do the rest of the work. For $a_{1,3}^{\mathrm{ALT}}$, the actuarial-fairness condition implies
$$
\lambda_{1}\textcolor{black}{b_{1}}a_{1,1}^{\mathrm{ALT}} + \lambda_{2}\textcolor{black}{b_{2}}a_{1,2}^{\mathrm{ALT}} + \lambda_{3}\textcolor{black}{b_{3}}a_{1,3}^{\mathrm{ALT}} = \lambda_{1}\textcolor{black}{b_{1}}\Rightarrow
a_{1,3}^{\mathrm{ALT}} = 0.250.
$$
Similarly, for $a_{2,3}^{\mathrm{ALT}}$, using the full-allocation property, we have the following,
$$
a_{1,3}^{\mathrm{ALT}}+ a_{2,3}^{\mathrm{ALT}}+a_{3,3}^{\mathrm{ALT}}  = 1
\Rightarrow
a_{2,3}^{\mathrm{ALT}} = 0.050.
$$
Then, using actuarial fairness again, we obtain $a_{2,1}^{\mathrm{ALT}}$,
$$
    \lambda_{1}\textcolor{black}{b_{1}}a_{2,1}^{\mathrm{ALT}} + \lambda_{2}\textcolor{black}{b_{2}}a_{2,2}^{\mathrm{ALT}} + \lambda_{3}\textcolor{black}{b_{3}}a_{2,3}^{\mathrm{ALT}}  = \lambda_{2}\textcolor{black}{b_{2}}\Rightarrow
a_{2,1}^{\mathrm{ALT}} = \textcolor{black}{0.038}.
$$
Finally, using the full-allocation property provides us with $a_{3,1}^{\mathrm{ALT}}$ and $a_{3,2}^{\mathrm{ALT}}$ since
\begin{eqnarray*}
a_{3,1}^{\mathrm{ALT}}& =& 1 - a_{2,1}^{\mathrm{ALT}} - a_{1,1}^{\mathrm{ALT}} = \textcolor{black}{0.163}, \\
a_{3,2}^{\mathrm{ALT}} & =& 1 - a_{2,2}^{\mathrm{ALT}} - a_{1,2}^{\mathrm{ALT}} =0.5.
    \label{NumericalIllustrations-Section5-Equation17}
\end{eqnarray*}
This finally results in the following set of transfer ratios:
$$
\Avec^{\mathrm{ALT}} = \begin{pmatrix}
    0.800 & 0.100 & 0.250 \\ \\
    \textcolor{black}{0.038} & 0.400 & 0.050 \\ \\
    \textcolor{black}{0.163} & 0.500 & 0.700 
    \end{pmatrix}.
$$
%resulting in
%$$
%\left( \beta_{i,j}^{\mathrm{ALT}} \right) =  \begin{pmatrix}
%    0.625 & 20 & 4 \\ \\
%    13.3333 & 5 & 20 \\ \\
%    3.07692 & 4 & 1.42857
%    \end{pmatrix}.
%$$
It is straightforward to verify that the proposed alternative proportional risk-sharing scheme based on the transfer ratios $a_{i,j}^{\mathrm{ALT}}$ is actuarially fair.

The Exponential distributions form a scale family with scale parameter $\textcolor{black}{b_j = \tfrac{1}{\alpha_{j}}}$. Moreover, for both risk-sharing schemes based on $\Avec^{\mathrm{MP}}$ and $\Avec^{\mathrm{ALT}}$ under consideration, the condition $0 \leq a_{i,j} b_j \leq b_i$ is satisfied. Consequently, the assumptions of Proposition~\ref{ImpactofPooling-Section4-Proposition1} hold. Furthermore, this setting allows for analytic expressions of ruin probabilities in both the individual and pooled surplus models. Indeed, it is well known that when claim sizes are Exponentially distributed, the ruin probability in the compound Poisson surplus model \eqref{LinearRiskSharing-Section2-Equation2} admits the closed-form expression
\begin{equation}
    \psi_{i}(\kappa_{i})=\frac{\lambda_{i}}{\alpha_{i} c_{i}} e^{-\left(\alpha_{i} -\frac{\lambda_{i}}{c_{i}}\right) \kappa_{i}}.
    \label{NumericalIllustrations-Section5-Equation18}
\end{equation}
See, for example, \textcolor{black}{Equation (5.3.8)} in Rolski et al. (1999) or Corollary 3.2 in Asmussen and Albrecher (2010, \textcolor{black}{Chapter IV}).

Likewise, when claim sizes follow a mixture of Exponential distributions, the ruin probability in the compound Poisson surplus model \eqref{LinearRiskSharing-Section2-Equation10} is available in the closed form
\begin{eqnarray*}
 \psi_1^{\mathrm{pool-MP}}(\kappa_{1}) &=& \textcolor{black}{0.673}e^{\textcolor{black}{-0.341}\kappa_{1}} + \textcolor{black}{0.035}e^{\textcolor{black}{-1.524}\kappa_{1}} + \textcolor{black}{0.007}e^{\textcolor{black}{-3.626}\kappa_{1}},\\
 \psi_2^{\mathrm{pool-MP}}(\kappa_{2}) &=& \textcolor{black}{0.673}e^{\textcolor{black}{-2.726}\kappa_{2}} + \textcolor{black}{0.035}e^{\textcolor{black}{-12.196}\kappa_{2}} + \textcolor{black}{0.007}e^{\textcolor{black}{-29.007}\kappa_{2}},\\
 \psi_3^{\mathrm{pool-MP}}(\kappa_{3}) &=& \textcolor{black}{0.673}e^{\textcolor{black}{-0.454}\kappa_{3}} + \textcolor{black}{0.035}e^{\textcolor{black}{-2.033}\kappa_{3}} +\textcolor{black}{0.007}e^{\textcolor{black}{-4.835}\kappa_{3}},
\end{eqnarray*}
for the mean-proportional linear risk sharing scheme,  
\begin{eqnarray*}
 \psi_1^{\mathrm{pool-ALT}} (\kappa_{1}) &=& \textcolor{black}{0.677}e^{\textcolor{black}{-0.205}\kappa_{1}} +\textcolor{black}{0.034}e^{\textcolor{black}{-3.519}\kappa_{1}} + \textcolor{black}{0.002}e^{\textcolor{black}{-19.83}\kappa_{1}},\\
 \psi_2^{\mathrm{pool-ALT}} (\kappa_{2}) &=& \textcolor{black}{0.616}e^{\textcolor{black}{-2.165}\kappa_{2}} + 0.08e^{-10\kappa_{2}} + \textcolor{black}{0.018}e^{\textcolor{black}{-17.597}\kappa_{2}},\\
 \psi_3^{\mathrm{pool-ALT}} (\kappa_{3}) &=& \textcolor{black}{0.686}e^{\textcolor{black}{-0.475}\kappa_{3}} + \textcolor{black}{0.023}e^{\textcolor{black}{-2.728}\kappa_{3}} + \textcolor{black}{0.005}e^{\textcolor{black}{-3.874}\kappa_{3}},
 \end{eqnarray*}
for the alternative proportional risk-sharing scheme. \textcolor{black}{For the pooled fund, the ruin probability for the compound Poisson surplus model \eqref{LinearRiskSharing-Section2-Equation5} is given by} 
\begin{equation}
\textcolor{black}{\psi(\kappa) = 0.673e^{-0.182\kappa} +0.035e^{-0.813\kappa} + 0.007e^{-1.934\kappa}.}
\label{ruinprob:pooledex1}
\end{equation}
See, e.g., Equations (13.6.9)-(13.6.\textcolor{black}{13}) in Bowers et al. (1997).

Figure \ref{NumericalIllustrations-Section5-Figure1} displays the ruin probabilities for \textcolor{black}{the pooled fund, and} both the individual and pooled surplus models, under each of the two proportional risk-sharing schemes considered in this section for all participants in the pool. Consistent with Proposition \ref{ImpactofPooling-Section4-Proposition1}, the ruin probabilities associated with the pooled surplus process are lower than the corresponding individual ruin probabilities for every participant. \textcolor{black}{The ruin probability \eqref{ruinprob:pooledex1} for the pooled fund is also shown in Figure \ref{NumericalIllustrations-Section5-Figure1}. In practice, participants may contribute different initial reserve amounts, $\kappa_{i}$, for $i=1,2,3$, to the pool, while the fund\rq s overall financial position depends on the total amount contributed, $\kappa = \kappa_{1} + \kappa_{2} + \kappa_{3}$. For example, if we know that $\kappa_{1}=10$ and $\kappa=20$, then participants $2$ and $3$ may have deposited any amounts $\kappa_{2}$ and $\kappa_{3}$, respectively, satisfying $\kappa_{2}+\kappa_{3}=10$. Although different combinations (e.g., $(\kappa_{2},\kappa_{3})=(5,5)$, $(8,2)$, and $(9.5,0.5)$) may lead to the same initial reserve of the pooled fund and hence the same pooled-fund\rq s ruin probability, these combinations may result in different individual ruin probabilities for participants $2$ and $3$. Thus, Figure \ref{NumericalIllustrations-Section5-Figure1} and subsequent figures provide a way to examine the relationship between the financial position of the pooled fund and the solvency of its individual participants.}

\begin{figure}[H]
  	\centering
	% Plot generated with the R code: Example 3.4.1 - Exponentially-distributed losses.R
        % We could generate other plots if needed.
        \begin{subfigure}[b]{0.32\linewidth}
  	\includegraphics[width=4.5cm, height=4.5cm]{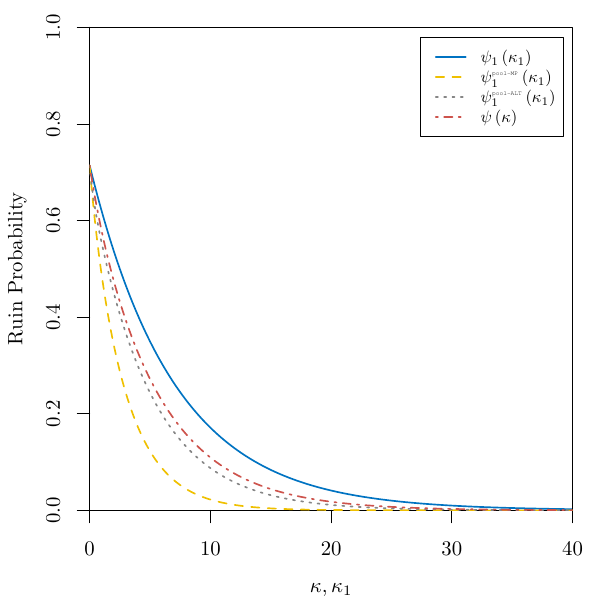}
	\caption{}
  	\label{NumericalIllustrations-Section3-Figure1-a}
	\end{subfigure}
 	 \hspace{0.01cm}
	% Plot generated with the R code: Example 3.4.1 - Exponentially-distributed losses.R
        % We could generate other plots if needed.
         \begin{subfigure}[b]{0.32\linewidth}
 	 \includegraphics[width=4.5cm, height=4.5cm]{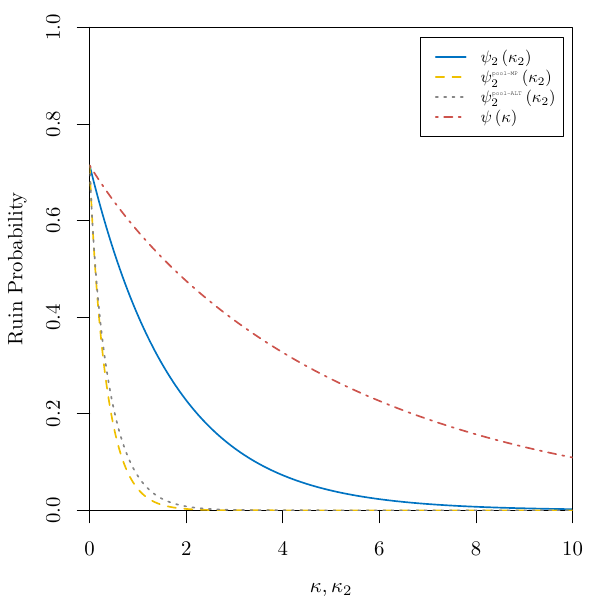}
	 \caption{}
  	\label{NumericalIllustrations-Section3-Figure1-b}
	\end{subfigure}
  	 \hspace{0.01cm}
	% Plot generated with the R code: Example 3.4.1 - Exponentially-distributed losses.R
        % We could generate other plots if needed.
          \begin{subfigure}[b]{0.32\linewidth}
 	 \includegraphics[width=4.5cm, height=4.5cm]{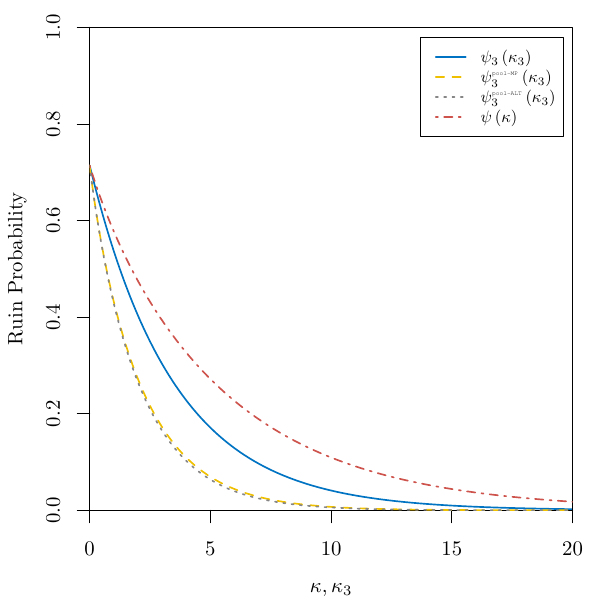}
	 \caption{}
  	\label{NumericalIllustrations-Section3-Figure1-c}
	\end{subfigure}
\caption{\textcolor{black}{Ruin probabilities for the pooled fund, individual and pooled surplus models under the mean-proportional  $\big(\mathrm{pool-MP}\big)$ risk-sharing scheme and the alternative $\big(\mathrm{pool-ALT}\big)$ proportional risk-sharing scheme in the pool of
Section \ref{NumericalIllustrations-Section5-Example1}.}}
\label{NumericalIllustrations-Section5-Figure1}
\end{figure}

\subsubsection{LogNormally-distributed losses}
\label{NumericalIllustrations-Section5-ExampleLN}

Consider a pool of three participants ($n = 3$) where the claim amounts $Y_{i,1}, Y_{i,2}, \ldots$ are now assumed to be independent and distributed according to the LogNormal distribution with parameters $\mu_i$ and $\sigma^2_i$ in Example \ref{LinearRiskSharing-Section3-ExampleLognormal}.
We consider the following parameter values:
\begin{itemize}
    \item[(i)] Claim frequency: $\lambda_1 = 2$, $\lambda_2 = 1$, and $\lambda_3 = 3$,
    \item[(ii)] Claim severity: $\mu_1 = 0$, $\mu_2 = 0.5$, $\mu_3 = -0.3$,
and $\sigma^{2}_1=\sigma^{2}_2 = \sigma^{2}_3 = 1$.
\end{itemize}
The corresponding mean claim sizes are
$\textcolor{black}{b_{1}} = e^{\mu_{1} + \tfrac{\sigma^{2}_{1}}{2}} = \textcolor{black}{1.649}$,
$\textcolor{black}{b_{2}} = e^{\mu_{2} + \tfrac{\sigma^{2}_{2}}{2}} = \textcolor{black}{2.718}$, and
$\textcolor{black}{b_{3}} = e^{\mu_{3} + \tfrac{\sigma^{2}_{3}}{2}} = \textcolor{black}{1.221}$.
Additionally, we assume a safety loading factor of $\eta = \tfrac{2}{5}$. Under these conditions, it is clear that the net profit condition \eqref{LinearRiskSharing-Section2-Equation3} is satisfied for the three participants in the pool. Hence, it is also satisfied for the pooled individual risk account for participant $i$.
In this setting, the probability density function of the random variables $Z_{i,k}$, for $i = 1,2,3$, is given by
$$
    f_{Z_i}(z)
    = \frac{1}{\lambda_{\bullet}}\sum_{j=1}^{3}\frac{\lambda_j}{a_{i,j}}
    f_{Y_j}\!\left(\frac{z}{a_{i,j}}\right)
    = \frac{1}{\lambda_{\bullet}} \sum_{j=1}^{3} \frac{\lambda_{j}}{z \sigma_{\textcolor{black}{j}} \sqrt{2 \pi}} e^{-\frac{\left(\ln \left(\tfrac{z}{a_{i,j}}\right)-\mu_{\textcolor{black}{j}}\right)^2}{2 \sigma^2_{\textcolor{black}{j}}}}, \qquad z > 0.
$$

Again, we consider two different risk-sharing schemes: the mean-proportional risk-sharing scheme
and an alternative scheme defined by
\begin{eqnarray*}
 \Avec^{\mathrm{MP}} & =& \begin{pmatrix}
    \frac{\lambda_{1}e^{\mu_{1} + \tfrac{\sigma^{2}_{1}}{2}}}{\sum\limits^{3}_{j=1}\lambda_{j}e^{\mu_{j} + \tfrac{\sigma^{2}_{j}}{2}}} & \frac{\lambda_{1}e^{\mu_{1} + \tfrac{\sigma^{2}_{1}}{2}}}{\sum\limits^{3}_{j=1}\lambda_{j}e^{\mu_{j} + \tfrac{\sigma^{2}_{j}}{2}}} & \frac{\lambda_{1}e^{\mu_{1} + \tfrac{\sigma^{2}_{1}}{2}}}{\sum\limits^{3}_{j=1}\lambda_{j}e^{\mu_{j} + \tfrac{\sigma^{2}_{j}}{2}}} \\ \\
    \frac{\lambda_{2}e^{\mu_{2} + \tfrac{\sigma^{2}_{2}}{2}}}{\sum\limits^{3}_{j=1}\lambda_{j}e^{\mu_{j} + \tfrac{\sigma^{2}_{j}}{2}}} & \frac{\lambda_{2}e^{\mu_{2} + \tfrac{\sigma^{2}_{2}}{2}}}{\sum\limits^{3}_{j=1}\lambda_{j}e^{\mu_{j} + \tfrac{\sigma^{2}_{j}}{2}}} & \frac{\lambda_{2}e^{\mu_{2} + \tfrac{\sigma^{2}_{2}}{2}}}{\sum\limits^{3}_{j=1}\lambda_{j}e^{\mu_{j} + \tfrac{\sigma^{2}_{j}}{2}}}\\ \\
    \frac{\lambda_{3}e^{\mu_{3} + \tfrac{\sigma^{2}_{3}}{2}}}{\sum\limits^{3}_{j=1}\lambda_{j}e^{\mu_{j} + \tfrac{\sigma^{2}_{j}}{2}}} & \frac{\lambda_{3}e^{\mu_{3} + \tfrac{\sigma^{2}_{3}}{2}}}{\sum\limits^{3}_{j=1}\lambda_{j}e^{\mu_{j} + \tfrac{\sigma^{2}_{j}}{2}}} & \frac{\lambda_{3}e^{\mu_{3} + \tfrac{\sigma^{2}_{3}}{2}}}{\sum\limits^{3}_{j=1}\lambda_{j}e^{\mu_{j} + \tfrac{\sigma^{2}_{j}}{2}}}
    \end{pmatrix} = \begin{pmatrix}
    \textcolor{black}{0.341} & \textcolor{black}{0.341} & \textcolor{black}{0.341} \\ \\
    \textcolor{black}{0.281} & \textcolor{black}{0.281} & \textcolor{black}{0.281} \\ \\
    \textcolor{black}{0.379} & \textcolor{black}{0.379} & \textcolor{black}{0.379}
    \end{pmatrix},
    \label{NumericalIllustrations-Section5-Equation28}
\end{eqnarray*}
and
$$
 \Avec^{\mathrm{ALT}} = \begin{pmatrix}
    0.400 & \textcolor{black}{0.300} & \textcolor{black}{0.317} \\ \\
    \textcolor{black}{0.374} & 0.300 & \textcolor{black}{0.183} \\ \\
    \textcolor{black}{0.226} & \textcolor{black}{0.400} & 0.500
    \end{pmatrix},
$$
respectively. The transfer ratios $a_{i,j}^{\mathrm{ALT}}$ have been obtained by the procedure 
described in Section \ref{NumericalIllustrations-Section5-Example1}.
Both risk-sharing schemes are actuarially fair.

\begin{figure}[H]
  	\centering
	% Plot generated with the R code: Example 3.4.2 - LogNormally-distributed losses.R
        % We could generate other plots if needed.
        \begin{subfigure}[b]{0.32\linewidth}
  	\includegraphics[width=4.5cm, height=4.5cm]{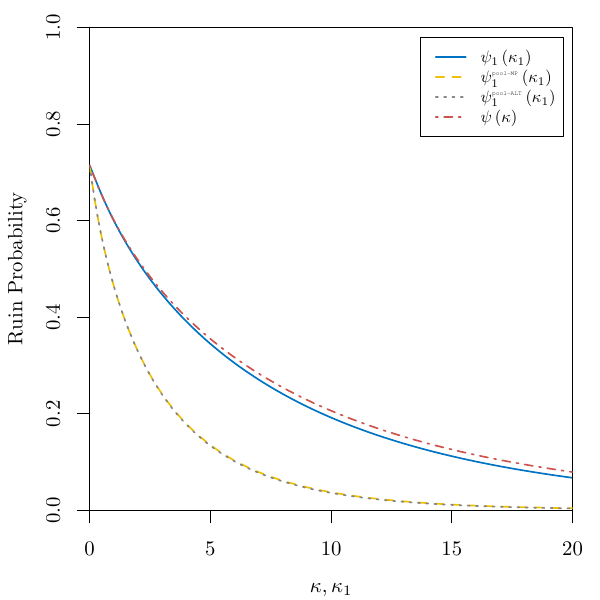}
	\caption{}
  	\label{NumericalIllustrations-Section3-Figure2-a}
	\end{subfigure}
 	 \hspace{0.01cm}
	% Plot generated with the R code: Example 3.4.2 - LogNormally-distributed losses.R
        % We could generate other plots if needed.
         \begin{subfigure}[b]{0.32\linewidth}
 	 \includegraphics[width=4.5cm, height=4.5cm]{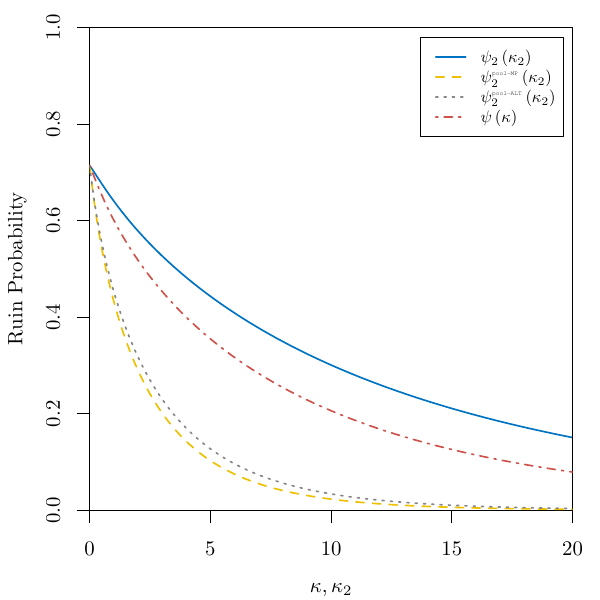}
	 \caption{}
  	\label{NumericalIllustrations-Section3-Figure2-b}
	\end{subfigure}
  	 \hspace{0.01cm}
	% Plot generated with the R code: Example 3.4.2 - LogNormally-distributed losses.R
        % We could generate other plots if needed.
          \begin{subfigure}[b]{0.32\linewidth}
 	 \includegraphics[width=4.5cm, height=4.5cm]{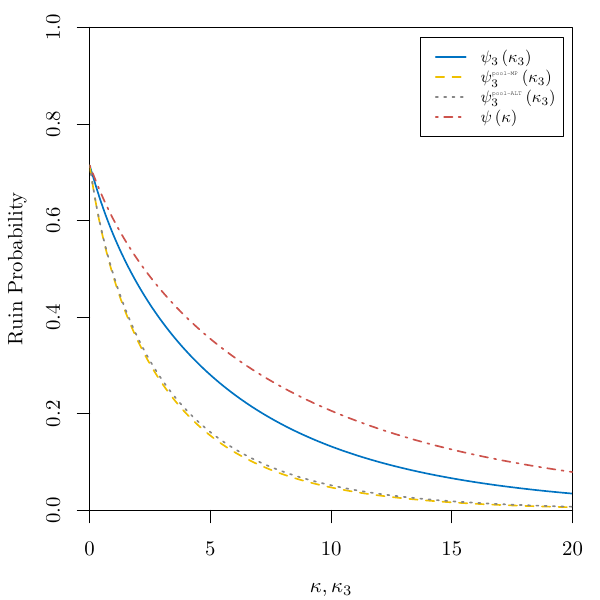}
	 \caption{}
  	\label{NumericalIllustrations-Section3-Figure2-c}
	\end{subfigure}
\caption{\textcolor{black}{Ruin probabilities for the pooled fund, individual and pooled surplus models under the mean-proportional  $\big(\mathrm{pool-MP}\big)$ risk-sharing scheme and the alternative $\big(\mathrm{pool-ALT}\big)$ proportional risk-sharing scheme in the pool of
Section
\ref{NumericalIllustrations-Section5-ExampleLN}.}}
\label{NumericalIllustrations-Section3-Figure2}
\end{figure}

For equal $\sigma_i^{2}$ (as assumed in this section), we know from Example \ref{LinearRiskSharing-Section3-ExampleLognormal} that the LogNormal distributions form a scale family with scale parameter $\textcolor{black}{b_j = e^{\mu_j+\sigma^{2}/2}}$. This restricted family assumes that the ratio of the mean to the median claim size is equal among participants. The condition $0 \leq a_{i,j} b_j \leq b_i$ is satisfied for both risk-sharing rules under consideration. As a result, the assumptions of Proposition~\ref{ImpactofPooling-Section4-Proposition1} hold. Unfortunately, in contrast to Section \ref{NumericalIllustrations-Section5-Example1}, the LogNormal setting does not yield closed-form expressions for infinite-time ruin probabilities.
The representation of the infinite-time ruin probability as the survival function of a compound Geometric sum
resulting from the celebrated Pollaczek–Khinchine formula (also referred to as Beekman's formula in the
actuarial literature) provides a practical way to compute $\psi_{i}(\kappa_{i})$ and $\psi_i^{\mathrm{pool}}(\kappa_{i})$
with the help of Panjer algorithm. 
See, for example, Equation (5.3.16) in Rolski et al. (1999) or Equation (2.2) in Asmussen and Albrecher (2010, \textcolor{black}{Chapter IV}).

Figure \ref{NumericalIllustrations-Section3-Figure2} presents the ruin probabilities for \textcolor{black}{the pooled fund, and} both the individual and pooled surplus models under the two proportional risk-sharing schemes considered in this example (the mean-proportional scheme and the alternative scheme) for all participants in the pool. Consistent with Proposition \ref{ImpactofPooling-Section4-Proposition1}, the ruin probabilities for the pooled surplus process are lower than the corresponding individual ruin probabilities for every participant.

\section{Discussions of the assumptions}
\label{sec:Discussions of the assumptions}

Section~\ref{sec:Linear risk sharing at occurrence time} established a sufficient set of conditions under which proportional pooling reduces
the infinite-time ruin probability of every participant. The purpose of the present section is
to clarify the scope of this result by examining more closely the role of the underlying
assumptions. In particular, we discuss what may happen when Assumption
\ref{LinearRiskSharing-Section3-AssumptionSF} is relaxed, so that claim severities no longer
belong to a common scale family, and when Assumption
\ref{LinearRiskSharing-Section3-AssumptionCC} is not imposed, so that some bilateral
transfers may become too large relative to the participants' own severity levels. These
questions are important because the assumptions of Proposition
\ref{ImpactofPooling-Section4-Proposition1} are sufficient rather than necessary, and the
following subsections are intended to better understand both their practical meaning and
their limitations.

\subsection{On the scale family assumption}
\label{discussion scale family}

\subsubsection{Beyond scale families}

Assumption \ref{LinearRiskSharing-Section3-AssumptionSF} requires the normalized severities $Y_i/b_i$ to be equal in distribution across
participants. This is a convenient sufficient condition for Proposition \ref{ImpactofPooling-Section4-Proposition1}, but it is stronger
than what is actually needed in the convex-order argument underlying the proof. The purpose
of this subsection is to identify a weaker condition that is nevertheless necessary if one wants
the key convex-order comparison to hold for all admissible proportional sharing rules. We
show that, even without assuming a common scale family, the normalized severities must be
ordered in the convex order.

To focus on the role of severities only, we consider here the case of homogeneous claim
frequencies, that is, we assume $\lambda_1=\cdots=\lambda_n=\lambda$.
For each participant $i\in\{1,\ldots,n\}$, let $Y_i\ge 0$ denote the individual claim severity,
with finite mean $b_i=\Esp[Y_i]\in(0,\infty)$. Conditionally on an occurrence in the pool, the triggering participant index $J$ is then
uniformly distributed over $\{1,\ldots,n\}$, that is,
\begin{equation}
    \Prob[J=j]=\frac1n,
    \qquad j=1,\ldots,n.
    \label{BeyondScaleFamilies-Section4-Equation1}
\end{equation}

Let $\Avec=(a_{i,j})_{1\le i,j\le n}$ be a proportional sharing matrix satisfying the full-allocation
property, so that $\sum_{i=1}^n a_{i,j}=1$ holds for every $j=1,\ldots,n$. For participant $i$, the amount borne at an occurrence time is
\begin{equation}
    Z_i:=a_{i,J}Y_J.
    \label{BeyondScaleFamilies-Section4-Equation2}
\end{equation}
As in the proof of Proposition \ref{ImpactofPooling-Section4-Proposition1}, the stand-alone benchmark can be written as
\begin{equation}
    Y_i':=B_iY_i,
    \label{BeyondScaleFamilies-Section4-Equation3}
\end{equation}
where $B_i\sim \mathrm{Bernoulli}(1/n)$ is independent of $Y_i$.

Under homogeneous claim frequencies, actuarial fairness becomes $\sum_{j=1}^n a_{i,j}b_j=b_i$,
$i=1,\ldots,n$, and this already implies the capacity constraint
$a_{i,j}b_j\le b_i$, $i,j=1,\ldots,n$,
because all summands are non-negative. Throughout this subsection, an admissible
sharing matrix means a matrix satisfying full allocation and actuarial fairness.

We can now state the main result of this subsection. \textcolor{black}{The proof of Proposition \ref{BeyondScaleFamilies-Section4-Proposition1} is provided in Appendix \ref{sec:Proof of Proposition 4.1}}.

\begin{proposition}
\label{BeyondScaleFamilies-Section4-Proposition1}
Assume that, for every admissible sharing matrix $\Avec$, the convex-order comparison
\begin{equation}
    Z_i \lcx Y_i'
    \qquad\text{holds for all } i=1,\ldots,n,
    \label{BeyondScaleFamilies-Section4-Equation4}
\end{equation}
with $Z_i$ and $Y_i'$ defined by \eqref{BeyondScaleFamilies-Section4-Equation2} and
\eqref{BeyondScaleFamilies-Section4-Equation3}. Then, for every pair $(i,j)$ such that
$b_i\le b_j$,
\begin{equation}
    \frac{Y_j}{b_j}\lcx\frac{Y_i}{b_i}.
    \label{BeyondScaleFamilies-Section4-Equation5}
\end{equation}
Equivalently, if $b_{(1)}\le \cdots \le b_{(n)}$ denote the ordered means, then
\begin{equation}
    \frac{Y_{(n)}}{b_{(n)}}\lcx     \frac{Y_{(n-1)}}{b_{(n-1)}}\lcx    \cdots \lcx    \frac{Y_{(1)}}{b_{(1)}}.
    \label{BeyondScaleFamilies-Section4-Equation6}
\end{equation}
\end{proposition}

Proposition \ref{BeyondScaleFamilies-Section4-Proposition1} shows that Assumption \ref{LinearRiskSharing-Section3-AssumptionSF} is
not necessary in its full strength. What is necessary, at least if one requires the convex-order
comparison in the proof of Proposition \ref{ImpactofPooling-Section4-Proposition1} to hold for all admissible sharing rules, is the
existence of a monotone convex-order structure across normalized severities. In particular,
participants with larger mean severities must have less dispersed loss ratios $Y_i/b_i$ in the
convex-order sense.

This necessary condition is strictly weaker than equality in distribution of the normalized
severities. Indeed, the chain \eqref{BeyondScaleFamilies-Section4-Equation6} does not imply
that the variables $Y_i/b_i$ are identically distributed. Hence, it does not force a common
scale family. \textcolor{black}{Appendix \ref{sec:Examples illustrating the role of AssumptionSF}} provides explicit counterexamples showing that one may
have \eqref{BeyondScaleFamilies-Section4-Equation5} without having
$Y_1/b_1 \eqD \cdots \eqD Y_n/b_n$.
Conversely, if one had the reverse convex ordering as well, then equality in distribution
would follow, and one would recover the scale-family structure.

{\color{black}
\begin{remark}[On heterogeneous claim frequencies]
The homogeneous-frequency assumption in Proposition~\ref{BeyondScaleFamilies-Section4-Proposition1} is mainly used to isolate the role of claim severities from the role of claim frequencies. Without this assumption, the result cannot be stated only in terms of the normalized severities \(Y_i/b_i\). Indeed, the stand-alone benchmark of participant \(i\), when represented on the aggregate occurrence process, is a zero-inflated loss 
\(B_iY_i\), where \(B_i\sim\mathrm{Bernoulli}(p_i)\) with \(p_i=\lambda_i/\lambda_\bullet\). Hence, when the frequencies are heterogeneous, the Bernoulli thinning probabilities differ across participants.

A partial analogue of Proposition~\ref{BeyondScaleFamilies-Section4-Proposition1} can nevertheless be obtained. Suppose that admissible sharing matrices are required to satisfy full allocation, actuarial fairness and the capacity constraint. Consider two participants \(i\neq j\) such that
$\lambda_i\leq \lambda_j$ and $b_i\leq b_j$.
Then \(m_i=\lambda_i b_i\leq \lambda_j b_j=m_j\). Set
\(\alpha=m_i/m_j\). The two-agent transfer construction used in the proof of Proposition~\ref{BeyondScaleFamilies-Section4-Proposition1} in Appendix \ref{sec:Proof of Proposition 4.1}, with

\[
    a_{i,j}=\alpha,\qquad
    a_{j,j}=1-\alpha,\qquad
    a_{j,i}=1,\qquad
    a_{i,i}=0,
\]
and \(a_{\ell,\ell}=1\) for \(\ell\notin\{i,j\}\), satisfies full allocation, actuarial fairness and the capacity constraint. If the convex-order comparison $Z_\ell\lcx Y'_\ell$, $\ell=1,\ldots,n$,
is required to hold for every admissible sharing matrix, then applying it to this particular matrix yields
$B_j\alpha Y_j \lcx B_iY_i$,
where \(B_j\sim\mathrm{Bernoulli}(p_j)\), \(B_i\sim\mathrm{Bernoulli}(p_i)\), and \(p_\ell=\lambda_\ell/\lambda_\bullet\). Equivalently, after normalization by the common mean \(p_i b_i\),
\[
    B_j\frac{Y_j}{p_jb_j}
    \lcx
    B_i\frac{Y_i}{p_ib_i}.
\]
Thus, in the heterogeneous-frequency case, the relevant necessary condition involves frequency-thinned normalized losses rather than the normalized severities alone. When \(\lambda_i=\lambda_j\), this reduces to the comparison of normalized severities appearing in Proposition~\ref{BeyondScaleFamilies-Section4-Proposition1}. This explains why the homogeneous-frequency assumption is useful for obtaining the simple chain of convex-order comparisons in \eqref{BeyondScaleFamilies-Section4-Equation6}.
\end{remark}
}

\subsubsection{Pooling may not benefit all participants beyond scale families}
\label{SecPoolNotBene}

The previous subsection showed that Assumption \ref{LinearRiskSharing-Section3-AssumptionSF} is stronger than necessary for the
convex-order argument underlying Proposition \ref{ImpactofPooling-Section4-Proposition1}, since a weaker necessary condition can
still hold outside the scale-family framework. The purpose of the present subsection is to
make a different point. We provide a numerical example in which claim severities do not
belong to a common scale family and in which the conclusion of Proposition \ref{ImpactofPooling-Section4-Proposition1} no longer
holds uniformly across participants. More precisely, pooling may reduce the infinite-time
ruin probability for some participants while increasing it for others. Hence, outside the
scale-family setting, the ruin probabilities under stand-alone operation and under pool
participation are not necessarily uniformly comparable.

To construct this example, we build on the setting of Section \ref{Scale loss family constraint}, 
as the LogNormal distributions do not form a scale family when the parameters $\sigma_i^2$ differ across participants. 
We consider a pool of three participants ($n=3$) with LogNormally-distributed losses where
\begin{itemize}
    \item[(i)] Claim frequency: $\lambda_1 = 1$, $\lambda_2 = 1$, and $\lambda_3 = 1$,
    \item[(ii)] Claim severity: $\mu_1 = \textcolor{black}{-3.224}$, $\mu_2 = \textcolor{black}{-0.171}$, and $\mu_3 = \textcolor{black}{0.288}$, $\sigma^{2}_1 = \textcolor{black}{4.615}$, $\sigma^{2}_2 = \textcolor{black}{0.342}$, and $\sigma^{2}_3 = \textcolor{black}{0.236}$.
\end{itemize}
The corresponding mean claim sizes are
$\textcolor{black}{b_{1}} = e^{\mu_{1} + \tfrac{\sigma^{2}_{1}}{2}} = \textcolor{black}{0.4}$,
$\textcolor{black}{b_{2}} = e^{\mu_{2} + \tfrac{\sigma^{2}_{2}}{2}} = \textcolor{black}{1}$, and
$\textcolor{black}{b_{3}} = e^{\mu_{3} + \tfrac{\sigma^{2}_{3}}{2}} = \textcolor{black}{1.5}$,
while the corresponding variances are
\begin{align*}
    \sigma^{2}_{Y_1} = \left(e^{\sigma^{2}_{1}} - 1\right) e^{2\mu_{1} + \sigma^{2}_{1}} &= \textcolor{black}{16.004}, \qquad
    \sigma^{2}_{Y_2} = \left(e^{\sigma^{2}_{2}} - 1\right) e^{2\mu_{2} + \sigma^{2}_{2}} = \textcolor{black}{0.408}, \\
    &\text { and } \quad
    \sigma^{2}_{Y_3} = \left(e^{\sigma^{2}_{3}} - 1\right) e^{2\mu_{3} + \sigma^{2}_{3}} = \textcolor{black}{0.598}.
    %\label{NumericalIllustrations-Section5-Equation32}
\end{align*}
We report the variances as they play an important role in what follows. 

Participant $1$ has the smallest expected loss among the participants but a high relative volatility, as measured by its variance-to-mean ratio, $\tfrac{\sigma^{2}_{Y_{1}}}{\textcolor{black}{b_{1}}} = \tfrac{\textcolor{black}{16.004}}{\textcolor{black}{0.4}} = \textcolor{black}{40.006}$. In contrast, the other two participants are expected to experience larger losses, but their variance-to-mean ratios, $\tfrac{\sigma^{2}_{Y_{2}}}{\textcolor{black}{b_{2}}} = \tfrac{\textcolor{black}{0.408}}{\textcolor{black}{1}} = \textcolor{black}{0.408}$ and $\tfrac{\sigma^{2}_{Y_{3}}}{\textcolor{black}{b_{3}}} = \tfrac{\textcolor{black}{0.598}}{\textcolor{black}{1.5}} = \textcolor{black}{0.399}$, 
are much smaller, indicating relatively stable volatilities relative to their means.

Under these characteristics, since the standard deviation of participant $1$ is much larger than expected loss, the corresponding loss distribution is right-skewed with high tail risk. This tendency toward large loss realizations will be crucial in what follows, as we will see that when risk-sharing rules are not properly designed (even if they satisfy both the full-allocation property and the actuarial-fairness condition), pooling may be detrimental to some participants in the pool.

Additionally, we assume a safety loading factor of $\eta = \tfrac{2}{5}$. Under these conditions, it is clear that the net profit condition \eqref{LinearRiskSharing-Section2-Equation3} is satisfied for the three participants in the pool. Hence, it is also satisfied for the pooled individual risk account for each participant.

As before, we consider the mean-proportional scheme defined by
$$
\Avec^{\mathrm{MP}} = \begin{pmatrix}
     \textcolor{black}{0.138} & \textcolor{black}{0.138} & \textcolor{black}{0.138} \\ \\
     \textcolor{black}{0.345} & \textcolor{black}{0.345} & \textcolor{black}{0.345} \\ \\
     \textcolor{black}{0.517} & \textcolor{black}{0.517} & \textcolor{black}{0.517}
    \end{pmatrix}.
$$
Following the same approach as in Section \ref{NumericalIllustrations-Section5-Example1}, we also construct an alternative set of transfer
ratios. However, we choose these fractions so that participant $3$ bears the majority of the risk originating from participant $1$, while still retaining most of his or her own losses. In this sense, participant $3$ is exposed both to the high volatility of losses from participant $1$ and to limited risk-sharing of his or her own claims with the rest of the pool. This is formalized by
$$
\Avec^{\mathrm{ALT}} = \begin{pmatrix}
     \textcolor{black}{0.080} & \textcolor{black}{0.100} & \textcolor{black}{0.179} \\ \\
     \textcolor{black}{0.020} & \textcolor{black}{0.850} & \textcolor{black}{0.095} \\ \\
     \textcolor{black}{0.900} & \textcolor{black}{0.050} & \textcolor{black}{0.727}
    \end{pmatrix}.
$$
It can be verified that actuarial fairness holds for both the mean-proportional scheme and the proposed alternative scheme.

Figure \ref{NumericalIllustrations-Section4-Figure1} displays the ruin probabilities for \textcolor{black}{the pooled fund, and} both the individual and pooled surplus models under the two risk-sharing schemes under consideration.
As expected, participant $1$ is the only member of the pool who benefits from pooling under both schemes. This is because the corresponding stand-alone ruin probability is high, due to the large volatility of claim sizes relative to the mean, while after pooling its ruin probability is reduced since, under both risk-sharing rules, it retains only a small fraction of its own losses.

\begin{figure}[H]
  	\centering
	% Plot generated with the R code: Example 4.1.2 - LogNormally-distributed losses.R
        % We could generate other plots if needed.
        \begin{subfigure}[b]{0.32\linewidth}
  	\includegraphics[width=4.5cm, height=4.5cm]{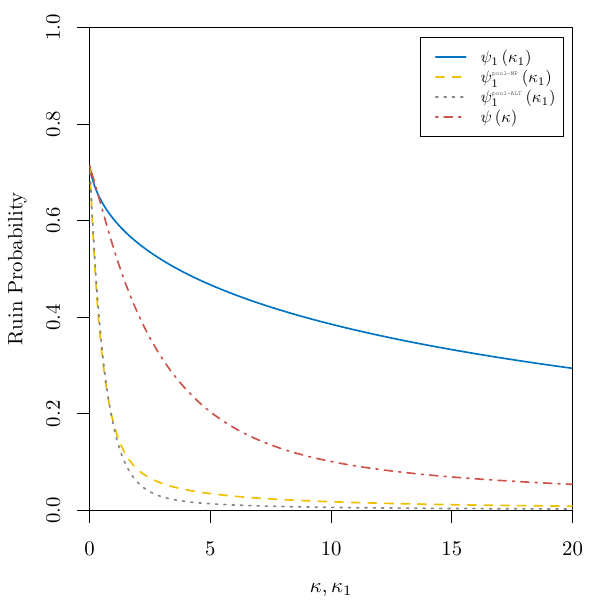}
	\caption{}
  	\label{NumericalIllustrations-Section4-Figure1-a}
	\end{subfigure}
 	 \hspace{0.01cm}
	% Plot generated with the R code: Example 4.1.2 - LogNormally-distributed losses.R
        % We could generate other plots if needed.
         \begin{subfigure}[b]{0.32\linewidth}
 	 \includegraphics[width=4.5cm, height=4.5cm]{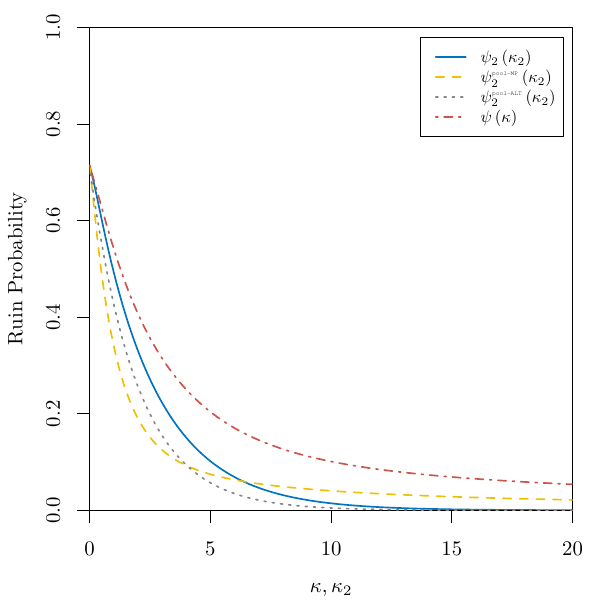}
	 \caption{}
  	\label{NumericalIllustrations-Section4-Figure1-b}
	\end{subfigure}
  	 \hspace{0.01cm}
	% Plot generated with the R code: Example 4.1.2 - LogNormally-distributed losses.R
        % We could generate other plots if needed.
          \begin{subfigure}[b]{0.32\linewidth}
 	 \includegraphics[width=4.5cm, height=4.5cm]{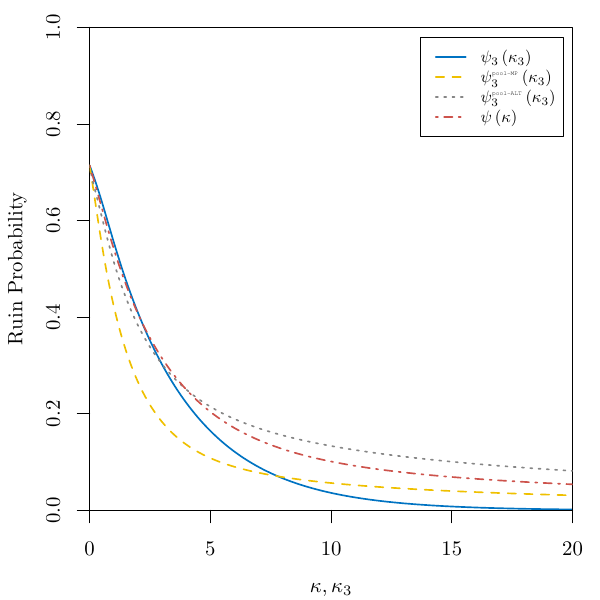}
	 \caption{}
  	\label{NumericalIllustrations-Section4-Figure1-c}
	\end{subfigure}
\caption{\textcolor{black}{Ruin probabilities for the pooled fund, individual and pooled surplus models under the mean-proportional  $\big(\mathrm{pool-MP}\big)$ risk-sharing scheme and the alternative $\big(\mathrm{pool-ALT}\big)$ risk-sharing scheme in the pool of
Section \ref{SecPoolNotBene}.}}
\label{NumericalIllustrations-Section4-Figure1}
\end{figure}

In contrast, participants $2$ and $3$ are adversely affected by pooling (under the mean-proportional scheme in the case of participant $2$, and under both the mean-proportional and alternative schemes in the case of participant $3$) as illustrated in Figures \ref{NumericalIllustrations-Section4-Figure1-b} and \ref{NumericalIllustrations-Section4-Figure1-c}. This outcome is mainly due to the fact that these participants retain a large proportion of their own losses while also contributing a material share to the losses incurred by the other members of the pool.

It is worth noting that pooling remains beneficial for the most vulnerable participants, that is, those with the lowest initial deposits (smallest $\kappa_i$), whereas it is not beneficial for those with higher initial deposits. This is intuitive, since bearing a substantial fraction of the losses incurred by participant $1$ exposes participants $2$ and $3$ to large transfers that may drive them to ruin even when they start from relatively strong financial positions.

\subsection{Capacity constraint}

In contrast with Section~\ref{discussion scale family}, where we showed that the scale loss family assumption is a
convenient sufficient condition but not a necessary one, the purpose of the present section is
to explain why the capacity constraint plays a more essential role in Proposition
\ref{ImpactofPooling-Section4-Proposition1}. More precisely, we show that even when claim
severities belong to a common scale family, and even when the sharing rule satisfies both
the full-allocation property and Assumption
\ref{LinearRiskSharing-Section3-AssumptionAF}, the risk-reducing effect of pooling may fail
as soon as Assumption \ref{LinearRiskSharing-Section3-AssumptionCC} is violated. Thus,
the examples below are designed to isolate the specific contribution of the capacity
constraint in the ruin-probability comparison.

To illustrate this point, we consider two numerical examples with Exponentially distributed
losses. The first example is constructed so that the capacity constraint fails only for one
participant, which makes it possible to identify clearly the local adverse effect of excessive
transfers on that participant’s ruin probability. The second example strengthens this message
by considering a setting in which the constraint is violated for two participants,
simultaneously. Taken together, these examples show that, without Assumption
\ref{LinearRiskSharing-Section3-AssumptionCC}, pooling does not necessarily improve solvency for all
members of the fund, even in a setting where the scale-family assumption is satisfied.

\subsubsection{A numerical illustration with failure of the constraint for one participant}
\label{SubSub421}

We again consider a pool of three participants ($n=3$) with Exponentially-distributed losses, specified as follows:
\begin{itemize}
    \item[(i)] Claim frequency: $\lambda_1=2, \lambda_2=5$, and $\lambda_3=40$,
    \item[(ii)] Claim severity: $\alpha_1=\frac{1}{10}, \alpha_2=\frac{1}{4}$, and $\alpha_3=2$.
\end{itemize}
The corresponding mean claim sizes are then
$\textcolor{black}{b_{1}}=\frac{1}{\alpha_1}=10$, $\textcolor{black}{b_{2}}=\frac{1}{\alpha_2}=4$, and
$\textcolor{black}{b_{3}}=\frac{1}{\alpha_3}=\frac{1}{2}$.

As in the previous section, we assume a safety loading factor of $\eta = \tfrac{2}{5}$ and we adopt either
the mean-proportional or alternative schemes defined by
$$
\Avec^{\mathrm{MP}}=\left(\begin{array}{ccc}
\frac{1}{3} & \frac{1}{3} & \frac{1}{3} \\ \\
\frac{1}{3} & \frac{1}{3} & \frac{1}{3} \\ \\
\frac{1}{3} & \frac{1}{3} & \frac{1}{3}
\end{array}\right)\text{ and }
\Avec^{\mathrm{ALT}}=\left(\begin{array}{ccc}
0.5 & 0.3 & 0.2 \\ \\
0.1 & 0.6 & 0.3 \\ \\
0.4 & 0.1 & 0.5
\end{array}\right),
$$
respectively.

Clearly, for participant 3, both the mean-proportional and alternative schemes violate Assumption \ref{LinearRiskSharing-Section3-AssumptionCC}. Specifically,
\begin{align*}
a^{\mathrm{MP}}_{3, 1} & = \frac{1}{3} \approx \textcolor{black}{0.333} > \frac{b_{3}}{b_{1}} = \frac{\frac{1}{2}}{10} = 0.05, \\ a^{\mathrm{MP}}_{3, 2} &= \frac{1}{3} \approx \textcolor{black}{0.333} > \frac{b_{3}}{b_{2}} = \frac{\frac{1}{2}}{4} = 0.125, \qquad \text{ and } \\ a^{\mathrm{ALT}}_{3, 1} &= 0.4 > \frac{b_{3}}{b_{1}} = \frac{\frac{1}{2}}{10} = 0.05.
\end{align*}

As a result, participant 3 does not benefit from pooling, as illustrated in Figure \ref{NumericalIllustrations-Section4-Figure2}.

\textcolor{black}{A supplementary example analogous to that presented in Section \ref{SubSub421}, but considering failure of constraint \eqref{Const_b} for two participants, is provided in Appendix \ref{sec:Supplementary example illustrating the role of AssumptionCC}.}

\begin{figure}[H]
  	\centering
	% Plot generated with the R code: Example 4.2.1 - Exponentially-distributed losses (Failure for one participant).R
        % We could generate other plots if needed.
        \begin{subfigure}[b]{0.32\linewidth}
  	\includegraphics[width=4.5cm, height=4.5cm]{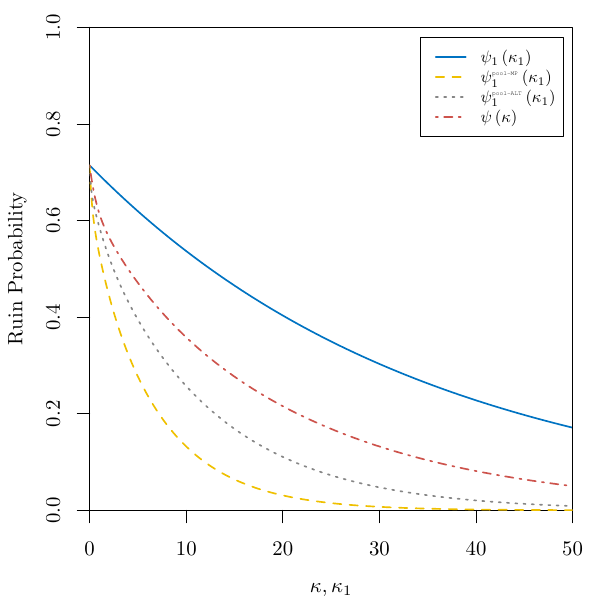}
	\caption{}
  	\label{NumericalIllustrations-Section4-Figure2-a}
	\end{subfigure}
 	 \hspace{0.01cm}
	% Plot generated with the R code: Example 4.2.1 - Exponentially-distributed losses (Failure for one participant).R
        % We could generate other plots if needed.
         \begin{subfigure}[b]{0.32\linewidth}
 	 \includegraphics[width=4.5cm, height=4.5cm]{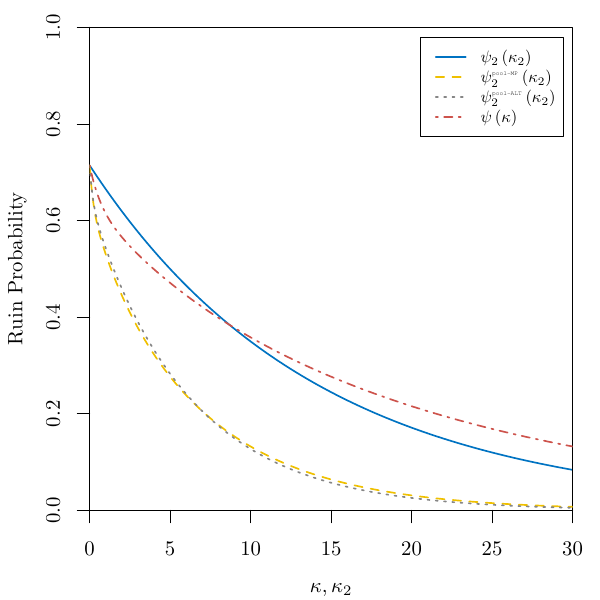}
	 \caption{}
  	\label{NumericalIllustrations-Section4-Figure2-b}
	\end{subfigure}
  	 \hspace{0.01cm}
	% Plot generated with the R code: Example 4.2.1 - Exponentially-distributed losses (Failure for one participant).R
        % We could generate other plots if needed.
          \begin{subfigure}[b]{0.32\linewidth}
 	 \includegraphics[width=4.5cm, height=4.5cm]{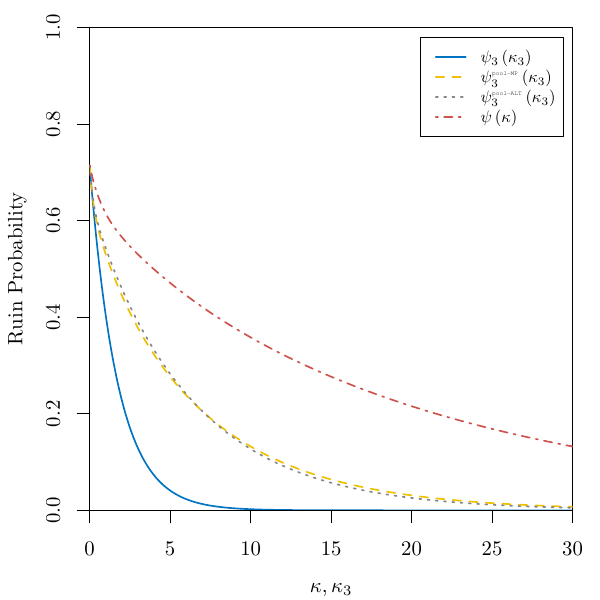}
	 \caption{}
  	\label{NumericalIllustrations-Section4-Figure2-c}
	\end{subfigure}
\caption{\textcolor{black}{Ruin probabilities for the pooled fund, individual and pooled surplus models under the mean-proportional (pool - MP) risk-sharing scheme and the alternative (pool - ALT) proportional risk-sharing scheme in the pool of Section \ref{SubSub421}.}}
\label{NumericalIllustrations-Section4-Figure2}
\end{figure}

{\color{black}
\section{Designing transfer ratios and Pareto-efficient candidates} \label{Pareto-OptimalityforRuinProbabilities-Section5}

\subsection{Pareto-efficient design of transfer ratios}
\label{sec:Pareto-efficient-design-transfer-ratios}

We now briefly discuss how the allocation matrix can be designed when several feasible sharing rules satisfy the sufficient conditions of Proposition~\ref{ImpactofPooling-Section4-Proposition1}. The objective of this section is not to provide a complete characterization of optimal sharing matrices, but rather to show how the preceding constraints can be used to formulate tractable design problems for the transfer ratios. Since improving the solvency of one participant may come at the expense of another, a natural criterion is Pareto efficiency with respect to the vector of individual ruin probabilities.

Let
\begin{align}
\mathcal{A}
=
\left\{
\Avec\in[0,1]^{n\times n}:
\Avec\mvec=\mvec,\quad
\mathbf{1}^{\mathsf{T}}\Avec=\mathbf{1}^{\mathsf{T}},\quad
0\leq a_{i,j}\leq \frac{b_i}{b_j},\ i,j=1,\ldots,n
\right\},
\label{Pareto-OptimalityforRuinProbabilities-Section5-Equation2}
\end{align}
where
\[
\mvec
=
(m_1,m_2,\ldots,m_n)^{\mathsf{T}}
=
(\lambda_1b_1,\lambda_2b_2,\ldots,\lambda_nb_n)^{\mathsf{T}}
\]
is the vector containing the expected losses per unit of time borne under stand-alone operation by the participants in the pool, and $\mathbf{1}$ denotes the $n\times 1$ all-ones vector. The three constraints defining $\mathcal{A}$ correspond respectively to actuarial fairness, full allocation, and the capacity condition.

For $\Avec\in\mathcal{A}$, write
\[
R(\Avec)
=
\left(
\psi^{\mathrm{pool}}_1(\kappa_1;\Avec),
\ldots,
\psi^{\mathrm{pool}}_n(\kappa_n;\Avec)
\right)
\]
for the vector of individual infinite-time ruin probabilities induced by the sharing matrix $\Avec$. A matrix $\Avec^\star\in\mathcal{A}$ is Pareto efficient if there is no $\Avec\in\mathcal{A}$ such that
\[
\psi^{\mathrm{pool}}_i(\kappa_i;\Avec)
\leq
\psi^{\mathrm{pool}}_i(\kappa_i;\Avec^\star),
\qquad i=1,\ldots,n,
\]
with at least one strict inequality.

A standard scalarization approach consists in choosing weights $w_1,\ldots,w_n>0$, with $\sum_{i=1}^n w_i=1$, and solving
\begin{align}
\Avec^\star(w)
\in
\arg\min_{\Avec\in\mathcal{A}}
\sum_{i=1}^n
w_i
\psi^{\mathrm{pool}}_i(\kappa_i;\Avec).
\label{Pareto-OptimalityforRuinProbabilities-Section5-Equation1}
\end{align}
Any global minimizer of \eqref{Pareto-OptimalityforRuinProbabilities-Section5-Equation1}, when all weights are strictly positive, is Pareto efficient. Indeed, if such a minimizer were dominated by another feasible matrix, the weighted sum of ruin probabilities would be strictly smaller, contradicting optimality. Conversely, without additional convexity assumptions on the attainable set $\left\{R(\Avec):\Avec\in\mathcal{A}\right\}$,
weighted-sum scalarizations need not generate all Pareto-efficient matrices. Therefore, the optimization problem \eqref{Pareto-OptimalityforRuinProbabilities-Section5-Equation1} should be viewed as a practical way to identify Pareto-efficient candidates rather than as a complete characterization of all efficient sharing rules.

Other scalarizations could also be considered. For instance, one may seek to maximize the worst individual improvement,
\[
\max_{\Avec\in\mathcal{A}}
\min_{1\leq i\leq n}
\left\{
\psi_i(\kappa_i)
-
\psi^{\mathrm{pool}}_i(\kappa_i;\Avec)
\right\},
\]
or its relative version. We leave a systematic comparison of such design criteria for future research.

In the following subsection, we consider an example in which the optimization problem \eqref{Pareto-OptimalityforRuinProbabilities-Section5-Equation1} admits a relatively tractable mathematical treatment, making it suitable for illustrating the proposed approach. For larger pools or more general loss distributions than those considered in Section~\ref{sec:Example for pool with two participants}, solving the optimization problem generally requires numerical methods.

\subsection{Example for a pool with two participants}
\label{sec:Example for pool with two participants}

As previously noted in Section~\ref{NumericalIllustrations-Section5-Example1}, when participants' losses are Exponentially distributed, the pooled ruin probabilities admit closed-form expressions. This provides a tractable baseline for the design problem \eqref{Pareto-OptimalityforRuinProbabilities-Section5-Equation1}. We now illustrate this point in the simple case of a pool with two participants.

Consider a pool of two participants, $n=2$, with Exponentially distributed losses. The claim frequency intensities are $\lambda_1$ and $\lambda_2$, and the claim severity rates are $\alpha_1$ and $\alpha_2$. The corresponding mean claim sizes are $b_1=\frac{1}{\alpha_1}$ and $b_2=\frac{1}{\alpha_2}$.
Let
\begin{align}
\Avec
=
\begin{pmatrix}
a_{1,1} & a_{1,2}\\
a_{2,1} & a_{2,2}
\end{pmatrix}.
\label{Pareto-OptimalityforRuinProbabilities-Section5-Equation3}
\end{align}
The actuarial fairness condition gives
\[
\lambda_1b_1
=
a_{1,1}\lambda_1b_1+a_{1,2}\lambda_2b_2,
\Rightarrow
a_{1,2}
=
\frac{\lambda_1b_1}{\lambda_2b_2}
(1-a_{1,1}).
\]
The full-allocation property gives $a_{2,1}=1-a_{1,1}$ and $a_{2,2}=1-a_{1,2}$.
Thus, all entries of $\Avec$ can be expressed in terms of a single transfer ratio $a:=a_{1,1}$:
\begin{align}
\Avec(a)
=
\begin{pmatrix}
a
&
\displaystyle
\frac{\lambda_1b_1}{\lambda_2b_2}(1-a)
\\[2mm]
1-a
&
\displaystyle
1-\frac{\lambda_1b_1}{\lambda_2b_2}(1-a)
\end{pmatrix}.
\label{Pareto-OptimalityforRuinProbabilities-Section5-Equation4}
\end{align}
The requirement $0\leq a\leq 1$, together with actuarial fairness, full allocation, and capacity, leads to the admissible range
\begin{align}
\max\left\{
0,\,
1-\frac{\lambda_2b_2}{\lambda_1b_1},\,
1-\frac{b_2}{b_1},\,
1-\frac{\lambda_2}{\lambda_1}
\right\}
\leq a\leq 1.
\label{Pareto-OptimalityforRuinProbabilities-Section5-Equation4bis}
\end{align}

Under \eqref{Pareto-OptimalityforRuinProbabilities-Section5-Equation4}, the payment borne by participant $1$ at a generic pooled claim occurrence is an Exponential mixture with probability density function given by
\begin{align*}
f_{Z_1}(z)
&=
\frac{1}{\lambda_{\bullet}}
\left(
\frac{\alpha_1\lambda_1}{a}
\exp\left\{-\frac{\alpha_1}{a}z\right\}
+
\frac{\alpha_2\lambda_2^2b_2}{\lambda_1b_1(1-a)}
\exp\left\{
-\frac{\alpha_2\lambda_2b_2}{\lambda_1b_1(1-a)}z
\right\}
\right),
\qquad z\geq 0.
\end{align*}
Similarly, the density of the payment borne by participant $2$ is
\begin{align*}
f_{Z_2}(z)
&=
\frac{1}{\lambda_{\bullet}}
\left(
\frac{\alpha_1\lambda_1}{1-a}
\exp\left\{-\frac{\alpha_1}{1-a}z\right\}\right.\\
&\left.+
\frac{\alpha_2\lambda_2^2b_2}
{\lambda_2b_2-\lambda_1b_1(1-a)}
\exp\left\{
-\frac{\alpha_2\lambda_2b_2}
{\lambda_2b_2-\lambda_1b_1(1-a)}z
\right\}
\right),
\qquad z\geq 0.
\end{align*}
The corresponding moment generating functions are
\begin{align}
M_{Z_1}(r)
&=
\frac{1}{\lambda_{\bullet}}
\left(
\frac{\alpha_1\lambda_1}{\alpha_1-ar}
+
\frac{\alpha_2\lambda_2^2b_2}
{\alpha_2\lambda_2b_2-\lambda_1b_1(1-a)r}
\right),
\notag\\
&\hspace{35mm}
 r<
\min\left\{
\frac{\alpha_1}{a},
\frac{\alpha_2\lambda_2b_2}{\lambda_1b_1(1-a)}
\right\},
\label{Pareto-OptimalityforRuinProbabilities-Section5-Equation44}
\\
M_{Z_2}(r)
&=
\frac{1}{\lambda_{\bullet}}
\left(
\frac{\alpha_1\lambda_1}{\alpha_1-(1-a)r}
+
\frac{\alpha_2\lambda_2^2b_2}
{\alpha_2\lambda_2b_2-
\left(\lambda_2b_2-\lambda_1b_1(1-a)\right)r}
\right),
\notag\\
&\hspace{35mm}
 r<
\min\left\{
\frac{\alpha_1}{1-a},
\frac{\alpha_2\lambda_2b_2}
{\lambda_2b_2-\lambda_1b_1(1-a)}
\right\}.
\notag
\end{align}

Using the classical expression for the infinite-time ruin probability with a mixture of Exponential claim sizes, see Equation~(13.6.13) in Bowers et al.~(1997), the ruin probabilities can be written as finite sums of exponentials:
\[
\psi^{\mathrm{pool}}_i(\kappa_i;\Avec(a))
=
\sum_{k=1}^{2}
C_{i,k}(a)
\exp\{-r_{i,k}(a)\kappa_i\},
\qquad i=1,2,
\]
where the coefficients $C_{i,k}(a)$ and adjustment roots $r_{i,k}(a)$ are obtained by partial fraction decomposition from $M_{Z_i}$. More explicitly, for participant $1$, set
\[
p=\frac{\lambda_1}{\lambda_{\bullet}},
\qquad
k_{1,1}=\frac{a}{\alpha_1},
\qquad
k_{1,2}=
\frac{\lambda_1b_1(1-a)}{\alpha_2\lambda_2b_2},
\]
so that
\[
\Esp[Z_1]
=
pk_{1,1}+(1-p)k_{1,2}
=
\frac{\lambda_1}{\alpha_1\lambda_{\bullet}}.
\]
The two roots $r_{1,1}(a)$ and $r_{1,2}(a)$ are the roots of
\[
P_1(r)
=
(1+\eta)\Esp[Z_1](1-k_{1,1}r)(1-k_{1,2}r)
-\Esp[Z_1]
+k_{1,1}k_{1,2}r.
\]
The associated constants are obtained from the residue formula. Rewriting $P_1(r)$ as follows
\[
P_1(r)=c^{P}_{1,2}r^2+c^{P}_{1,1}r+c^{P}_{1,0},
\]
then
\[
B_{1,\ell}
=
-
\frac{
\eta\left(\Esp[Z_1]-k_{1,1}k_{1,2}r_{1,\ell}\right)
}
{
(1+\eta)P_1'(r_{1,\ell})
},
\qquad
C_{1,\ell}
=
\frac{B_{1,\ell}}{r_{1,\ell}},
\qquad
\ell=1,2.
\]

\begin{figure}[H]
    \begin{subfigure}[b]{0.5\linewidth}
          \includegraphics[width=7.5cm, height=7.5cm]{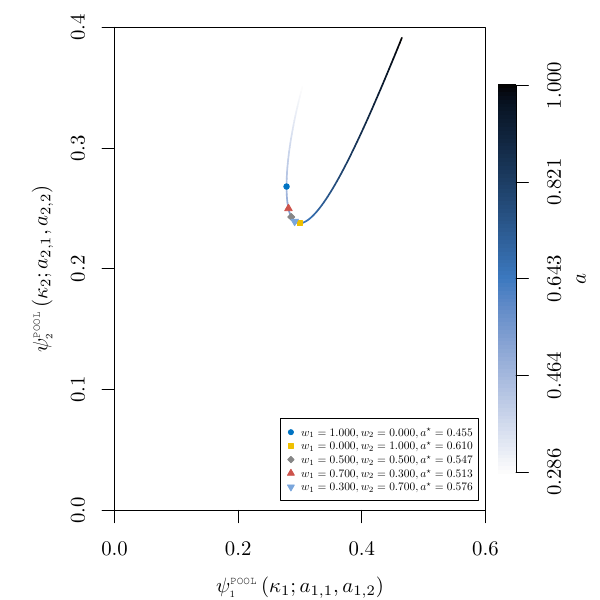}
        \caption{}
          \label{Pareto-OptimalityforRuinProbabilities-Section5-Figure1-a}
    \end{subfigure}
    \begin{subfigure}[b]{0.5\linewidth}
          \includegraphics[width=7.5cm, height=7.5cm]{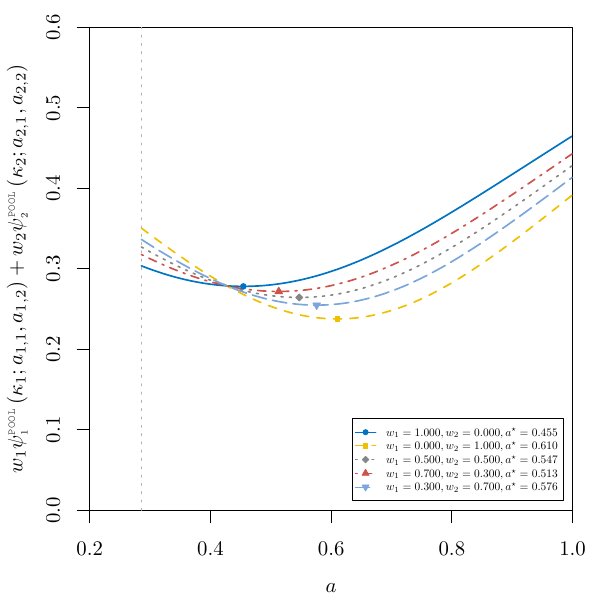}
        \caption{}
          \label{Pareto-OptimalityforRuinProbabilities-Section5-Figure1-b}
    \end{subfigure}
    \caption{\textcolor{black}{Weighted-sum optimal transfer ratios $a^{\star}$ for a pool of two individuals with initial reserves $\kappa_{1}=\kappa_{2}=3$. The model parameters are $\eta=2/5$, $\lambda_{1}=0.5$, $\lambda_{2}=0.6$, and $Y_{i,k}\sim\operatorname{Exponential}(\alpha_i)$ for $i=1,2$, with $\alpha_{1}=0.5$ and $\alpha_{2}=0.7$. The optimal transfer ratios are computed under five different weight scenarios. Plot (a) shows $\psi^{\mathrm{pool}}_{1}(\kappa_{1};a_{1,1},a_{1,2})$ versus $\psi^{\mathrm{pool}}_{2}(\kappa_{2};a_{2,1},a_{2,2})$ for varying values of $a$, with markers indicating the weighted-sum optimal transfer ratio $a^{\star}$ for each scenario. Plot (b) shows the objective function \eqref{Pareto-OptimalityforRuinProbabilities-Section5-Equation11} as a function of $a$, with markers indicating the corresponding minimizers $a^{\star}$. The cases with one zero weight are included as single-participant benchmarks; the Pareto-efficiency interpretation applies to minimizers obtained with strictly positive weights.}}
    \label{Pareto-OptimalityforRuinProbabilities-Section5-Figure1}
\end{figure}

The same calculation applies to participant $2$ with
\[
k_{2,1}=\frac{1-a}{\alpha_1},
\hspace{3mm}
k_{2,2}
=
\frac{\lambda_2b_2-\lambda_1b_1(1-a)}
{\alpha_2\lambda_2b_2},
\text{ and }
\Esp[Z_2]
=
\frac{\lambda_2}{\alpha_2\lambda_{\bullet}}.
\]

Consequently, in the two-participant Exponential case, the weighted-sum design problem reduces to the one-dimensional problem
\begin{align}
\min_{a}
\Bigg\{
&w_1
\left(
C_{1,1}(a)e^{-r_{1,1}(a)\kappa_1}
+
C_{1,2}(a)e^{-r_{1,2}(a)\kappa_1}
\right)
\notag\\
&\quad+
 w_2
\left(
C_{2,1}(a)e^{-r_{2,1}(a)\kappa_2}
+
C_{2,2}(a)e^{-r_{2,2}(a)\kappa_2}
\right)
\Bigg\},
\label{Pareto-OptimalityforRuinProbabilities-Section5-Equation11}
\end{align}
subject to the admissible range \eqref{Pareto-OptimalityforRuinProbabilities-Section5-Equation4bis}. We denote by $a^\star$ a global minimizer of \eqref{Pareto-OptimalityforRuinProbabilities-Section5-Equation11} and refer to it as the weighted-sum optimal transfer ratio.

Figure~\ref{Pareto-OptimalityforRuinProbabilities-Section5-Figure1} illustrates the behavior of $a^\star$ under different weight scenarios. The cases $w_1=1,w_2=0$ and $w_1=0,w_2=1$ correspond to single-participant benchmark objectives, whereas strictly positive weights correspond to genuine multi-objective scalarizations. As the relative weight assigned to a participant increases, the optimal transfer ratio moves toward the value that is most favorable to that participant. Figure~\ref{Pareto-OptimalityforRuinProbabilities-Section5-Figure1-a} displays the attainable pairs of ruin probabilities, while Figure~\ref{Pareto-OptimalityforRuinProbabilities-Section5-Figure1-b} shows the corresponding weighted objective functions as functions of $a$.

}

\section{Conclusion}
\label{sec:Conclusion}

Community-based insurance schemes (CBIs) play a crucial role in many resource-constrained communities where the availability of conventional insurance products is limited. Indeed, CBIs are prevalent in low-income communities and serve as one of the most common mechanisms for mitigating risk. This article examined the benefits of joining CBIs, focusing on whether pooling reduces individual ruin probabilities and makes participation preferable to remaining outside these risk-sharing arrangements.

Recognizing that simple and transparent risk mitigation strategies are essential in communities of this kind, this article focuses on CBIs with linear risk-sharing rules, in which each participant contributes a fraction of the loss suffered by the pool at the time it occurs. Such rules are widely regarded as easier to interpret and implement in practice, making them a natural choice for reducing complexity while enhancing understanding and trust among community members.

The analysis developed in this paper shows that proportional risk sharing at claim occurrence time can reduce the infinite-time ruin probability of every participant in a community-based insurance pool, provided that the allocation rule satisfies full allocation, actuarial fairness, and the capacity constraint, and that claim severities belong to a common scale family. Under these assumptions, pooling improves individual solvency not by lowering the expected loss borne by each participant, but by reducing the dispersion of the claim stream faced inside the pool. In this sense, the benefit of participation comes from a genuine risk-pooling effect. The numerical illustrations confirm this theoretical result, while the discussion in Section~\ref{sec:Discussions of the assumptions} clarifies its scope by showing, on the one hand, that the scale-family assumption is sufficient but not necessary in full strength and, on the other hand, that violations of the capacity constraint may destroy the uniform ruin-reduction effect. Overall, the paper highlights that transparent linear sharing rules can improve individual solvency, but only when they are designed so as to balance fairness and protection against excessive bilateral transfers.

It is important to acknowledge some limitations of our study. In its current form, the conditions under which contributions keep the CBI beneficial are identified only when individual losses belong to the same scale family. Furthermore, individual accounts are modeled using the simplest form of the classical Cramér–Lundberg risk process. Alternative models, such as those discussed by Asmussen and Albrecher (2010) and Mandjes and Boxma (2023) may offer a more realistic representation and yield additional insights in specific situations.

{\color{black} One of the drawbacks of the model introduced in Section \ref{sec:Loss model} is the assumption of independent compound Poisson processes across participants in the pool. In practice, community members are typically exposed to risks exhibiting a certain degree of dependence. In Appendix \ref{sec:Incorporating dependencies}, we introduce a common-shock model that reduces to some extent the limitations of this assumption. The main results of this paper are extended to this particular dependence setting. Despite the independence assumption, the model introduced in Section \ref{sec:Loss model} captures the key features of the risk-sharing mechanism and provides valuable insights into its implications.

In addition to finite-time ruin probabilities, an important avenue for future research would be to assess the effect of pooling on other risk measures commonly studied in the classical Cramér–Lundberg risk process. These include the time to ruin, the surplus immediately prior to ruin, the severity of ruin, the time in red (the time spent below zero from the first ruin to the first time the surplus trajectory upcrosses the level zero) and the amount of the claim causing ruin (see Gerber et al. (1987), Dufresne and Gerber (1988), Egídio dos Reis (1993), Picard (1994), Gerber and Shiu (1998) and Loisel (2005) for instance).}

Finally, future research could also explore subsidized CBIs through the lens of ruin theory (see, e.g., Flores-Contró et al. (2025)), to assess how subsidies affect both the benefits and the conditions governing contributions. This is particularly relevant given that, over the past decades, governments in many countries have promoted CBIs and fostered community participation through subsidies. See Zhang and Wang (2008) for an example of a subsidized community-based health insurance (CBHI) scheme in China and Parmar et al. (2012) for another example in Burkina \textcolor{black}{Faso}, where a subsidy of 50\% of the premium is offered to poor households. 

\section*{Github 
repository}
A repository containing the R scripts used to generate the numerical illustrations presented in this paper is available on GitHub at: \url{https://github.com/josemiguelflorescontrouclouvain/CommunityBasedInsuranceintheCompoundPoissonSurplusModel.git}.

\section*{Funding} \label{Funding-Section} 

Michel Denuit and José Miguel Flores-Contró gratefully acknowledge funding from the FWO and F.R.S.-FNRS under the Excellence of Science (EOS) programme, project ASTeRISK (40007517).

{\color{black} \section*{Acknowledgements}  \label{Acknowledgements-Section}}

\textcolor{black}{The authors are grateful to the anonymous referees for their numerous helpful and constructive comments on an earlier version of this study.}

\section*{References}

\begin{enumerate}[label={[\arabic*]}]

\item \textcolor{black}{Abdikerimova, S., T. J. Boonen, and R. Feng (2024). Multiperiod Peer-to-Peer Risk Sharing. Journal of Risk and Insurance 91 (4), 943–982.}

\item Abdikerimova, S. and R. Feng (2022). Peer-to-Peer Multi-Risk Insurance and Mutual Aid. European Journal of Operational Research 299 (2), 735–749.

\item \textcolor{black}{Albrecher, H., P. Azcue, and N. Muler (2017). Optimal Dividend Strategies for Two Collaborating Insurance Companies. Advances in Applied Probability 49 (2), 515–548.}

%\item \textcolor{black}{Ali, M. J. and K. Pärna (2022). Ruin probability for merged risk processes with correlated arrivals. In A. Malyarenko, Y. Ni, M. Rančić, and S. Silvestrov (Eds.), Stochastic Processes, Statistical Methods, and Engineering Mathematics, Cham, pp. 15–32. Springer International Publishing.}

\item Asmussen, S. and H. Albrecher (2010). Ruin Probabilities. Singapore: World Scientific.

\item \textcolor{black}{Boonen, T. J. and Z. Zhou (2026). Robust Peer-to-Peer Risk Sharing in Continuous Time. Available at SSRN.}

%\item Bhattamishra, R. and C. B. Barrett (2010). Community-Based Risk Management Arrangements: A Review. World Development 38 (7), 923–932.

\item Bowers, N. L., H. U. Gerber, J. C. Hickman, and C. J. Nesbitt (1997). Actuarial Mathematics. Illinois, United States of America: Society of Actuaries.

\item Carrin, G., M.-P. Waelkens, and B. Criel (2005). Community-Based Health Insurance in Developing Countries: A Study of its Contribution to the Performance of Health Financing Systems. Tropical medicine \& international health 10 (8), 799–811.

%\item Chantarat, S., A. G. Mude, C. B. Barrett, and M. R. Carter (2013). Designing Index-Based Livestock Insurance for Managing Asset Risk in Northern Kenya. Journal of Risk and Insurance 80 (1), 205–237.

\item Charpentier, A. and P. Ratz (2025). Linear Risk Sharing on Networks. Working Paper.

%\item Choong, J. J., D. Wagenaar, M. L. Rabonza, P. Hamel, A. D. Switzer, and D. Lallemant (2025). Shared Hazards, Unequal Outcomes: Income-Driven Inequities in Disaster Risk. npj Natural Hazards 2 (1), 33.

%\item Churchill, C. (2007). Insuring the Low-Income Market: Challenges and Solutions for Commercial Insurers. The Geneva Papers on Risk and Insurance - Issues and Practice 32 (3), 401–412.

\item Cronk, L. and A. Aktipis (2021). Design Principles for Risk-Pooling Systems. Nature Human Behaviour 5 (7), 825–833.

\item Denuit, M. (2019). Size-Biased Transform and Conditional Mean Risk Sharing. ASTIN Bulletin 49 (3), 591–617.

\item Denuit, M. and J. Dhaene (2012). Convex Order and Comonotonic Conditional Mean Risk Sharing. Insurance: Mathematics and Economics 51 (2), 265–270.

\item Denuit, M., J. Dhaene, M. Goovaerts, and R. Kaas (2005). Actuarial Theory for Dependent Risks: Measures, Orders and Models. Singapore: John Wiley \& Sons, Ltd.

\item Denuit, M., J. Dhaene, and C. Y. Robert (2022). Risk-Sharing Rules and their Properties, with Applications to Peer-to-Peer Insurance. Journal of Risk and Insurance 89 (3), 615–667.

\item Denuit, M. and C. Y. Robert (2023). Conditional Mean Risk Sharing of Losses at Occurrence Time in the Compound Poisson Surplus Model. Insurance: Mathematics and Economics 112, 23–32.

\item Denuit, M. and C. Y. Robert (2024). Conditional Mean Risk Sharing of Independent Discrete Losses in Large Pools. Methodology and Computing in Applied Probability 26 (4), 36.

\item Dercon, S. (2002). Income Risk, Coping Strategies, and Safety Nets. The World Bank Research Observer 17 (2), 141–166.

\item Dercon, S., J. De Weerdt, T. Bold, and A. Pankhurst (2006). Group-Based Funeral Insurance in Ethiopia and Tanzania. World Development 34 (4), 685–703.

\item Dhaene, J. and M. Denuit (1999). The Safest Dependence Structure Among Risks. Insurance: Mathematics and Economics 25 (1), 11–21.

\item Dhaene, J., C. Y. Robert, K. C. Cheung, and M. Denuit (2026). An Axiomatic Characterization of the Quantile Risk-Sharing Rule. Scandinavian Actuarial Journal 2026 (1), 1–20.

\item Dror, D. M., S. A. S. Hossain, A. Majumdar, T. L. Pérez Koehlmoos, D. John, and P. K. Panda (2016). What Factors Affect Voluntary Uptake of Community-Based Health Insurance Schemes in Low- and Middle-Income Countries? A Systematic Review and Meta-Analysis. PLOS ONE 11 (8), 1–31.

\item \textcolor{black}{Dufresne, F. and H. U. Gerber (1988). The Surpluses Immediately Before and at Ruin, and the Amount of the Claim Causing Ruin. Insurance: Mathematics and Economics 7 (3), 193–199.}

\item \textcolor{black}{Egídio dos Reis, A. (1993). How Long is the Surplus Below Zero? Insurance: Mathematics and Economics 12 (1), 23–38.}

\item Ekman, B. (2004). Community-Based Health Insurance in Low-Income Countries: A Systematic Review of the Evidence. Health Policy Plan 19 (5), 249–70.

\item Eze, P., S. Ilechukwu, and L. O. Lawani (2023). Impact of Community-Based Health Insurance in Low- and Middle-Income Countries: A Systematic Review and Meta-Analysis. PLOS ONE 18 (6), 1–47.

\item Fadlallah, R., F. El-Jardali, N. Hemadi, R. Z. Morsi, C. A. Abou Samra, A. Ahmad, K. Arif, L. Hishi, G. Honein-AbouHaidar, and E. A. Akl (2018). Barriers and Facilitators to Implementation, Uptake and Sustainability of Community-Based Health Insurance Schemes in Low- and Middle-Income Countries: A Systematic Review. International Journal for Equity in Health 17 (1), 13.

\item Feng, R. (2023). Decentralized Insurance: Technical Foundation of Business Models. Springer.	

\item Fitzsimons, E., B. Malde, and M. Vera-Hernández (2018). Group Size and the Efficiency of Informal Risk Sharing. The Economic Journal 128 (612), F575–F608.

\item Flores-Contró, J. M., K. Henshaw, S.-H. Loke, S. Arnold, and C. Constantinescu (2025). Subsidizing Inclusive Insurance to Reduce Poverty. North American Actuarial Journal 29 (1), 44–73.

\item \textcolor{black}{Gerber, H. U., M. J. Goovaerts, and R. Kaas (1987). On the Probability and Severity of Ruin. ASTIN Bulletin 17 (2), 151–163.}

\item \textcolor{black}{Gerber, H. U. and E. S. Shiu (1998). On the Time Value of Ruin. North American Actuarial Journal 2 (1), 48–72.}

\item \textcolor{black}{Gu, J.-W., M. Steffensen, and H. Zheng (2018). Optimal Dividend Strategies of Two Collaborating Businesses in the Diffusion Approximation Model. Mathematics of Operations Research 43 (2), 377–398.}

\item \textcolor{black}{Hanbali, H., H. Warnakulasooriya, and J. W. Y. Leung (2025). Time-Varying Pareto Optimal Risk Sharing for Annuities. ASTIN Bulletin 55 (3), 695–720.}

%\item \textcolor{black}{Kreinin, A. (2016). Correlated Poisson Processes and Their Applications in Financial Modeling. In Financial Signal Processing and Machine Learning, Chapter 9, pp. 191–232. John Wiley \& Sons, Ltd.}

\item Lemay-Boucher, P. (2012). Insurance for the Poor: the Case of Informal Insurance Groups in Benin. The Journal of Development Studies 48 (9), 1258–1273.

\item Levantesi, S. and G. Piscopo (2022). Mutual Peer-to-Peer Insurance: The Allocation of Risk. Journal of Co-operative Organization and Management 10 (1), 100154.

\item \textcolor{black}{Loisel, S. (2005). Differentiation of Some Functionals of Risk Processes, and Optimal Reserve Allocation. Journal of Applied Probability 42 (2), 379–392.}

\item Mandjes, M. and O. Boxma (2023). The Cramér–Lundberg Model and Its Variants: A Queueing Perspective. Cham, Switzerland: Springer Nature Switzerland AG.

\item Mebratie, A. D., R. Sparrow, G. Alemu, and A. S. Bedi (2013). Community-Based Health Insurance Schemes: A Systematic Review. Working Paper.

\item Parmar, D., A. Souares, M. de Allegri, G. Savadogo, and R. Sauerborn (2012). Adverse Selection in a Community-Based Health Insurance Scheme in Rural Africa: Implications for Introducing Targeted Subsidies. BMC Health Services Research 12 (1), 181.

\item \textcolor{black}{Picard, P. (1994). On Some Measures of the Severity of Ruin in the Classical Poisson Model. Insurance: Mathematics and Economics 14 (2), 107–115.}

\item Rolski, T., H. Schmidli, V. Schmidt, and J. Teugels (1999). Stochastic Processes for Insurance and Finance. West Sussex, England: John Wiley \& Sons Ltd.

\item Saha, S. K. and J. Qin (2023). Financial Inclusion and Poverty Alleviation: An Empirical Examination. Economic Change and Restructuring 56 (1), 409–440.

\item Schumacher, J.M. (2018). Linear Versus Nonlinear Allocation Rules in Risk Sharing Under Financial Fairness. ASTIN Bulletin 48, 995-1024.

\item \textcolor{black}{Shaked, M. and J. G. Shanthikumar (2007). Stochastic Orders. Springer.}

%\item Winsemius, H. C., B. Jongman, T. I. Veldkamp, S. Hallegatte, M. Bangalore, and P. J. Ward (2018). Disaster Risk, Climate Change, and Poverty: Assessing the Global Exposure of Poor People to Floods and Droughts. Environment and Development Economics 23 (3), 328–348.

%\item Woolsey Biggart, N. (2001). Banking on Each Other Banking: The Situational Logic of Rotating Savings and Credit Association. Advances in Qualitative Organization Research 3, 129–153.

\item \textcolor{black}{Tao, C., Y. Shen, and T. K. Siu (2025). Pareto-Optimal Risk Exchange in a Continuous-Time Economy: Application to Target Benefit Pension. ASTIN Bulletin 55 (3), 615–643.}

\item Yang, J. and W. Wei (2025). On the Optimality of Linear Residual Risk Sharing. ASTIN Bulletin 55 (3), 514–536.

\item Zhang, L. and H. Wang (2008). Dynamic Process of Adverse Selection: Evidence from a Subsidized Community-Based Health Insurance in Rural China. Social Science \& Medicine 67 (7), 1173–1182.

\end{enumerate}

\newpage

% Appendices come here

\section*{\textcolor{black}{Appendices}}
\addcontentsline{toc}{section}{Appendices}

\renewcommand{\thesubsection}{\Alph{subsection}}
\renewcommand{\thesubsubsection}{\thesubsection.\arabic{subsubsection}}
\setcounter{subsection}{0}

\numberwithin{equation}{subsubsection}

% Number propositions by appendix subsection: A.1, A.2, ...
\makeatletter
\@addtoreset{proposition}{subsection}
\makeatother
\renewcommand{\theproposition}{\thesubsection.\arabic{proposition}}

{\color{black}
\subsection{Incorporating dependence through proportional common shocks} \label{sec:Incorporating dependencies}

The purpose of this appendix is to discuss a simple extension of the model in Section~\ref{sec:Loss model} in which the individual risk processes are no longer mutually independent. The dependence is introduced through a common-shock component affecting all members of the pool simultaneously. The model is deliberately kept close to the scale-family framework used in Proposition~\ref{ImpactofPooling-Section4-Proposition1}, in order to show that the ruin-reducing effect of proportional pooling can still be obtained under a specific form of dependence.

\subsubsection{Individual accounts with proportional common shocks}

Let $\{N_{i,t},t\geq 0\}$, $i=1,\ldots,n$, be the idiosyncratic claim-counting processes introduced in Section~\ref{sec:Loss model}. In addition, let $\{P_t,t\geq 0\}$, be a Poisson process with intensity $\lambda^{\mathrm{cs}}>0$, independent of the idiosyncratic counting processes. The process $\{P_t,t\geq 0\}$ records the occurrence times of systemic events affecting the whole community.

We assume that idiosyncratic losses belong to a common scale family. More precisely, there exists a non-negative random variable $W$ such that $\Esp[W]=1$ and
\begin{align}
Y_{i,k}\eqD b_i W_{i,k},
\qquad i=1,\ldots,n,\quad k\geq 1,
\label{eq:app-cs-idiosyncratic-scale}
\end{align}
where $W_{i,k}$ are independent copies of $W$.

For the systemic component, we assume that all participants are hit by the same underlying shock, but with participant-specific exposure coefficients. Thus, at the occurrence time of the $k$th systemic event, participant $i$ suffers the loss
\begin{align}
Y^{\mathrm{cs}}_{i,k}=b^{\mathrm{cs}}_i W^{\mathrm{cs}}_k,
\qquad i=1,\ldots,n,\quad k\geq 1,
\label{eq:app-cs-proportional-systemic-loss}
\end{align}
where $b^{\mathrm{cs}}_i>0$ is the systemic exposure of participant $i$, and where $W^{\mathrm{cs}}_1,W^{\mathrm{cs}}_2,\ldots$ are independent copies of a non-negative random variable $W^{\mathrm{cs}}$ satisfying $\Esp[W^{\mathrm{cs}}]=1$. In the most direct scale-family extension of Proposition~\ref{ImpactofPooling-Section4-Proposition1}, one may take
\[
    W^{\mathrm{cs}}\eqD W.
\]
This means that idiosyncratic and systemic claim severities have the same normalized shape, while their magnitudes differ through scale parameters only.

Under this common-shock model, the cumulative claim amount reported by participant $i$ up to time $t$ is
\begin{align}
S^{\mathrm{cs}}_{i,t}
=
\sum_{k=1}^{N_{i,t}}Y_{i,k}
+
\sum_{k=1}^{P_t}Y^{\mathrm{cs}}_{i,k},
\qquad t\geq 0.
\label{eq:app-cs-standalone-loss}
\end{align}
The corresponding stand-alone surplus process is
\begin{align}
V^{\mathrm{cs}}_{i,t}
=
c^{\mathrm{cs}}_i t-S^{\mathrm{cs}}_{i,t}+\kappa_i,
\qquad t\geq 0,
\label{eq:app-cs-standalone-surplus}
\end{align}
where
\begin{align}
c^{\mathrm{cs}}_i
=
(1+\eta)\Esp[S^{\mathrm{cs}}_{i,1}]
=
(1+\eta)\left(\lambda_i b_i+\lambda^{\mathrm{cs}}b^{\mathrm{cs}}_i\right),
\qquad \eta>0.
\label{eq:app-cs-premium}
\end{align}

The processes $\{S^{\mathrm{cs}}_{i,t},t\geq 0\}$ are dependent because they share the same systemic counting process $\{P_t,t\geq 0\}$ and the same systemic shock sizes $W^{\mathrm{cs}}_k$. For instance, if $N^{\mathrm{cs}}_{i,t}=N_{i,t}+P_t$ denotes the total number of claim occurrences affecting participant $i$, then
\[
    \operatorname{Corr}\left[N^{\mathrm{cs}}_{i,t},N^{\mathrm{cs}}_{j,t}\right]
    =
    \frac{\lambda^{\mathrm{cs}}}
    {\sqrt{\lambda_i+\lambda^{\mathrm{cs}}}\sqrt{\lambda_j+\lambda^{\mathrm{cs}}}},
    \qquad i\neq j.
\]
This is a standard common-shock source of dependence across participants.

\subsubsection{Pooled fund and occurrence-time representation}

Let

\[
    \lambda_\bullet=\sum_{j=1}^n\lambda_j,
    \qquad
    \Lambda=\lambda_\bullet+\lambda^{\mathrm{cs}}.
\]
The total occurrence process at the pool level is
\[
    N^{\mathrm{cs}}_t
    =
    \sum_{j=1}^n N_{j,t}+P_t,
    \qquad t\geq 0,
\]
which is a Poisson process with intensity $\Lambda$. Let

\[
    T^{\mathrm{cs}}_k=\inf\{t\geq 0:N^{\mathrm{cs}}_t=k\}
\]
be the $k$th occurrence time at the pool level.

At each occurrence time, the triggering event is either an idiosyncratic claim from one participant or a systemic event affecting all participants. Let

\[
    I^{\mathrm{cs}}_{0,k}, I^{\mathrm{cs}}_{1,k},\ldots,I^{\mathrm{cs}}_{n,k}
\]
be the corresponding indicators, where $I^{\mathrm{cs}}_{0,k}=1$ means that the $k$th occurrence is systemic, while $I^{\mathrm{cs}}_{j,k}=1$ means that it is an idiosyncratic claim generated by participant $j$. Thus, exactly one of these indicators is equal to one, and

\[
    \Prob[I^{\mathrm{cs}}_{j,k}=1]=\frac{\lambda_j}{\Lambda},
    \qquad j=1,\ldots,n,
    \qquad
    \Prob[I^{\mathrm{cs}}_{0,k}=1]=\frac{\lambda^{\mathrm{cs}}}{\Lambda}.
\]

The aggregate claim amount at the $k$th occurrence time is then
\begin{align}
X^{\mathrm{cs}}_k
=
\sum_{j=1}^n I^{\mathrm{cs}}_{j,k}Y_{j,N_{j,T^{\mathrm{cs}}_k}}
+
I^{\mathrm{cs}}_{0,k}
\sum_{j=1}^n b^{\mathrm{cs}}_j W^{\mathrm{cs}}_{P_{T^{\mathrm{cs}}_k}}.
\label{eq:app-cs-aggregate-claim}
\end{align}
Consequently,

\[
    S^{\mathrm{cs}}_t
    =
    \sum_{j=1}^n S^{\mathrm{cs}}_{j,t}
    \overset{\mathrm{d}}{=}
    \sum_{k=1}^{N^{\mathrm{cs}}_t}X^{\mathrm{cs}}_k,
    \qquad t\geq 0.
\]

The aggregate surplus process of the pooled fund is therefore
\begin{align}
V^{\mathrm{cs}}_t
=
c^{\mathrm{cs}}t-S^{\mathrm{cs}}_t+\kappa,
\qquad
c^{\mathrm{cs}}=\sum_{i=1}^n c^{\mathrm{cs}}_i,
\qquad
\kappa=\sum_{i=1}^n\kappa_i.
\label{eq:app-cs-aggregate-surplus}
\end{align}

\subsubsection{Proportional sharing under common shocks}

Let $\Avec=(a_{i,j})_{1\leq i,j\leq n}$ be the allocation matrix introduced in Section~\ref{sec:Linear risk sharing at occurrence time}. If the $k$th occurrence is an idiosyncratic claim generated by participant $j$, then participant $i$ pays the amount $  a_{i,j}Y_{j,N_{j,T^{\mathrm{cs}}_k}}$.

If the $k$th occurrence is systemic, the original systemic loss profile is
\[
    \left(b^{\mathrm{cs}}_1,\ldots,b^{\mathrm{cs}}_n\right)W^{\mathrm{cs}}_{P_{T^{\mathrm{cs}}_k}},
\]
and the proportional sharing rule allocates to participant $i$ the amount
\[
    \sum_{j=1}^n a_{i,j}b^{\mathrm{cs}}_j W^{\mathrm{cs}}_{P_{T^{\mathrm{cs}}_k}}.
\]
Thus, the payment borne by participant $i$ at the $k$th occurrence time is
\begin{align}
Z^{\mathrm{cs}}_{i,k}
=
\sum_{j=1}^n
I^{\mathrm{cs}}_{j,k}a_{i,j}Y_{j,N_{j,T^{\mathrm{cs}}_k}}
+
I^{\mathrm{cs}}_{0,k}
\left(
\sum_{j=1}^n a_{i,j}b^{\mathrm{cs}}_j
\right)
W^{\mathrm{cs}}_{P_{T^{\mathrm{cs}}_k}}.
\label{eq:app-cs-pooled-payment}
\end{align}
The cumulative claim amount allocated to participant $i$ after pooling is
\[
    S^{\mathrm{cs-pool}}_{i,t}
    =
    \sum_{k=1}^{N^{\mathrm{cs}}_t}Z^{\mathrm{cs}}_{i,k},
    \qquad t\geq 0,
\]
and the corresponding surplus process is
\begin{align}
V^{\mathrm{cs-pool}}_{i,t}
=
c^{\mathrm{cs}}_i t-S^{\mathrm{cs-pool}}_{i,t}+\kappa_i,
\qquad t\geq 0.
\label{eq:app-cs-pooled-surplus}
\end{align}

We impose the following conditions on the allocation matrix. First, the full-allocation condition is
\begin{align}
\sum_{i=1}^n a_{i,j}=1,
\qquad j=1,\ldots,n.
\label{eq:app-cs-full-allocation}
\end{align}
Second, the idiosyncratic part is actuarially fair:
\begin{align}
\sum_{j=1}^n \lambda_j a_{i,j}b_j
=
\lambda_i b_i,
\qquad i=1,\ldots,n.
\label{eq:app-cs-idiosyncratic-fairness}
\end{align}
Third, the usual capacity constraint is imposed on idiosyncratic transfers:
\begin{align}
a_{i,j}b_j\leq b_i,
\qquad i,j=1,\ldots,n.
\label{eq:app-cs-capacity}
\end{align}
Finally, we assume that the allocation matrix preserves each participant's systemic exposure:
\begin{align}
\sum_{j=1}^n a_{i,j}b^{\mathrm{cs}}_j
=
b^{\mathrm{cs}}_i,
\qquad i=1,\ldots,n.
\label{eq:app-cs-systemic-preservation}
\end{align}
Equivalently, with $b^{\mathrm{cs}}=(b^{\mathrm{cs}}_1,\ldots,b^{\mathrm{cs}}_n)^\top$, condition~\eqref{eq:app-cs-systemic-preservation} reads
\[
    \Avec b^{\mathrm{cs}}=b^{\mathrm{cs}}.
\]
This condition means that the proportional sharing rule redistributes idiosyncratic losses, while it does not alter the systemic exposure borne by each participant. Under \eqref{eq:app-cs-idiosyncratic-fairness} and \eqref{eq:app-cs-systemic-preservation}, the expected loss rate borne by participant $i$ is unchanged by pooling:
\[
    \sum_{j=1}^n \lambda_j a_{i,j}b_j
    +
    \lambda^{\mathrm{cs}}
    \sum_{j=1}^n a_{i,j}b^{\mathrm{cs}}_j
    =
    \lambda_i b_i+\lambda^{\mathrm{cs}}b^{\mathrm{cs}}_i.
\]

\subsubsection{Ruin comparison}

Define the infinite-time ruin probabilities
\[
    \psi_i^{\mathrm{cs}}(\kappa_i)
    =
    \Prob\left[
        \inf_{t\geq 0}V^{\mathrm{cs}}_{i,t}<0
        \mid V^{\mathrm{cs}}_{i,0}=\kappa_i
    \right],
\]
and
\[
    \psi_i^{\mathrm{cs-pool}}(\kappa_i)
    =
    \Prob\left[
        \inf_{t\geq 0}V^{\mathrm{cs-pool}}_{i,t}<0
        \mid V^{\mathrm{cs-pool}}_{i,0}=\kappa_i
    \right].
\]
The following result shows the effect of the systematic risk on ruin probabilities.

\begin{proposition}
\label{prop:app-cs-ruin-reduction}
Assume that \eqref{eq:app-cs-idiosyncratic-scale}--\eqref{eq:app-cs-proportional-systemic-loss} hold. Suppose moreover that the allocation matrix $\Avec$ satisfies the full-allocation condition \eqref{eq:app-cs-full-allocation}, the idiosyncratic actuarial-fairness condition \eqref{eq:app-cs-idiosyncratic-fairness}, the capacity constraint \eqref{eq:app-cs-capacity}, and the systemic-exposure preservation condition \eqref{eq:app-cs-systemic-preservation}. Then, for every participant $i=1,\ldots,n$ and every initial reserve $\kappa_i\geq 0$,
\[
    \psi_i^{\mathrm{cs-pool}}(\kappa_i)
    \leq
    \psi_i^{\mathrm{cs}}(\kappa_i).
\]

\end{proposition}

\begin{proof}
Fix $i\in{1,\ldots,n}$. We represent both the stand-alone and pooled models on the same occurrence process ${N^{\mathrm{cs}}_t,t\geq 0}$, with intensity $\Lambda=\lambda_\bullet+\lambda^{\mathrm{cs}}$.

For the stand-alone model, define the claim amount borne by participant $i$ at a generic occurrence time as
\[
    \widetilde Y^{\mathrm{cs}}_i
    =
    \begin{cases}
        b_i W, & \text{with probability } \lambda_i/\Lambda,\\
        b^{\mathrm{cs}}_i W^{\mathrm{cs}}, & \text{with probability } \lambda^{\mathrm{cs}}/\Lambda,\\
        0, & \text{with probability } (\lambda_\bullet-\lambda_i)/\Lambda.
    \end{cases}
\]
Then,
\[
    S^{\mathrm{cs}}_{i,t}
    \overset{\mathrm{d}}{=}
    \sum_{k=1}^{N^{\mathrm{cs}}_t}\widetilde Y^{\mathrm{cs}}_{i,k},
\]
where $\widetilde Y^{\mathrm{cs}}_{i,1},\widetilde Y^{\mathrm{cs}}_{i,2},\ldots$ are independent copies of $\widetilde Y^{\mathrm{cs}}_i$.

For the pooled model, the claim amount borne by participant $i$ at a generic occurrence time is distributed as
\[
    Z^{\mathrm{cs}}_i
    =
    \begin{cases}
        a_{i,j}b_j W, & \text{with probability } \lambda_j/\Lambda,\quad j=1,\ldots,n,\\
        \left(\sum_{j=1}^n a_{i,j}b^{\mathrm{cs}}_j\right) W^{\mathrm{cs}},
        & \text{with probability } \lambda^{\mathrm{cs}}/\Lambda.
    \end{cases}
\]
Using \eqref{eq:app-cs-systemic-preservation}, this becomes
\[
    Z^{\mathrm{cs}}_i
    =
    \begin{cases}
        a_{i,j}b_j W, & \text{with probability } \lambda_j/\Lambda,\quad j=1,\ldots,n,\\
        b^{\mathrm{cs}}_i W^{\mathrm{cs}},
        & \text{with probability } \lambda^{\mathrm{cs}}/\Lambda.
    \end{cases}
\]
Moreover, by \eqref{eq:app-cs-idiosyncratic-fairness},
\[
    \Esp[Z^{\mathrm{cs}}_i]
    =
    \frac{1}{\Lambda}
    \sum_{j=1}^n\lambda_j a_{i,j}b_j
    +
    \frac{\lambda^{\mathrm{cs}}}{\Lambda}b^{\mathrm{cs}}_i
    =
    \frac{\lambda_i b_i+\lambda^{\mathrm{cs}}b^{\mathrm{cs}}_i}{\Lambda}
    =
    \Esp[\widetilde Y^{\mathrm{cs}}_i].
\]

We now prove that $    Z^{\mathrm{cs}}_i  \lcx  \widetilde Y^{\mathrm{cs}}_i$.
Let $g$ be any convex function for which the expectations below exist. Using the capacity condition \eqref{eq:app-cs-capacity}, we have, for every $j=1,\ldots,n$,

\[
    0\leq \frac{a_{i,j}b_j}{b_i}\leq 1.
\]
Hence, by convexity of $g$,
\[
    \Esp[g(a_{i,j}b_j W)]
    \leq
    \left(1-\frac{a_{i,j}b_j}{b_i}\right)g(0)
    +
    \frac{a_{i,j}b_j}{b_i}
    \Esp[g(b_i W)].
\]
Therefore,
\begin{align*}
\Esp\!\left[g\!\left(Z^{\mathrm{cs}}_{i}\right)\right]
&=
\sum_{j=1}^{n}
\frac{\lambda_{j}}{\Lambda}
\Esp\!\left[
g\!\left(a_{i,j}b_{j}W\right)
\right]
+
\frac{\lambda^{\mathrm{cs}}}{\Lambda}
\Esp\!\left[
g\!\left(b^{\mathrm{cs}}_{i}W^{\mathrm{cs}}\right)
\right]
\\
&\leq
g(0)
\sum_{j=1}^{n}
\frac{\lambda_{j}}{\Lambda}
\left(
1-\frac{a_{i,j}b_{j}}{b_{i}}
\right)
\\
&\quad
+
\Esp\!\left[g\!\left(b_{i}W\right)\right]
\sum_{j=1}^{n}
\frac{\lambda_{j}}{\Lambda}
\frac{a_{i,j}b_{j}}{b_{i}}
+
\frac{\lambda^{\mathrm{cs}}}{\Lambda}
\Esp\!\left[
g\!\left(b^{\mathrm{cs}}_{i}W^{\mathrm{cs}}\right)
\right].
\end{align*}
By \eqref{eq:app-cs-idiosyncratic-fairness},
\[
    \sum_{j=1}^n
    \frac{\lambda_j}{\Lambda}
    \frac{a_{i,j}b_j}{b_i}
    =
    \frac{\lambda_i}{\Lambda}.
\]
It follows that
\begin{align*}
\Esp[g(Z^{\mathrm{cs}}_i)]
&\leq
\frac{\lambda_\bullet-\lambda_i}{\Lambda}g(0)
+
\frac{\lambda_i}{\Lambda}\Esp[g(b_i W)]
+
\frac{\lambda^{\mathrm{cs}}}{\Lambda}
\Esp[g(b^{\mathrm{cs}}_i W^{\mathrm{cs}})]\\
&=
\Esp[g(\widetilde Y^{\mathrm{cs}}_i)].
\end{align*}
Since the two random variables have the same mean, this proves
that $ Z^{\mathrm{cs}}_i     \lcx  \widetilde Y^{\mathrm{cs}}_i$ holds true.

The stand-alone and pooled surplus processes of participant $i$ have the same premium rate $c^{\mathrm{cs}}_i$, the same claim-arrival intensity $\Lambda$, and claim severities ordered in the convex order. The standard ruin-probability comparison theorem for the Cramér--Lundberg model therefore gives
\[
    \psi_i^{\mathrm{cs-pool}}(\kappa_i)
    \leq
    \psi_i^{\mathrm{cs}}(\kappa_i),
    \qquad \kappa_i\geq 0.
\]
This completes the proof.
\end{proof}

\begin{remark}
The assumption $\Avec b^{\mathrm{cs}}=b^{\mathrm{cs}}$ is essential in the above argument. It prevents the sharing rule from increasing the systemic exposure borne by some participants. If this condition is relaxed, the pooled systemic payment of participant $i$ becomes

\[
    \left(\sum_{j=1}^n a_{i,j}b^{\mathrm{cs}}_j\right)W^{\mathrm{cs}},
\]
which need not be smaller than $b^{\mathrm{cs}}_i W^{\mathrm{cs}}$ in convex order. In that case, equality of total expected losses alone is not sufficient to guarantee a reduction of individual ruin probabilities.
\end{remark}

\begin{remark}
A simple case in which $\Avec b^{\mathrm{cs}}=b^{\mathrm{cs}}$ is automatically satisfied is obtained when the systemic exposure vector $b^{\mathrm{cs}}$ is proportional to the idiosyncratic expected-loss vector
\[
    m=(\lambda_1 b_1,\ldots,\lambda_n b_n)^\top.
\]
Indeed, if $b^{\mathrm{cs}}=\rho m$ for some $\rho>0$, then the idiosyncratic actuarial-fairness condition $\Avec m=m$ implies $\Avec b^{\mathrm{cs}}=b^{\mathrm{cs}}$. In particular, when claim frequencies are homogeneous, this condition reduces to the more transparent requirement that systemic exposures are proportional to the mean idiosyncratic claim severities.
\end{remark}

}

\subsection{Mathematical Proofs} \label{sec:Mathematical Proofs}

\subsubsection{Proof of Proposition \ref{ImpactofPooling-Section4-Proposition1}} \label{sec:Proof of Proposition 3.8}

\begin{proof}
Fix $i\in\{1,\ldots,n\}$. We first rewrite the stand-alone loss process of participant $i$
using the aggregate occurrence process $\{N_t,\hspace{2mm}t\geq 0\}$ introduced in Section~\ref{sec:Loss model}.
Let $Y_{i,1}^\prime,Y_{i,2}^\prime,\ldots$ be independent and identically distributed random variables such that
\begin{align}
    \Prob\!\left[Y_{i,1}^\prime=0\right]
    =
    1-\frac{\lambda_i}{\lambda_\bullet},
    \qquad
    \Prob\!\left[Y_{i,1}^\prime>y\right]
    =
    \frac{\lambda_i}{\lambda_\bullet}\bigl(1-F_{Y_i}(y)\bigr),
    \quad y\ge 0.
    \label{ImpactofPooling-Section4-Equation4}
\end{align}
Define
\begin{align}
    S_{i,t}^\prime=\sum_{k=1}^{N_t} Y_{i,k}^\prime,
    \qquad t\ge 0.
    \label{ImpactofPooling-Section4-Equation5}
\end{align}
By the classical thinning property of independent Poisson processes,
$ \{S_{i,t},\,t\ge 0\}\eqD \{S_{i,t}^\prime,\,t\ge 0\}$.
Hence, the stand-alone and pooled models can be represented on the same aggregate arrival
process $\{N_t,\hspace{2mm}t\geq 0\}$.

Under Assumption \ref{LinearRiskSharing-Section3-AssumptionSF}, we have
$Y_{i,k}\eqD b_i W$, where $W\ge 0$ and $\Esp[W]=1$. Therefore,
\[
    Y_{i,1}^\prime
    \eqD
    \begin{cases}
        0, & \text{with probability } 1-\lambda_i/\lambda_\bullet,\\[1mm]
        b_i W, & \text{with probability } \lambda_i/\lambda_\bullet.
    \end{cases}
\]
In particular, $ \Esp[Y_{i,1}^\prime]=\frac{\lambda_i}{\lambda_\bullet}b_i$.

On the pooled side, the payment borne by participant $i$ at the $k$th occurrence time is
$ Z_{i,k}=\sum_{j=1}^n Z_{i,j,k}$,
and, conditionally on the event that the triggering claim comes from participant $j$, one has
\[
    Z_{i,k}\eqD a_{i,j}Y_{j,1}.
\]
Using again Assumption \ref{LinearRiskSharing-Section3-AssumptionSF}, it follows that
\[
    Z_{i,k}\eqD a_{i,J_k} b_{J_k} W,
\]
where $J_k$ defined in \eqref{DefJk} is the index of the participant generating the $k$th claim.
Hence, $\Esp[Z_{i,k}]  =   \sum_{j=1}^n \frac{\lambda_j}{\lambda_\bullet} a_{i,j} b_j$.
By Assumption \ref{LinearRiskSharing-Section3-AssumptionAF}, we get from
\eqref{PropertiesoftheLinearRisk-SharingRule-Section3-Equation5} that
\[
    \Esp[Z_{i,k}]
    =
    \frac{\lambda_i}{\lambda_\bullet}b_i
    =
    \Esp[Y_{i,1}^\prime].
\]

We now prove that $  Z_{i,k}\lcx Y_{i,k}^\prime$.
Let $g$ be any convex function such that the expectations below exist. Then
\begin{align*}
    \Esp[g(Z_{i,k})]
    &=
    \sum_{j=1}^n \frac{\lambda_j}{\lambda_\bullet}
    \Esp\!\left[g(a_{i,j}b_jW)\right].
\end{align*}
By Assumption \ref{LinearRiskSharing-Section3-AssumptionCC}, we have
\[
    0\le \frac{a_{i,j}b_j}{b_i}\le 1,
    \qquad j=1,\ldots,n.
\]
Therefore, for each $j$,
\[
    a_{i,j}b_jW
    =
    \left(1-\frac{a_{i,j}b_j}{b_i}\right)\! \cdot 0
    +
    \frac{a_{i,j}b_j}{b_i}\,(b_iW),
\]
and convexity of $g$ yields
\begin{align*}
    \Esp[g(a_{i,j}b_jW)]
    &\le
    \left(1-\frac{a_{i,j}b_j}{b_i}\right) g(0)
    +
    \frac{a_{i,j}b_j}{b_i}\, \Esp[g(b_iW)].
\end{align*}
Summing over $j$ gives
\begin{align*}
    \Esp[g(Z_{i,k})]
    &\le
    \sum_{j=1}^n \frac{\lambda_j}{\lambda_\bullet}
    \left(
        \left(1-\frac{a_{i,j}b_j}{b_i}\right) g(0)
        +
        \frac{a_{i,j}b_j}{b_i}\,\Esp[g(b_iW)]
    \right)
    \\
    &=
    g(0)\sum_{j=1}^n \frac{\lambda_j}{\lambda_\bullet}
    \left(1-\frac{a_{i,j}b_j}{b_i}\right)
    +
    \Esp[g(b_iW)]
    \sum_{j=1}^n \frac{\lambda_j}{\lambda_\bullet}\frac{a_{i,j}b_j}{b_i}.
\end{align*}
Using \eqref{PropertiesoftheLinearRisk-SharingRule-Section3-Equation5}, we obtain
\begin{align*}
    \sum_{j=1}^n \frac{\lambda_j}{\lambda_\bullet}\frac{a_{i,j}b_j}{b_i}
    =
    \frac{1}{\lambda_\bullet b_i}\sum_{j=1}^n \lambda_j a_{i,j}b_j
    =
    \frac{\lambda_i}{\lambda_\bullet}.
\end{align*}
Consequently,
$$
    \Esp[g(Z_{i,k})]\le
    \left(1-\frac{\lambda_i}{\lambda_\bullet}\right) g(0)
    +
    \frac{\lambda_i}{\lambda_\bullet}\Esp[g(b_iW)]
    =
    \Esp[g(Y_{i,1}^\prime)].
$$
Since we already proved that $\Esp[Z_{i,k}]=\Esp[Y_{i,1}^\prime]$, this shows that
$ Z_{i,k}\lcx Y_{i,1}^\prime$.

Now consider the surplus process
\[
    \widetilde V_{i,t}=c_it-S_{i,t}^\prime+\kappa_i,
    \qquad t\ge 0.
\]
Because $\{S_{i,t},\,t\ge 0\}\eqD\{S_{i,t}^\prime,\,t\ge 0\}$, the ruin probability associated with
$\widetilde V_{i,t}$ is exactly $\psi_i(\kappa_i)$. Considering $V_{i,t}^{\mathrm{pool}}$, 
both models have the same premium rate $c_i$, the same claim arrival intensity
$\lambda_\bullet$, and claim severities ordered by convex order. Therefore, Proposition~8.2
in Asmussen and Albrecher (2010) implies that the inequality $ \psi_i^{\mathrm{pool}}(\kappa_i)\le \psi_i(\kappa_i)$
holds true, as announced. This ends the proof.
\end{proof}

\begin{remark}
Notice that the proof also yields a \textcolor{black}{pointwise} convex-order comparison at claim occurrence times. Indeed,
by closure of the convex order under convolution, for every $k\ge 1$,
\[
    \sum_{\ell=1}^k Z_{i,\ell} \lcx \sum_{\ell=1}^k Y_{i,\ell}^\prime.
\]
Thus, after any fixed number of aggregate claim occurrences, the total amount borne by
participant $i$ inside the pool is less dispersed, in the convex-order sense, than the
corresponding stand-alone amount.
\end{remark}

\subsubsection{Proof of Proposition \ref{BeyondScaleFamilies-Section4-Proposition1}} \label{sec:Proof of Proposition 4.1}

\begin{proof}
Fix two distinct participants $i\neq j$ such that $b_i\le b_j$, and set
$\alpha:=\frac{b_i}{b_j}\in(0,1]$.
Consider the matrix $\Avec^{(i\leftarrow j)}$ defined by
\begin{equation}
    a_{i,j}=\alpha,\qquad
    a_{j,j}=1-\alpha,\qquad
    a_{j,i}=1,\qquad
    a_{i,i}=0,\qquad a_{\ell,\ell}=1 \quad \text{for } \ell\notin\{i,j\},
    \label{BeyondScaleFamilies-Section4-Equation7}
\end{equation}
and all other entries equal to $0$.

Let us first check that $\Avec^{(i\leftarrow j)}$ is admissible. Full allocation is immediate.
Actuarial fairness also holds. For row $i$, one has
\[
\sum_{k=1}^n a_{i,k}b_k = a_{i,j}b_j=\alpha b_j=b_i.
\]
For row $j$,
\[
\sum_{k=1}^n a_{j,k}b_k = a_{j,i}b_i+a_{j,j}b_j
= b_i+(1-\alpha)b_j
= b_i+b_j-b_i=b_j.
\]
For every $\ell\notin\{i,j\}$, fairness is trivial since $a_{\ell,\ell}=1$.
Because claim frequencies are homogeneous, admissibility also implies the capacity constraint.

Under the sharing rule $\Avec^{(i\leftarrow j)}$, participant $i$ pays
\[
Z_i=
\begin{cases}
\alpha Y_j, & \text{if } J=j,\\
0, & \text{otherwise}.
\end{cases}
\]
Thus,
\begin{equation}
    Z_i \eqD \widetilde B\,\alpha Y_j,
    \label{BeyondScaleFamilies-Section4-Equation8}
\end{equation}
where $\widetilde B\sim \mathrm{Bernoulli}(1/n)$ and $\widetilde B$ is independent of $Y_j$.
On the other hand,
\[
Y_i' \eqD B_iY_i,
\qquad B_i\sim \mathrm{Bernoulli}(1/n),
\]
with $B_i$ independent of $Y_i$.

By assumption, \eqref{BeyondScaleFamilies-Section4-Equation4} holds for this admissible
matrix, so that
\[
\widetilde B\,\alpha Y_j \lcx B_iY_i.
\]
Now, for every convex function $\varphi$ for which the expectations exist,
\[
\Esp[\varphi(\widetilde B\,\alpha Y_j)]
=
\left(1-\frac1n\right)\varphi(0)
+\frac1n\,\Esp[\varphi(\alpha Y_j)],
\]
and similarly,
\[
\Esp[\varphi(B_iY_i)]
=
\left(1-\frac1n\right)\varphi(0)
+\frac1n\,\Esp[\varphi(Y_i)].
\]
Therefore,
\[
\widetilde B\,\alpha Y_j \lcx B_iY_i
\qquad\Leftrightarrow\qquad
\alpha Y_j \lcx Y_i.
\]
Since
\[
\Esp[\alpha Y_j]=\alpha b_j=b_i=\Esp[Y_i],
\]
we conclude that $\alpha Y_j \lcx Y_i$. Dividing both sides by $b_i$ yields
$\frac{Y_j}{b_j}\lcx\frac{Y_i}{b_i}$,
which proves \eqref{BeyondScaleFamilies-Section4-Equation5}. The chain
\eqref{BeyondScaleFamilies-Section4-Equation6} follows by ordering the participants according
to their means.
\end{proof}

\subsection{Examples illustrating the role of Assumption \ref{LinearRiskSharing-Section3-AssumptionSF}} \label{sec:Examples illustrating the role of AssumptionSF}

\subsubsection{An example with random and constant stand-alone losses for $n=2$}
\label{sec:counterexample}

We now exhibit an explicit example (for $n=2$) in which the stand-alone loss is random for one participant
while it is constant for the other. This example shows that even when capacity and actuarial fairness are satisfied and both convex-order constraints
\(
Z_1\lcx Y'_1
\)
and
\(
Z_2\lcx Y'_2
\)
are fulfilled, the severities need not form a common scale family.

Consider
\begin{equation}
Y_2 \equiv 2,
\text{ and }
Y_1 =
\begin{cases}
0, & \text{with probability } 1/2,\\
2, & \text{with probability } 1/2.
\end{cases}
\label{eq:Ychoice}
\end{equation}
Then, $b_2=\Esp[Y_2]=2$ and $b_1=\Esp[Y_1]=1$. Hence, the normalized severities differ:
\(
Y_2/b_2\equiv 1
\)
is constant whereas
\(
Y_1/b_1 = Y_1
\)
is non-degenerate, so there is no common scale family.

Define the sharing matrix
\begin{equation}
\Avec=\begin{pmatrix}
a_{1,1} & a_{1,2}\\
a_{2,1} & a_{2,2}
\end{pmatrix}
=
\begin{pmatrix}
0 & 1/2\\
1 & 1/2
\end{pmatrix}.
\label{eq:Amatrix}
\end{equation}
\textcolor{black}{This matrix satisfies the full-allocation, actuarial-fairness and capacity constraints used in Proposition~\ref{ImpactofPooling-Section4-Proposition1}, although the scale-family assumption itself is not fulfilled}. Indeed,
full allocation obviously holds (each column of $\Avec$ sums to $1$). Capacity is valid since
$a_{1,2}b_2=(1/2)\cdot2=1=b_1$ and $a_{2,1}b_1=1\cdot1=1\le2=b_2$.
Actuarial fairness is also satisfied because
\(
a_{1,1}b_1+a_{1,2}b_2=0\cdot1+(1/2)\cdot2=1=b_1
\)
and
\(
a_{2,1}b_1+a_{2,2}b_2=1\cdot1+(1/2)\cdot2=2=b_2.
\)

The reasoning proceeds in 3 steps:

\medskip
\noindent\textbf{Step 1: compute $Z_1,Z_2$.}
If $J=1$, then $(Z_1,Z_2)=(0,Y_1)$.
If $J=2$, then $(Z_1,Z_2)=((1/2)Y_2,(1/2)Y_2)=(1,1)$ since $Y_2\equiv2$.
Therefore,
\begin{equation}
Z_1 =
\begin{cases}
0, & \text{with probability } 1/2,\\
1, & \text{with probability } 1/2,
\end{cases}
\quad \text{ and } \quad
Z_2 =
\begin{cases}
Y_1, & \text{with probability } 1/2,\\
1, & \text{with probability } 1/2.
\end{cases}
\label{eq:Z12}
\end{equation}
It follows that
\begin{equation}
\Prob[Z_2=0]=1/4,\qquad \Prob[Z_2=1]=1/2,\qquad \text{ and } \qquad \Prob[Z_2=2]=1/4.
\label{eq:Z2law}
\end{equation}

\medskip
\noindent\textbf{Step 2: compute $Y'_1,Y'_2$.}
Here $Y'_i=B_iY_i$ with $B_i\sim\mathrm{Bernoulli}(1/2)$.
Thus
\begin{equation}
Y'_1 =
\begin{cases}
0, & \text{with probability } 3/4,\\
2, & \text{with probability } 1/4,
\end{cases}
\qquad
Y'_2 =
\begin{cases}
0, & \text{with probability } 1/2,\\
2, & \text{with probability } 1/2.
\end{cases}
\label{eq:Yprime12}
\end{equation}

\medskip
\noindent\textbf{Step 3: verify $Z_1\lcx Y'_1$ and $Z_2\lcx Y'_2$ (via stop-loss transforms).}
For integrable non-negative random variables, $X\lcx Y$ is equivalent to equality of means and the stop-loss inequalities
\(
\Esp[(X-t)_+]\le \Esp[(Y-t)_+]
\)
for all $t\ge0$.
We check this criterion.

\smallskip
\noindent\emph{(i) $Z_1\lcx Y'_1$.}
Both means are $1/2$.
For $t\in[0,1]$,
\(
\Esp[(Z_1-t)_+]=\tfrac12(1-t)
\),
and for $t\ge1$ it is $0$.
For $Y'_1$, for $t\in[0,2]$,
\(
\Esp[(Y'_1-t)_+]=\tfrac14(2-t)
\),
and for $t\ge2$ it is $0$.
One checks that
\(
\tfrac12(1-t)\le \tfrac14(2-t)
\)
for $t\in[0,1]$ and also $0\le \tfrac14(2-t)$ for $t\in[1,2]$.
Hence, $Z_1\lcx Y'_1$.

\smallskip
\noindent\emph{(ii) $Z_2\lcx Y'_2$.}
Both means equal $1$.
For $t\in[0,1]$, using \eqref{eq:Z2law},
\begin{align*}
\Esp[(Z_2-t)_+]
&=\frac12(1-t)+\frac14(2-t)=1-\frac34 t.
\end{align*}
For $t\in[1,2]$,
\(
\Esp[(Z_2-t)_+]=\frac14(2-t)
\),
and it is $0$ for $t\ge2$.
For $Y'_2$, for $t\in[0,2]$,
\(
\Esp[(Y'_2-t)_+]=\frac12(2-t)=1-\frac12 t
\),
and it is $0$ for $t\ge2$.
Thus, for $t\in[0,1]$,
\(
1-\frac34 t \le 1-\frac12 t
\),
and for $t\in[1,2]$,
\(
\frac14(2-t)\le \frac12(2-t)
\).
Hence, $Z_2\lcx Y'_2$.

\medskip
In conclusion, all constraints are fulfilled and both convex-order inequalities are satisfied, yet $(Y_1,Y_2)$ do not belong to a common scale family.

\subsubsection{An example with random stand-alone losses for $n=2$}
\label{sec:counterexample2}

We now provide an example in which the stand-alone losses are random for each participant. Under this setting, as in the previous case, even when capacity, actuarial fairness and both convex-order constraints
\(
Z_1\lcx Y'_1
\)
and
\(
Z_2\lcx Y'_2
\)
are fulfilled, the severities need not form a common scale family.

Consider the random variables
\begin{equation}
Y_1 =
\begin{cases}
0, & \text{with probability } 1/3,\\
100, & \text{with probability } 1/3, \\
200, & \text{with probability } 1/3,
\end{cases}
\text{ \quad and \quad }
Y_2 =
\begin{cases}
0, & \text{with probability } 1/3,\\
150, & \text{with probability } 1/3, \\
400, & \text{with probability } 1/3.
\end{cases}
\label{eq:Ychoice}
\end{equation}
Hence, we have that $b_1=\Esp[Y_1]=100$ and $b_2=\Esp[Y_2]=\textcolor{black}{183.333}$. Thus, the normalized severities differ:
\(
Y_1/b_1 = Y_1/100
\)
whereas
\(
Y_2/b_2 = Y_2/\textcolor{black}{183.333}
\). Therefore, there is no common scale family.

Furthermore, let us define the sharing matrix
\begin{equation}
\Avec=\begin{pmatrix}
a_{1,1} & a_{1,2}\\
a_{2,1} & a_{2,2}
\end{pmatrix}
=
\begin{pmatrix}
1/2 & \textcolor{black}{0.273}\\
1/2 & \textcolor{black}{0.727}
\end{pmatrix},
\label{eq:Amatrix}
\end{equation}
which clearly fulfills \textcolor{black}{the full-allocation, actuarial-fairness and capacity constraints used in Proposition~\ref{ImpactofPooling-Section4-Proposition1}, although the scale-family assumption itself is not fulfilled}. That is, full allocation holds since each column of $\Avec$ sums to $1$, while the capacity constraints are satisfied because $a_{1,2}b_2=(\textcolor{black}{0.273})\cdot\textcolor{black}{183.333}=\textcolor{black}{50.05}<100=b_1$ and $a_{2,1}b_1=(1/2)\cdot100=50<\textcolor{black}{183.333}=b_2$. Moreover, actuarial fairness is also satisfied:
\(
a_{1,1}b_1+a_{1,2}b_2=(1/2)\cdot100+\textcolor{black}{0.273}\cdot\textcolor{black}{183.333}=100=b_1
\)
and
\(
a_{2,1}b_1+a_{2,2}b_2=(1/2)\cdot100+\textcolor{black}{0.727}\cdot\textcolor{black}{183.333}=\textcolor{black}{183.333}=b_2.
\)

Our reasoning follows that of the previous example:

\medskip
\noindent\textbf{Step 1: compute $Z_1,Z_2$.}
If $J=1$, then $(Z_1,Z_2)=((1/2)Y_1,(1/2)Y_1)$.
If $J=2$, then $(Z_1,Z_2)=(\textcolor{black}{0.273}Y_2,\textcolor{black}{0.727}Y_2)$.
Therefore,
\begin{equation}
Z_1 =
\begin{cases}
(1/2)Y_1, & \text{with probability } 1/2,\\
\textcolor{black}{0.273}Y_2, & \text{with probability } 1/2,
\end{cases}
\text{ and }
Z_2 =
\begin{cases}
(1/2)Y_1, & \text{with probability } 1/2,\\
\textcolor{black}{0.727}Y_2, & \text{with probability } 1/2.
\end{cases}
\label{eq:Z12}
\end{equation}
It follows that
\begin{align}
\Prob[Z_1=0]=1/3,\quad &\Prob[Z_1=50]=1/6,\quad \Prob[Z_1=100]=1/6,\quad 
\Prob[Z_1=\textcolor{black}{40.909}]=1/6, \nonumber\\
\Prob[Z_1=\textcolor{black}{109.091}]=1/6,\quad
&\Prob[Z_2=0]=1/3,\quad \Prob[Z_2=50]=1/6,\quad \Prob[Z_2=100]=1/6, \nonumber \\ 
&\Prob[Z_2=\textcolor{black}{109.091}]=1/6,\quad \text{ and } \quad
\Prob[Z_2=\textcolor{black}{290.909}]=1/6. \label{eq:Z2law2}
\end{align}

\medskip
\noindent\textbf{Step 2: compute $Y'_1,Y'_2$.}
Here $Y'_i=B_iY_i$ with $B_i\sim\mathrm{Bernoulli}(1/2)$.
Thus
\begin{equation}
Y'_1 =
\begin{cases}
0, & \text{with probability } 2/3,\\
100, & \text{with probability } 1/6,\\
200, & \text{with probability } 1/6.
\end{cases}
\qquad
Y'_2 =
\begin{cases}
0, & \text{with probability } 2/3,\\
150, & \text{with probability } 1/6,\\
400, & \text{with probability } 1/6.
\end{cases}
\label{eq:Yprime12}
\end{equation}

\medskip
\noindent\textbf{Step 3: verify $Z_1\lcx Y'_1$ and $Z_2\lcx Y'_2$ (via stop-loss transforms).}
We check the stop-loss criterion.

\smallskip
\noindent\emph{(i) $Z_1\lcx Y'_1$.}
Both means are $50$.
Using \eqref{eq:Z2law2}, it yields
\begin{align*}
\Esp[(Z_1-t)_+]=
\begin{cases}
\frac{1}{6}\left(\textcolor{black}{190.909} - 3t\right), & \text{for } 0 \leq t < \textcolor{black}{40.909},\\
\frac{1}{6}\left(150 - 2t\right), & \text{for } \textcolor{black}{40.909} \leq t < 50, \\
\frac{1}{6}\left(100 - t\right), & \text{for } 50 \leq t < 100,\\
0, & \text{for } 100 \leq t.
\end{cases}
\end{align*}
On the other hand, for $Y'_1$, it yields,
\begin{align*}
\Esp[(Y'_1-t)_+]=
\begin{cases}
\frac{1}{6}\left(300 - 2t\right), & \text{for } 0 \leq t < 100,\\
\frac{1}{6}\left(200 - t\right), & \text{for } 100 \leq t < 200, \\
0, & \text{for } 200 \leq t.
\end{cases}
\end{align*}
One checks that
\(
\Esp[(Z_1-t)_+] \leq \Esp[(Y'_1-t)_+]
\)
for $t\in[0, \infty)$.
Hence, $Z_1\lcx Y'_1$.

\smallskip
\noindent\emph{(ii) $Z_2\lcx Y'_2$.}
Both means equal $\textcolor{black}{91.667}$.
Using \eqref{eq:Z2law2}, it yields
\begin{align*}
\Esp[(Z_2-t)_+]=
\begin{cases}
\frac{1}{6}\left(550 - 4t\right), & \text{for } 0 \leq t < 50,\\
\frac{1}{6}\left(500 - 3t\right), & \text{for } 50 \leq t < 100, \\
\frac{1}{6}\left(400 - 2t\right), & \text{for } 100 \leq t < \textcolor{black}{109.091},\\
\frac{1}{6}\left(\textcolor{black}{290.909} - t\right), & \text{for } \textcolor{black}{109.091} \leq t < \textcolor{black}{290.909},\\
0, & \text{for } \textcolor{black}{290.909} \leq t.
\end{cases}
\end{align*}
For $Y'_2$, we have the following,
\begin{align*}
\Esp[(Y'_2-t)_+]=
\begin{cases}
\frac{1}{6}\left(550 - 2t\right), & \text{for } 0 \leq t < 150,\\
\frac{1}{6}\left(400 - t\right), & \text{for } 150 \leq t < 400, \\
0, & \text{for } 400 \leq t.
\end{cases}
\end{align*}
One checks that
\(
\Esp[(Z_2-t)_+] \leq \Esp[(Y'_2-t)_+]
\)
for $t\in[0, \infty)$. Hence, $Z_2\lcx Y'_2$.

\medskip
Thus, we conclude that all constraints are fulfilled and both convex-order inequalities are satisfied, yet $(Y_1,Y_2)$ do not belong to a common scale family.

%\medskip
%Requiring $Z_i\lcx Y'_i$ for all participants and all admissible sharing matrices yields the inverse ordering \eqref{eq:inverse_order} of normalized severities (hence the chain \eqref{eq:chain}).
%The constraint for agent $j$ further yields additive inequalities of the form \eqref{eq:additive_restriction}, which translate into moment and stop-loss restrictions.
%However, even with capacity and actuarial fairness, these conditions do not force a common scale family, as shown by the counterexamples worked out in this section.

{\color{black}
\subsubsection{An example with continuous random stand-alone losses for $n=2$}
\label{sec:counterexample3}

We now provide an example in which the stand-alone losses are continuous random variables for each participant. Under this setting, as in the two previous cases, even when capacity, actuarial fairness and both convex-order constraints
\(
Z_1\lcx Y'_1
\)
and
\(
Z_2\lcx Y'_2
\)
are fulfilled, the severities need not form a common scale family.

Let $Y_1$ and $Y_2$ be continuous random variables with probability density functions
\begin{equation}
f_{Y_1}(y) =
\begin{cases}
\frac{1}{100}, & \text{for }  0 \leq y \leq 100,\\
0, & \text{otherwise},
\end{cases}
\text{ \quad and \quad }
f_{Y_2}(y) =
\begin{cases}
\frac{2\left(100 - y\right)}{10000}, & \text{for } 0 \leq y \leq 100,\\
0, & \text{otherwise},
\end{cases}
\label{eq:Ychoice3}
\end{equation}
respectively. Specifically, let $Y_{1} \sim \text{Uniform}\left(0, 100\right)$ and let $Y_{2}$ follow a left-triangular distribution on $(0,100)$. Hence, we have that $b_1=\Esp[Y_1]=\frac{100}{2}=50$ and $b_2=\Esp[Y_2]= \frac{100}{3} = 33.333$. Thus, the normalized severities differ:
\(
Y_1/b_1 = Y_1/50
\)
whereas
\(
Y_2/b_2 = \frac{3}{100}Y_2
\). Therefore, there is no common scale family.

Furthermore, let us define the sharing matrix
\begin{equation}
\Avec=\begin{pmatrix}
a_{1,1} & a_{1,2}\\
a_{2,1} & a_{2,2}
\end{pmatrix}
=
\begin{pmatrix}
0.6 & 0.6\\
0.4 & 0.4
\end{pmatrix},
\label{eq:Amatrix3}
\end{equation}
which clearly fulfills the full-allocation, actuarial-fairness and capacity constraints used in Proposition~\ref{ImpactofPooling-Section4-Proposition1}, although the scale-family assumption itself is not fulfilled. That is, full allocation holds since each column of $\Avec$ sums to $1$, while the capacity constraints are satisfied because $a_{1,2}b_2=0.6\cdot \left(\frac{100}{3}\right)=20<50=b_1$ and $a_{2,1}b_1=0.4\cdot50=20<\frac{100}{3}= 33.333 =b_2$. Moreover, actuarial fairness is also satisfied:
\(
a_{1,1}b_1+a_{1,2}b_2=0.6\cdot50+0.6\cdot\left(\frac{100}{3}\right)=50=b_1
\)
and
\(
a_{2,1}b_1+a_{2,2}b_2=0.4\cdot50+0.4\cdot\left(\frac{100}{3}\right)=\frac{100}{3}= 33.333 =b_2.
\)

Our reasoning is similar to that of the previous examples:

\medskip
\noindent\textbf{Step 1: derive the survival functions of $Z_1,Z_2$.}

\begin{equation}
\overline{F}_{Z_1}(z) := \Prob[Z_1 > z] =
\begin{cases}
\frac{1}{2}\left(1 - \frac{z}{60} + \left(1 - \frac{z}{60}\right)^{2}\right), & \text{for }  0 \leq z \leq 60,\\
0, & \text{otherwise},
\end{cases}
\label{eq:Z1survival}
\end{equation}
\text{ \quad and \quad }
\begin{equation}
\overline{F}_{Z_2}(z) := \Prob[Z_2 > z] =
\begin{cases}
\frac{1}{2}\left(1 - \frac{z}{40} + \left(1 - \frac{z}{40}\right)^{2}\right), & \text{for }  0 \leq z \leq 40,\\
0, & \text{otherwise}.
\end{cases}
\label{eq:Z2survival}
\end{equation}

\medskip
\noindent\textbf{Step 2: derive the survival functions of $Y'_1,Y'_2$.}

\begin{equation}
\overline{F}_{Y'_1}(y) := \Prob[Y'_1 > y] =
\begin{cases}
\frac{1}{2}\left(1 - \frac{y}{100}\right), & \text{for }  0 \leq y \leq 100,\\
0, & \text{otherwise},
\end{cases}
\label{eq:Y1survival}
\end{equation}
\text{ \quad and \quad }
\begin{equation}
\overline{F}_{Y'_2}(y) := \Prob[Y'_2 > y] =
\begin{cases}
\frac{1}{2}\left(1 - \frac{y}{100}\right)^{2}, & \text{for }  0 \leq y \leq 100,\\
0, & \text{otherwise}.
\end{cases}
\label{eq:Y2survival}
\end{equation}

\medskip
\noindent\textbf{Step 3: verify $Z_1\lcx Y'_1$ and $Z_2\lcx Y'_2$ (via stop-loss transforms).} Recall that the stop-loss transform of an integrable non-negative random variable $X$ can be obtained as the integral of its survival function $\overline{F}_{X}(x) := \Prob[X>x]$ from $t$ to infinity: $\Esp\left[(X-t)_{+}\right]=\int^{\infty}_{t}\overline{F}_{X}(x)\der x$. We now check the stop-loss criterion.

\smallskip

\noindent\emph{(i) $Z_1\lcx Y'_1$.}
Both means are $25$. Using \eqref{eq:Z1survival} and \eqref{eq:Y1survival}, for $60 \leq t \leq 100$ it yields

\begin{equation}
\Esp\left[(Y'_1-t)_{+}\right] - \Esp\left[(Z_1-t)_{+}\right] = \int^{100}_{t} \left(\overline{F}_{Y'_1}(x) -  \overline{F}_{Z_1}(x)\right) \der x = \frac{\left(100 - t\right)^{2}}{400} \geq 0.
\end{equation}

Similarly, for $0 \leq t < 60$, it yields

\begin{equation}
\Esp\left[(Y'_1-t)_{+}\right] - \Esp\left[(Z_1-t)_{+}\right] = \int^{60}_{t} \left(\overline{F}_{Y'_1}(x) -  \overline{F}_{Z_1}(x)\right) \der x = \frac{t}{21600}\left(t^{2} - 216t + 10800\right) \geq 0.
\end{equation}

The last inequality follows from the fact that $\nu_{1}(t) := \frac{t}{21600} \geq 0$ for $0 \leq t < 60$. Moreover, $\varphi_{1}(t) := t^{2} - 216t + 10800$ is an upward-opening parabola with vertex at $t = \frac{216}{2} = 108$, which lies outside the interval $[0, 60)$. Therefore, $\varphi_{1}(t)$ is strictly decreasing on $0 \leq t <60$, and its minimum over this interval is attained at the right endpoint: $\varphi_{1}(60) = (60)^{2} - 216(60) + 10800 = 1440 > 0$. Hence, both $\nu_{1}(t)$ and $\varphi_{1}(t)$ are non-negative on $[0,60)$, implying that their product is also non-negative. Thus, we conclude $Z_1\lcx Y'_1$.

\smallskip
\noindent\emph{(ii) $Z_2\lcx Y'_2$.}
Both means are $16.667$. Using \eqref{eq:Z2survival} and \eqref{eq:Y2survival}, for $40 \leq t \leq 100$ it yields

\begin{equation}
\Esp\left[(Y'_2-t)_{+}\right] - \Esp\left[(Z_2-t)_{+}\right] = \int^{100}_{t} \left(\overline{F}_{Y'_2}(x) -  \overline{F}_{Z_2}(x)\right) \der x = \frac{50}{3}\left(1 - \frac{t}{100}\right)^{3} \geq 0.
\end{equation}

Then, for $0 \leq t < 40$, it yields

\begin{align}
\begin{split}
\Esp\left[(Y'_2-t)_{+}\right] - \Esp\left[(Z_2-t)_{+}\right] & = \int^{40}_{t} \left(\overline{F}_{Y'_2}(x) -  \overline{F}_{Z_2}(x)\right) \der x \\ & = \frac{t}{80000}\left(7t^{2} - 1100t + 40000\right) \geq 0.
\end{split}
\end{align}

Again, the last inequality follows from the fact that $\nu_{2}(t) := \frac{t}{80000} \geq 0$ for $0 \leq t < 40$. Furthermore, $\varphi_{2}(t) := 7t^{2} - 1100t + 40000$ is an upward-opening parabola with vertex at $t = \frac{1100}{14} = 78.571$, which lies outside the interval $[0, 40)$. Thus, $\varphi_{2}(t)$ is strictly decreasing on $0 \leq t < 40$, and its minimum over this interval is attained at the right endpoint: $\varphi_{2}(40) = 7(40)^{2} - 1100(40) + 40000 = 7200 > 0$. Hence, both $\nu_{2}(t)$ and $\varphi_{2}(t)$ are non-negative on $[0,40)$, implying that their product is also non-negative. Thus, we conclude $Z_2\lcx Y'_2$.

An alternative approach to verifying the stochastic order inequalities $Z_1\lcx Y'_1$ and $Z_2\lcx Y'_2$ is to apply the one-crossing condition for survival functions; see, for example, Theorem 3.A.44 in Shaked and Shanthikumar (2007). According to this result, if $X$ and $Y$ are two continuous random variables with equal means and survival functions $\overline{F}_{X}$ and $\overline{G}_{Y}$, respectively. Then, $X\lcx Y$ if the following condition holds: $S^{-}(\overline{F}_{X} - \overline{G}_{Y})=1$ and the sign sequence is $+,-$. Here, $S^{-}(h) = \sup S^{-}\left(h(x_1),h(x_2),...,h(x_m)\right)$ denotes the number of sign changes of the function $h$ defined on the subset $I$ of the real line, where $S^{-}(y_1,y_2,...,y_m)$ is the number of sign changes of the indicated sequence, zero terms being discarded and the supremum is extended over all sets $x_1 < x_2< \dots < x_m$ such that $x_{i} \in I$ and $m<\infty$. We now verify this criterion.

\smallskip
\noindent\emph{(i) $Z_1\lcx Y'_1$.} Both means are $25$. Using \eqref{eq:Z1survival} and \eqref{eq:Y1survival}, for $0 \leq x < 60$, we have the following

\begin{equation}
D_{1}(x) := \overline{F}_{Z_1}(x) - \overline{F}_{Y'_1}(x) = \frac{1}{2}\left(1 - \frac{x}{60}\right)^{2} - \frac{x}{300}.
\end{equation}

First, observe that $D_{1}(0) = \frac{1}{2} > 0$. Moreover,

\begin{equation}
D'_{1}(x) = -\frac{1}{60}\left(1 - \frac{x}{60}\right)-\frac{1}{300} < 0, \quad 0 \leq x < 60,
\end{equation}

implying that $D_{1}(x)$ is strictly decreasing on $[0, 60)$. Consequently, $D_{1}(x)$ admits a unique root at $x = 32.201$ within this interval. It follows that

\begin{equation}
D_{1}(x) \geq 0, \quad 0 \leq x \leq 32.201, \quad \text{and} \quad D_{1}(x) < 0, \quad 32.201 < x < 60.
\end{equation}

It therefore remains to show that $D_{1}(x) < 0$ for all $60 \leq x \leq 100$. This establishes that $D_{1}(x)$ changes sign exactly once over its support and the sign sequence is $+,-$, thereby satisfying the one-crossing condition for survival functions. For $60 \leq x \leq 100$, we obtain

\begin{equation}
D_{1}(x) = -\frac{1}{2}\left(1 - \frac{x}{100}\right) \leq 0, \quad 60 \leq x \leq 100.
\end{equation}

Then, we conclude $Z_1\lcx Y'_1$.

\smallskip
\emph{(ii) $Z_2\lcx Y'_2$.} Both means are $16.667$. Using \eqref{eq:Z2survival} and \eqref{eq:Y2survival}, for $0 \leq x < 40$, we have the following

\begin{equation}
D_{2}(x) := \overline{F}_{Z_2}(x) - \overline{F}_{Y'_2}(x) = \frac{1}{2}\left(1 - \frac{x}{40} + \left(1 - \frac{x}{40}\right)^{2} - \left(1 - \frac{x}{100}\right)^{2}\right).
\end{equation}

First, observe that $D_{2}(0) = \frac{1}{2} > 0$. Moreover,

\begin{equation}
D'_{2}(x) = \frac{1}{2}\left(\frac{1}{50}\left(1 - \frac{x}{100}\right) - \frac{1}{40} - \frac{1}{20}\left(1 - \frac{x}{40}\right)\right) < 0, \quad 0 \leq x < 40,
\end{equation}
implying that $D_{2}(x)$ is strictly decreasing on $[0, 40)$. Consequently, $D_{2}(x)$ admits a unique root at $x = 23.415$ within this interval. It follows that

\begin{equation}
D_{2}(x) \geq 0, \quad 0 \leq x \leq 23.415, \quad \text{and} \quad D_{2}(x) < 0, \quad 23.415 < x < 40.
\end{equation}

It therefore remains to show that $D_{2}(x) < 0$ for all $40 \leq x \leq 100$. This establishes that $D_{2}(x)$ changes sign exactly once over its support and the sign sequence is $+,-$, thereby satisfying the one-crossing condition for survival functions. For $40 \leq x \leq 100$, we obtain
\begin{equation}
D_{2}(x) = -\frac{1}{2}\left(1 - \frac{x}{100}\right)^2 \leq 0, \quad 40 \leq x \leq 100.
\end{equation}
Consequently, we conclude $Z_2\lcx Y'_2$.
}

\subsection{Supplementary example illustrating the role of Assumption \ref{LinearRiskSharing-Section3-AssumptionCC}}\label{sec:Supplementary example illustrating the role of AssumptionCC}

\subsubsection{A numerical illustration with failure of the constraint for two participants}
\label{SubSub422}

For this second example, we examine a pool of three participants ($n=3$) with Exponentially-distributed losses, defined as follows:
\begin{itemize}
    \item[(i)] Claim frequency: $\lambda_1=100, \lambda_2=2$, and $\lambda_3=100$,
    \item[(ii)] Claim severity: $\alpha_1=1, \alpha_2=\frac{1}{50}$, and $\alpha_3=1$.
\end{itemize}
The resulting mean claim sizes are
$\textcolor{black}{b_{1}}=\frac{1}{\alpha_1}=1$, $\textcolor{black}{b_{2}}=\frac{1}{\alpha_2}=50$, and
$\textcolor{black}{b_{3}}=\frac{1}{\alpha_3}=1$.

We again assume a safety loading factor of $\eta = \tfrac{2}{5}$. Losses are allocated according to either
the mean-proportional or the alternative  schemes, defined by
$$
\Avec^{\mathrm{MP}}=\left(\begin{array}{ccc}
\frac{1}{3} & \frac{1}{3} & \frac{1}{3} \\ \\
\frac{1}{3} & \frac{1}{3} & \frac{1}{3} \\ \\
\frac{1}{3} & \frac{1}{3} & \frac{1}{3}
\end{array}\right)
\text{ and }
\Avec^{\mathrm{ALT}}=\left(\begin{array}{ccc}
0.5 & 0.4 & 0.1 \\ \\
0.3 & 0.2 & 0.5 \\ \\
0.2 & 0.4 & 0.4
\end{array}\right),
$$
respectively.

Both the mean-proportional and alternative schemes violate condition \eqref{Const_b} for participants 1 and 3 because
\begin{align*}
a^{\mathrm{MP}}_{1, 2} &= \frac{1}{3} \approx \textcolor{black}{0.333} > \frac{b_{1}}{b_{2}} = \frac{1}{50} = 0.02, \\ 
a^{\mathrm{MP}}_{3, 2} &= \frac{1}{3} \approx \textcolor{black}{0.333} > \frac{b_{3}}{b_{2}} = \frac{1}{50} = 0.02, \\
a^{\mathrm{ALT}}_{1, 2} &= 0.4 > \frac{b_{1}}{b_{2}} = \frac{1}{50} = 0.02, \qquad \text{ and } \\ 
a^{\mathrm{ALT}}_{3, 2} &= 0.4 > \frac{b_{3}}{b_{2}} = \frac{1}{50} = 0.02.
\end{align*}
As a result, Figure \ref{NumericalIllustrations-Section4-Figure3} shows that pooling does not benefit participants 1 and 3.

\begin{figure}[H]
  	\centering
	% Plot generated with the R code: Example 4.2.2 - Exponentially-distributed losses (Failure for two participants).R
        % We could generate other plots if needed.
        \begin{subfigure}[b]{0.32\linewidth}
  	\includegraphics[width=4.5cm, height=4.5cm]{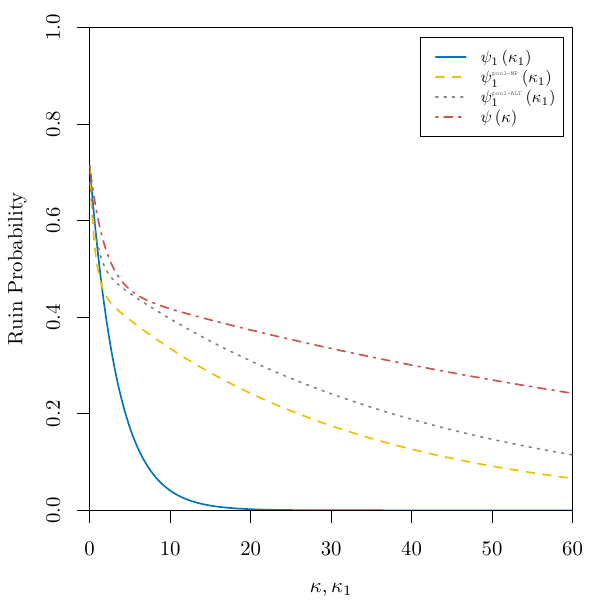}
	\caption{}
  	\label{NumericalIllustrations-Section4-Figure3-a}
	\end{subfigure}
 	 \hspace{0.01cm}
	% Plot generated with the R code: Example 4.2.2 - Exponentially-distributed losses (Failure for two participants).R
        % We could generate other plots if needed.
         \begin{subfigure}[b]{0.32\linewidth}
 	 \includegraphics[width=4.5cm, height=4.5cm]{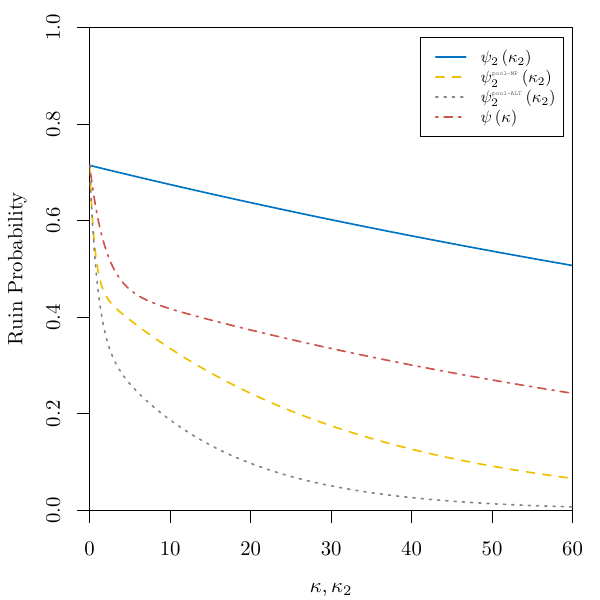}
	 \caption{}
  	\label{NumericalIllustrations-Section4-Figure3-b}
	\end{subfigure}
  	 \hspace{0.01cm}
	% Plot generated with the R code: Example 4.2.2 - Exponentially-distributed losses (Failure for two participants).R
        % We could generate other plots if needed.
          \begin{subfigure}[b]{0.32\linewidth}
 	 \includegraphics[width=4.5cm, height=4.5cm]{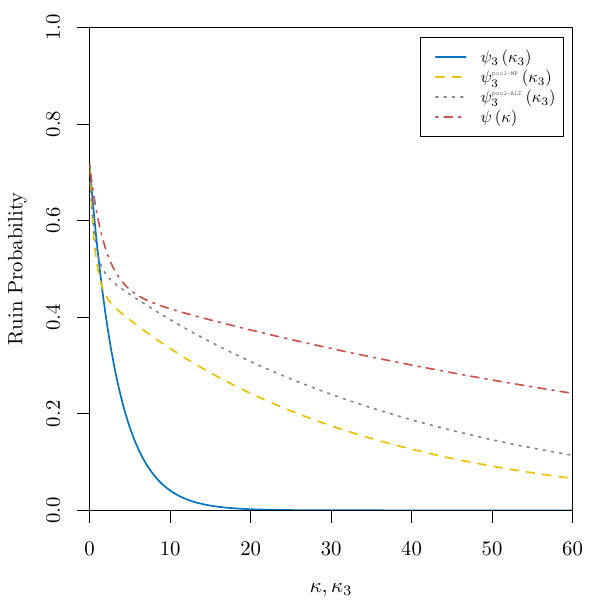}
	 \caption{}
  	\label{NumericalIllustrations-Section4-Figure3-c}
	\end{subfigure}
\caption{\textcolor{black}{Ruin probabilities for the pooled fund, individual and pooled surplus models under the mean-proportional (pool - MP) risk-sharing scheme and the alternative (pool - ALT) proportional risk-sharing scheme in the pool of Appendix \ref{SubSub422}.}}
\label{NumericalIllustrations-Section4-Figure3}
\end{figure}

\end{document}